\documentclass[11pt,a4paper]{article}
\usepackage{amsmath,amssymb,amsthm}
\usepackage{dsfont,mathrsfs,enumerate}
\usepackage{epic,eepic,pmgraph}
\usepackage[linktocpage]{hyperref}

\newtheorem{thm}{Theorem}[section]
\newtheorem{cor}[thm]{Corollary}
\newtheorem{prop}[thm]{Proposition}
\newtheorem{lemma}[thm]{Lemma}
\newtheorem{defithm}[thm]{Definition -- Theorem}

\theoremstyle{definition}
\newtheorem{defi}[thm]{Definition}
\newtheorem{rem}[thm]{Remark}
\newtheorem{exa}[thm]{Example}

\numberwithin{equation}{section}

\def\N{\mathds N}
\def\Z{\mathds Z}
\def\Q{\mathds Q}

\def\C{\mathds C}
\def\epsilon{\varepsilon}
\def\phi{\varphi} 
\def\rho{\varrho}
\def\id{{\it id}}
\def\A{\mathds A}
\def\G{\mathds G}
\def\P{\mathds P}
\def\O{\mathcal O}

\DeclareMathOperator{\Mor}{Mor}
\DeclareMathOperator{\Hom}{Hom}

\DeclareMathOperator{\SHom}{{\mathscr H}\!\!\textit{om}\,}
\DeclareMathOperator{\SEnd}{{\mathscr E}\!\textit{nd}\,}

\DeclareMathOperator{\UHom}{\underline{Hom}}
\DeclareMathOperator{\UEnd}{\underline{End}}
\DeclareMathOperator{\UAut}{\underline{Aut}}

\DeclareMathOperator{\coker}{coker}
\DeclareMathOperator{\St}{St}

\DeclareMathOperator{\GL}{GL}
\DeclareMathOperator{\PGL}{PGL}
\DeclareMathOperator{\Spec}{Spec}
\DeclareMathOperator{\Proj}{Proj}

\DeclareMathOperator{\conv}{conv}
\DeclareMathOperator{\interior}{int}
\DeclareMathOperator{\relint}{relint}

\DeclareMathOperator{\Pic}{Pic}
\DeclareMathOperator{\NS}{NS}

\DeclareMathOperator{\Lim}{Lim}

\begin{document}

\title{\vspace*{-1.1cm}Quivers and moduli spaces of pointed curves of genus zero}
\author{Mark Blume, Lutz Hille}
\date{}

\maketitle

\begin{abstract}
We construct moduli spaces of representations of quivers over arbitrary
schemes and show how moduli spaces of pointed curves of genus zero 
like the Grothendieck-Knudsen moduli spaces $\overline{M}_{0,n}$
and the Losev-Manin moduli spaces $\overline{L}_n$ can be interpreted 
as inverse limits of moduli spaces of representations of certain bipartite quivers.
We also investigate the case of more general Hassett moduli spaces
$\overline{M}_{0,a}$ of weighted pointed stable curves of genus zero. 
\end{abstract}

\smallskip

\tableofcontents

\bigskip

\thispagestyle{empty}
\enlargethispage{4mm}

\noindent
------------

{\small Research supported by DFG CRC 878 Groups, Geometry and Actions.}

\section*{Introduction}
\addcontentsline{toc}{section}{Introduction}

The main topic of this paper is the relation of moduli spaces of pointed curves 
of genus zero, in particular the Grothendieck-Knudsen moduli spaces
$\overline{M}_{0,n}$ and the Losev-Manin moduli spaces $\overline{L}_n$,
but also more general Hassett moduli spaces of weighted pointed
stable curves of genus zero $\overline{M}_{0,a}$,
to moduli spaces of representations of the following quivers 
$Q_n$ and $P_n$ with fixed dimension vector as indicated by the numbers
at the vertices:

\bigskip

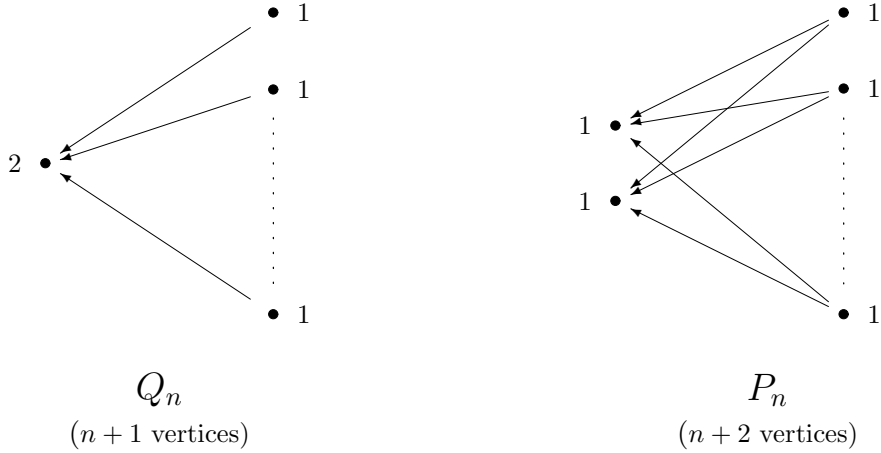
\setlength{\unitlength}{1mm}
\noindent
\begin{picture}(150,70)(0,0)
\filltype{black}

\put(20,40){\circle*{1.2}}
\put(50,60){\circle*{1.2}}
\put(50,49.8){\circle*{1.2}}
\dottedline{2}(50,46)(50,24)
\put(50,20){\circle*{1.2}}
\put(47,58){\vector(-3,-2){25}}
\put(47,48.8){\vector(-3,-1){25}}
\put(47,22){\vector(-3,2){25}}
\put(16,40){\makebox(0,0)[c]{\small$2$}}
\put(54,60){\makebox(0,0)[c]{\small$1$}}
\put(54,49.8){\makebox(0,0)[c]{\small$1$}}
\put(54,20){\makebox(0,0)[c]{\small$1$}}
\put(35,10){\makebox(0,0)[c]{\Large$Q_n$}}
\put(35,4){\makebox(0,0)[c]{(\small$n+1$ vertices)}}

\put(95,35){\circle*{1.2}}
\put(95,45){\circle*{1.2}}
\put(125,60){\circle*{1.2}}
\put(125,49.9){\circle*{1.2}}
\dottedline{2}(125,46)(125,24)
\put(125,20){\circle*{1.2}}
\put(123,59){\vector(-2,-1){26}}
\put(123,58.33){\vector(-6,-5){26}}
\put(123,49.5){\vector(-6,-1){26}}
\put(123,48.8){\vector(-2,-1){26}}
\put(123,21){\vector(-2,1){26}}
\put(123,21.66){\vector(-6,5){26}}
\put(91,35){\makebox(0,0)[c]{\small$1$}}
\put(91,45){\makebox(0,0)[c]{\small$1$}}
\put(129,60){\makebox(0,0)[c]{\small$1$}}
\put(129,49.9){\makebox(0,0)[c]{\small$1$}}
\put(129,20){\makebox(0,0)[c]{\small$1$}}
\put(115,10){\makebox(0,0)[c]{\Large$P_n$}}
\put(115,4){\makebox(0,0)[c]{(\small$n+2$ vertices)}}
\end{picture}

\bigskip\bigskip

Moduli spaces of representations $\mathcal M^\theta(Q)$ of a quiver $Q$ 
with some fixed dimension vector are constructed via geometric invariant theory 
(GIT) and depend on the choice of a weight $\theta$. 
The collection of $\mathcal M^\theta(Q)$ for $\theta$ representing
the finitely many GIT equivalence classes forms an inverse system and
as inverse limit we have a space $\varprojlim_\theta\,\mathcal M^\theta(Q)$
independent of choice of weights.

\medskip

In this paper we show over arbitrary base schemes, see theorems \ref{thm:Ln-limPn} 
and \ref{thm:M0n-limQn} (over $\C$ this is essentially known via different methods,
though usually formulated without quivers, see below):

\medskip

\noindent
{\bf Theorem.} 
\[
\overline{L}_n\;\,\cong\;\;\varprojlim\limits_\theta\,\mathcal M^\theta(P_n)
\qquad\qquad
\overline{M}_{0,n}\;\cong\;\;\varprojlim\limits_\theta\,\mathcal M^\theta(Q_n)
\]

Our methods also apply to more general Hassett moduli spaces $\overline{M}_{0,a}$ 
of weighted pointed stable curves of genus zero \cite{Ha03}. 
In theorem \ref{thm:M0a-limQna} we show that any of these Hassett moduli spaces 
is a limit of GIT quotients over an area in the weight space
(considering normalised weights), which is new already over $\C$:

\medskip

\noindent
{\bf Theorem.} 
\[\overline{M}_{0,a}\;\cong\;\:\varprojlim_{\theta<a}\mathcal M^\theta(Q_n).\]

\medskip

It is natural to work over $\Spec\Z$ as the moduli spaces of pointed curves 
are defined over the integers, and so are moduli spaces of representations 
of quivers. However, we have to extend the results on moduli spaces of 
representations of quivers in \cite{Ki94} to arbitrary base schemes 
(see theorem \ref{thm:modulireprquiver}):

\medskip

\noindent
{\bf Theorem.} {\it For a finite quiver $Q$ with indivisible dimension vector 
and generic weight $\theta$ there is, over an arbitrary base scheme, a fine moduli 
space $\mathcal M^\theta(Q)$ whose $Y$-valued points correspond to $\theta$-stable
representations over $Y$ up to isomorphism over a Zariski cover.}

\medskip

Here the main issue is less the construction of these spaces by application 
of the theory of GIT over arbitrary schemes \cite{Al14}, \cite{Ses77} to this 
problem, but more the proof that these GIT quotients for generic weights are 
in fact fine moduli spaces of quiver representations and the application of 
the necessary techniques, like the result of \cite{An73} in the proof of 
proposition \ref{prop:Rs-Mtheta-torsor}.

\medskip

To show the isomorphisms between moduli spaces of pointed curves of genus $0$
and limits of GIT quotients we use the interpretation of these GIT quotients 
for generic weights as fine moduli spaces and compare the functor of the limit 
to the moduli functor of pointed curves via the following steps:

\noindent
\ -- The functor of the limit associates to a scheme $Y$ the set of collections 
$(\phi_\theta)_\theta$ of morphisms (or equivalently collections $(V^\theta)_\theta$ 
of equivalence classes of $\theta$-stable representations over $Y$) for generic 
$\theta$ (i.e.\ corresponding to full-dimensional GIT equivalence classes) that 
satisfy relations coming from  GIT classes of codimension one, see proposition 
\ref{prop:functorinvlim} and (\ref{eq:functorinvlim-QnPn}) at the beginning of 
subsection \ref{subsec:functor-limPnQn}.

\noindent
\ -- In proposition \ref{prop:functorinvlimQnPn} we express the relations between
the representations $(V^\theta)_\theta$ in the limit functor by certain equations. 
Whereas the original relations are local in the weight space, we have equations 
relating pairs of representations for all generic $\theta$. We can reduce these sets 
of equations by remark \ref{rem:restr-eq-theta-theta'}.

\noindent
\ -- In the proofs of the results for the individual moduli spaces of pointed curves
(theorems \ref{thm:Ln-limPn}, \ref{thm:M0n-limQn} and \ref{thm:M0a-limQna}) we
restrict the collections $(V^\theta)_\theta$ to certain sets of generic weights. 
Considering representations as $\P^1$-bundles with $n$ sections 
(see corollary \ref{cor:reprQn-P1n-GIT} and \ref{cor:reprPn-P1n-GIT}),
these weights $\theta^i$ (resp.\ $\theta^T$) correspond to sets of three sections, 
such that a representation is stable with respect to this weight if and only if 
these three sections are disjoint.

\noindent
\ -- Normalising the representations with respect to the three sections for each 
of these weights, the functor of the inverse limit becomes a functor consisting 
of sections of $\P^1$ that satisfy certain equations.
We compare these functors of the limit with the functors of the moduli spaces 
of pointed curves, described in theorems \ref{thm:Ln-ProdP1-functor}, 
\ref{thm:M0n-ProdP1-functor} and \ref{thm:functorM0a}, that arise from a natural 
closed embedding of these moduli spaces into a products of $\P^1$'s related to 
root systems of type $A$.

\medskip

There is the following known relation between moduli spaces of pointed curves 
of genus $0$ and GIT quotients of $(\P^1)^n$.
Working over the complex numbers, by \cite{KSZ91} under some assumptions
limits of GIT quotients of toric varieties $(X,T)$ by subtori of $T$ 
are the same as Chow quotients. 
Toric Chow quotients can be described by fibre polytopes, and in the case of 
quotients of projective spaces by secondary polytopes \cite{KSZ91}, \cite{GKZ}.
The Losev-Manin moduli space $\overline{L}_n$ is the projective toric variety
corresponding to the $(n-1)$-dimensional permutohedron \cite{LM00}, which arises 
as secondary polytope of the product of simplices $\Delta^1\times\Delta^{n-1}$ 
and as fibre polytope for the operation of $\G_m$ 
on $(\P^1)^n$. 
By the Gelfand-MacPherson correspondence (see \cite{Ka93a}) the GIT quotients 
$(\P^1)^n$ by $\PGL(2)$ and the GIT quotients of the Grassmannian $G(2,n)$ 
by the maximal torus $(\G_m)^n\subset\GL(n)$, and similarly the Chow quotients 
$(\P^1)^n/\!\!/\PGL(2)$ and $G(2,n)/\!\!/(\G_m)^n$ coincide over the complex numbers. 
By \cite{Ka93a} the Chow quotients $G(2,n)/\!\!/(\G_m)^n$ coincide with the 
Hilbert quotients
$G(2,n)/\!\!/\!\!/(\G_m)^n$ and by \cite{Ka93b} $\overline{M}_{0,n}$ 
is the closure of the space of Veronese curves in the Hilbert scheme or Chow variety.
This is used to show that the Chow quotient $G(2,n)/\!\!/(\G_m)^n$ is isomorphic
to $\overline{M}_{0,n}$, see \cite[(4.1.8)]{Ka93a}.
Together the works \cite{Ka93a}, \cite{Ka93b}, \cite{KSZ91} show that over
the complex numbers $\overline{M}_{0,n}$ is the limit of GIT quotients 
$(\P^1)^n/_{\!\theta}\PGL(2)$ (see also \cite{Th99}).

With the straightforward link between the moduli spaces of representations\linebreak
$\mathcal M^\theta(Q_n)$, $\mathcal M^\theta(P_n)$ and the GIT quotients 
$(\P^1)^n/_{\!\theta}\PGL(2)$, $(\P^1)^n/_{\!\theta}\,\G_m$ as described\linebreak
in subsection \ref{subsec:quivermoduli-GITP1}
(related to the Gelfand-MacPherson correspondence) we recover the results that 
over the complex numbers
$\overline{L}_n\cong\varprojlim_\theta\,(\P^1)^n/_{\!\theta}\,\G_m$
(\cite{KSZ91}, \cite{LM00}) and
$\overline{M}_{0,n}\cong\varprojlim_\theta\,(\P^1)^n/_{\!\theta}\PGL(2)$
(\cite{Ka93a}).
Our different, more direct methods are valid over arbitrary base schemes and
do not involve Chow varieties.

\medskip

It is known by \cite{GJM13} that over the complex numbers for fixed $n$ all 
Hassett moduli spaces of $n$-pointed weighted pointed stable curves of genus zero 
can be realised as GIT quotients of a certain space involving the component of 
the Chow variety of degree $d$ curves in $\P^d$ containing rational normal curves, 
which, as a parameter space of embedded $n$-pointed curves of genus zero, is already 
close to the moduli problem.

In the present paper we describe these Hassett moduli spaces as a limit of 
GIT quotients that come from quivers, hence in particular GIT quotients of 
affine space. Here the individual GIT quotients are quite simple, furthermore, 
in our construction we restrict to a collection of GIT quotients for generic 
weights all of which are products of $\P^1$'s.

\medskip

Observing that the limits of moduli spaces of representations  
for these very simple quivers give rise to relevant moduli spaces of pointed curves,
it may be interesting to describe and identify the varieties that 
arise this way from other quivers.

\bigskip\medskip

{\it Outline of the paper.}
In section \ref {sec:moduliquiver} we construct fine moduli spaces of quiver 
representations over arbitrary schemes. To construct the spaces 
$\mathcal M^{\theta-s}(Q,d)\subseteq\mathcal M^\theta(Q,d)$ of $\theta$-stable 
and $\theta$-semistable representations we apply the theory of GIT (see \cite{MFK}) 
over arbitrary base schemes, using the formulation in terms of stacks in \cite{Al14}.
We show in theorem \ref{thm:modulireprquiver} that for a finite quiver $Q$, 
an indivisible dimension vector $d$ and a weight $\theta$ the space 
$\mathcal M^{\theta-s}(Q,d)$ is a fine moduli space parametrising $\theta$-stable 
representations up to isomorphism over a Zariski cover. The main point here is 
to show that the natural group operation on the stable locus of the representation 
space is free and the morphism to the quotient a torsor, see proposition 
\ref{prop:Rs-Mtheta-torsor}. This extends the results of \cite{Ki94}.

The first step to describe the functors of points of inverse limits of 
moduli spaces of quiver representations is taken in proposition \ref{prop:functorinvlim}
by describing the functor of the limit by the functors of the moduli spaces
$\mathcal M^\theta(Q,d)$ for equivalence classes of generic weights $\theta$ 
subject to some relations. 

\pagebreak

Section \ref{sec:weightspace-moduli-PnQn} contains on the one hand some preparations 
and miscellaneous material concerning the quivers $P_n$ and $Q_n$: 
the structure of the weight spaces for the quivers $P_n$ and $Q_n$ 
(compare to \cite[(1.2),(1.3)]{Ka93a}, \cite[7.3.C]{GKZ});
the relation between $Q_{n+2}$ and $P_n$, their weights and stable representations
(this corresponds to the morphism $\overline{M}_{0,n+2}\to\overline{L}_n$ after 
a choice of two of the $n+2$ marked points);
the relation of moduli spaces of representations for $P_n$ and $Q_n$ and 
GIT quotients of $(\P^1)^n$. The last two points, discussed in subsections
\ref{subsec:Qn+2-Pn} and \ref{subsec:quivermoduli-GITP1}, both derive from 
different presentations of certain stacks.

On the other hand subsection \ref{subsec:functor-limPnQn} is one of the central parts 
of this paper where in proposition \ref{prop:functorinvlimQnPn} we describe the functor 
of the limit of moduli spaces of representations for the quivers $P_n$ and $Q_n$ 
in terms of families of representations $(V^\theta)_\theta$, where the weights $\theta$ 
represent the equivalence classes of generic weights, satisfying certain equations.

\medskip

In section \ref{sec:Ln-M0n} we consider the Grothendieck-Knudsen moduli spaces
and Losev-Manin moduli spaces.
It is known that the moduli spaces $\overline{L}_n$ and $\overline{M}_{0,n}$ 
can be embedded into a product of $\,\P^1$'s via the natural morphisms
$\overline{L}_n\to\overline{L}_2\cong\P^1$ and 
$\overline{M}_{0,n}\to\overline{M}_{0,4}\cong\P^1$ (see \cite{BB11}, \cite{GHP88}). 
These closed embeddings have interpretations in terms of root systems of type $A$.
In the case of the Losev-Manin moduli space $\overline{L}_n\subseteq(\P^1)^{\binom{n}{2}}$ 
each copy of $\,\P^1$ corresponds to a root subsystem in $A_{n-1}$ isomorphic 
to $A_1$, whereas the equations come from the root subsystems isomorphic 
to $A_2$, see \cite{BB11}.
In the case of the Grothendieck-Knudsen moduli space 
$\overline{M}_{0,n}\subseteq(\P^1)^{\binom{n}{4}}$ each copy
of $\P^1$ corresponds to a root subsystem in $A_{n-1}$ isomorphic 
to $A_3$, whereas the equations come from the root subsystems isomorphic 
to $A_4$. This way the moduli space $\overline{M}_{0,n}$ can also be interpreted
as the cross-ratio variety of root systems of type $A_3$ in $A_{n-1}$,
cf.\ \cite{Sek94}, \cite{Sek96}.
We summarise the results on these embeddings, the interpretation in terms 
of root systems and the related description of the functor of points of 
$\overline{L}_n$ and $\overline{M}_{0,n}$ in theorems \ref{thm:Ln-ProdP1-functor} 
and \ref{thm:M0n-ProdP1-functor}.

Proving that there are isomorphisms 
$\overline{L}_n\cong\varprojlim_\theta\,\mathcal M^\theta(P_n)$
(theorem \ref{thm:Ln-limPn}) and
$\overline{M}_{0,n}\cong\varprojlim_\theta\,\mathcal M^\theta(Q_n)$
(theorem \ref{thm:M0n-limQn}), we compare the moduli functors of pointed curves
with the corresponding functors of the limit.
We use the description of proposition \ref{prop:functorinvlimQnPn} and simplify it, 
where the main step consists of restricting the family of representations
$(V^\theta)_\theta$ to certain sets of weights of the form 
$\theta^i$ (resp.\ $\theta^T$).

\medskip

In section \ref{sec:M0a} we apply the same methods to the Hassett moduli spaces 
to show that $\overline{M}_{0,a}\cong\varprojlim_{\theta<a}\mathcal M^\theta(Q_n)$
(theorem \ref{thm:M0a-limQna}). Here some additions are necessary to take into 
account the notion of $a$-stability of $n$-pointed curves.
On the quiver side this corresponds to a restriction to weights in the convex
polytope $P(a)=\{\theta\in\Delta(2,n)\:|\:\theta<a\}$.
We also need a description of the moduli functor $\overline{M}_{0,a}$
analogous to that of $\overline{L}_n$ and $\overline{M}_{0,n}$ arising from 
an embedding into a product of $\P^1$'s, see theorem \ref{thm:functorM0a}.

\bigskip

{\it Notations and conventions.}
We work over a base scheme $S$, usually we think of $S$ as $\Spec\Z$.
Products are over the base scheme $S$ unless otherwise specified. 
We use the theory of schemes as developed in \cite{EGA}, \cite{EGA1}. 
We sometimes write $y\in Y$ for a geometric point $y\colon\Spec k\to Y$.
We use the standard theory of algebraic stacks, see \cite{SP}. 
We work with algebraic stacks in the fppf topology.
Concerning stability of quiver representations we use the sign convention 
opposite to \cite{Ki94}. For $d,d'\in\Q^n$ we write $d'\leq d$ if and only if 
$d'_i\leq d_i$ for all $i$, and we write $d'<d$ if $d'\leq d$ and $d'\neq d$.

\section{Moduli spaces of quiver representations}
\label{sec:moduliquiver}

\subsection{Moduli of quiver representations over arbitrary schemes}

Let $Q$ be a finite connected quiver with set of vertices $Q_0$ and
set of arrows $Q_1$. 
For an arrow $\alpha\in Q_1$ we denote $s(\alpha)$ and $t(\alpha)$ the source
and target, i.e.\ $s(\alpha)\stackrel{\alpha\:}{\longrightarrow}t(\alpha)$.

\begin{defi}
A {\it representation} of $Q$ of dimension vector $(d_q)_{q\in Q_0}$ over 
a scheme $Y$ is a tuple $V=((V_q)_{q\in Q_0},(\phi_\alpha)_{\alpha\in Q_1})$ 
where $V_q$ is a locally free $\O_Y$-module of rank $d_q$ and
$\phi_\alpha\colon V_{s(\alpha)}\to V_{t(\alpha)}$ a homomorphism of 
$\O_Y$-modules. A representation over $Y$ is called free if the underlying
$\O_Y$-module is free. The notions of homomorphisms of representations,
subrepresentations etc.\ are defined as usual.
\end{defi}

\begin{defi}
The {\it representation space} of $Q$ for the dimension vector $d=(d_q)_{q\in Q_0}$
over the base scheme $S$ is the affine space over $S$
\[
R(Q,d)=\prod_{\alpha\in Q_1}\UHom_S(\O_S^{\oplus d_{s(\alpha)}},
\O_S^{\oplus d_{t(\alpha)}})
\]
where $\UHom_S(\O_S^{\oplus d_{s(\alpha)}},\O_S^{\oplus d_{t(\alpha)}})$ 
is the Hom-scheme that represents the contravariant functor on $S$-schemes
$Y\mapsto\Hom_Y(\O_Y^{\oplus d_{s(\alpha)}},\O_Y^{\oplus d_{t(\alpha)}})$.
Equivalently, the scheme $R(Q,d)$ can be defined as the scheme that represents 
the contravariant functor on $S$-schemes
\[
Y\;\mapsto\;\big\{\;\,\textit{free representations with fixed basis of $(Q,d)$ 
over $Y$}\,\big\}.
\]
On the representation space $R(Q,d)$ the group scheme over $S$
\[
G(d)=\prod_{q\in Q_0}\GL(d_q)
\]
operates by $(g_q)_q(\nu_\alpha)_\alpha=(g_{t(\alpha)}\nu_\alpha 
g_{s(\alpha)}^{-1})_\alpha$ where $(g_q)_q\in G(d)(Y)$,
$(\nu_\alpha)_\alpha\in R(Q,d)(Y)$ for $S$-schemes $Y$.
The diagonal subgroup $\,\G_m\cong\Delta\subseteq G(d)$ operates trivially, 
giving rise to a $\,\overline{G(d)}=G(d)/\Delta$-operation on $R(Q,d)$.
\end{defi}

By \cite{Ki94} the following definition of (semi)stability is equivalent to 
(semi)stability in the sense of GIT of \cite{MFK} (with $\prod_q\det(g_q)^{-\theta_q}$
the character corresponding to $\theta$).

\begin{defi}
For a representation $V$ of $Q$ of dimension vector $d=(d_q)_{q\in Q_0}$
over an algebraically closed field and a weight 
$\theta=(\theta_q)_{q\in Q_0}\in\Q^{Q_0}$, using the notations
$\dim(V)=(\dim(V_q))_{q\in Q_0}$ and $\theta(d)=\sum_{q\in Q_0}\theta_qd_q$, 
we define:\\
{\rm(a)} The representation $V$ is called {\it $\theta$-semistable} if
$\theta(d)=0$ and every subrepresentation $V'\subseteq V$ satisfies 
$\theta(\dim(V'))\leq 0$.\\
{\rm(b)} A $\theta$-semistable representation $V$ is called {\it $\theta$-stable} if 
the only subrepresentations $V'\subseteq V$ with $\theta(\dim(V'))=0$ are $V$ and $0$.
\end{defi}

\begin{defi}
A representation $V$ of $Q$ of dimension vector $(d_q)_{q\in Q_0}$
over a scheme $Y$ is called {\it $\theta$-(semi)stable} if for all geometric points 
$y\colon\Spec k\to Y$ the representation $V_y=V\otimes_{\O_Y} k$ is 
$\theta$-(semi)stable.
\end{defi}

The functors on $S$-schemes 
\[
Y\;\mapsto\;\big\{\textit{free $\theta$-(semi)stable representations with fixed 
basis of $(Q,d)$ over $Y$}\big\}
\]
are represented by open $G(d)$-invariant subschemes $R^\theta(Q,d)\!\subseteq\!R(Q,d)$
(parametrising $\theta$-semistable representations) and 
$R^\textit{$\theta$-s}(Q,d)\!\subseteq\!R^\theta(Q,d)$ (parametrising $\theta$-stable
representations).

\medskip

In order to construct moduli spaces of representations, following \cite{Ki94} we 
consider GIT quotients by the (geometrically) reductive group scheme $G(d)$ or 
$\overline{G(d)}$ (see \cite[Section 9]{Al14} for some results on 
(geometrically) reductive group schemes).
GIT over arbitrary locally noetherian schemes has been worked out first in \cite{Ses77}, 
here we use the formulation of \cite{Al14} in terms of adequate moduli spaces 
of Artin stacks.

\medskip

Let $\mathcal M^\theta(Q,d)=R(Q,d)/_{\!\theta}\,G(d)$ 
the GIT quotient with respect to $\mathscr L=\O_{R(Q,d)}$ 
with the linearisation determined by the character $\theta$
(if $\theta$ is not integral replace $\theta$ by $l\theta$ for some $l\in\Z_{>0}$
such that $l\theta$ is integral), that is
\[\textstyle
\mathcal M^\theta(Q,d)
=\Proj_S\bigoplus_{k\geq 0}(p_*(\mathscr L^{\otimes k}))^{G(d)}
\]
where $p\colon R(Q,d)\to S$ is the structure morphism. We have the quotient morphism 
\[
R^\theta(Q,d)\to\mathcal X^\theta(Q,d)\stackrel{q\,}{\to}\mathcal M^\theta(Q,d).
\] 
where $\mathcal X^\theta(Q,d)=[R^\theta(Q,d)/\overline{G(d)}]$ is the 
quotient stack. The next theorem follows from the result in the affine case
\cite[Thm.\ 9.1.4]{Al14} as this property is local \cite[Prop.\ 5.2.9]{Al14}.

\begin{thm}\label{thm:adequatemodulispace}{\rm\cite{Al14}.}
$q\colon\mathcal X^\theta(Q,d)\to\mathcal M^\theta(Q,d)$
is an adequate moduli space.
\end{thm}

In particular, we have the following properties (see \cite{Al14}, cf.\ also \cite{Ses77}):
\begin{enumerate}[\rm(1)]
\item The homomorphism $\O_{\mathcal M^\theta(Q,d)}\to q_*\O_{\mathcal X^\theta(Q,d)}$
is an isomorphism.
\item $q$ is surjective, universally closed and universally submersive.
\item For an algebraically closed field $k$ the morphism $q$ induces a bijection
\[
[\mathcal X^\theta(Q,d)(k)]\,/\!\sim\;\;\to\;\;\mathcal M^\theta(Q,d)(k)
\]
where $\sim$ is the equivalence relation on the set of isomorphism classes
$[\mathcal X^\theta(Q,d)(k)]$ defined by $x_1\sim x_2$ if 
$\overline{\{x_1\}}\cap\overline{\{x_2\}}\neq\emptyset$ 
in $\mathcal X^\theta(Q,d)_s$ for $x_1,x_2\in\mathcal X^\theta(Q,d)(k)$ 
over $s\colon\Spec k\to S$.
\item Assume that $S$ is locally noetherian. Then $\mathcal M^\theta(Q,d)$ is of 
finite type over $S$ and $q_*\mathscr F$ is coherent for every coherent 
$\O_{\mathcal X^\theta(Q,d)}$-module $\mathscr F$.
\item $q$ is universal for morphisms to schemes, i.e.\ for an $S$-scheme $Z$ 
the map
\[
\Mor_S(\mathcal M^\theta(Q,d),Z)\to\Mor_S(\mathcal X^\theta(Q,d),Z),
\quad f\mapsto f\circ q
\]
is a bijection.
\item Let $Y\to\mathcal M^\theta(Q,d)$ be a flat morphism of $S$-schemes.
Then $\mathcal X^\theta(Q,d)\times_{\mathcal M^\theta(Q,d)}Y\to Y$
is an adequate moduli space.
\end{enumerate}

Statement (5) means that $R^\theta(Q,d)\to\mathcal M^\theta(Q,d)$ is a 
categorical quotient. Restricting to the open subscheme of stable points 
$R^\textit{$\theta$-s}(Q,d)\subseteq R^\theta(Q,d)$ there is the following 
result corresponding to the notion of geometric quotient.

\begin{cor} 
Let $\mathcal X^\textit{$\theta$-s}(Q,d)\!=\![R^\textit{$\theta$-s}(Q,d)/\overline{G(d)}]$
and $\mathcal M^\textit{$\theta$-s}(Q,d)$ the image of
$\mathcal X^\textit{$\theta$-s}(Q,d)$ in $\mathcal M^\theta(Q,d)$. 
Then $\mathcal X^\textit{$\theta$-s}(Q,d)\to\mathcal M^\textit{$\theta$-s}(Q,d)$
is an adequate moduli space. For algebraically closed fields $k$ the map 
$[\mathcal X^\textit{$\theta$-s}(Q,d)(k)]\to\mathcal M^\textit{$\theta$-s}(Q,d)(k)$ 
is a bijection, where $[\mathcal X^\textit{$\theta$-s}(Q,d)(k)]$ denotes the set of 
isomorphism classes of objects in $\mathcal X^\textit{$\theta$-s}(Q,d)(k)$. 
\end{cor}

In addition to these standard GIT results, in our situation we have:

\begin{prop}\label{prop:Rs-Mtheta-torsor}
The operation of $\overline{G(d)}$ on $R^\textit{$\theta$-s}(Q,d)$ is free
and the morphism\linebreak
$R^\textit{$\theta$-s}(Q,d)\to\mathcal M^\textit{$\theta$-s}(Q,d)$ 
is a $\overline{G(d)}$-torsor in the fppf topology.
\end{prop}
\begin{proof}
We first show that the operation of $\overline{G(d)}$ has trivial stabilisers.
For an $S$-scheme $Y$ the stabiliser $\St(v)\subseteq G(d)_Y$ of an element 
$v\in R^\textit{$\theta$-s}(Q,d)(Y)$ contains the diagonal group scheme 
$\Delta_Y\subseteq G(d)_Y$ and coincides with the automorphism group scheme 
$\UAut_Y(V)$ of the corresponding free representation $V=((V_q)_q,(\nu_\alpha)_\alpha)$ 
of $(Q,d)$ over $Y$.
The automorphism group scheme $\UAut_Y(V)$ is open dense in the endomorphism
scheme $\UEnd_Y(V)$. As the $\O_Y$-module $\SEnd_Y(V)$ is the kernel of 
the homomorphism of locally free sheaves
$\bigoplus_q\SEnd_Y(V_q)\to\bigoplus_\alpha\SHom_Y(V_{s(\alpha)},V_{t(\alpha)})$, 
$(\phi_q)_q\mapsto(\nu_\alpha\phi_{s(\alpha)}-\phi_{t(\alpha)}\nu_\alpha)_\alpha$,
we have $\UEnd_Y(V)=\A_Y(\mathscr C)$ with 
$\mathscr C=\coker\big(\bigoplus_\alpha\SHom_Y(V_{s(\alpha)},V_{t(\alpha)})^\vee
\!\to\bigoplus_q\SEnd_Y(V_q)^\vee\big)$.
Since $v$ is stable, all geometric fibres of $\St(v)=\UAut_Y(V)$ are of dimension $1$
(containing the diagonal group $\Delta$), therefore all geometric fibres 
of $\UEnd_Y(V)$ are $1$-dimensional affine spaces, so $\UEnd_Y(V)$ is an 
affine line bundle over $Y$.
It follows that $\St(v)=\UAut_Y(V)=\Delta_Y$.

We show that the operation is proper, that is 
$\Psi\colon\overline{G(d)}\times R^\textit{$\theta$-s}(Q,d)\to 
R^\textit{$\theta$-s}(Q,d)\times R^\textit{$\theta$-s}(Q,d)$, $(g,v)\mapsto(gv,v)$  
is a proper morphism, using the valuative criterion \cite[II, (7.3.8)]{EGA}.
We can assume we have the base scheme $S=\Spec\Z$. Let $A$ be a discrete valuation
ring with maximal ideal $\mathfrak m$, field of fractions $K$ and $k=A/\mathfrak m$,
and let $t$ be a generator of $\mathfrak m$.
Given a morphism $\Spec A\to R^\textit{$\theta$-s}(Q,d)\times R^\textit{$\theta$-s}(Q,d)$ 
we have to show that the map
$\Mor_{R^\textit{$\theta$-s}(Q,d)\times R^\textit{$\theta$-s}(Q,d)}
(\Spec A,\overline{G(d)}\times R^\textit{$\theta$-s}(Q,d))
\to\Mor_{R^\textit{$\theta$-s}(Q,d)\times R^\textit{$\theta$-s}(Q,d)}
(\Spec K,\overline{G(d)}\times R^\textit{$\theta$-s}(Q,d))$
coming from the canonical homomorphism $A\to K$
is surjective (it is injective because the morphism is separated).
The morphism $\Spec A\to R^\textit{$\theta$-s}(Q,d)\times R^\textit{$\theta$-s}(Q,d)$ 
corresponds to a pair of representations $V'=((V'_q)_q,(\nu'_\alpha)_\alpha)$, 
$V=((V_q)_q,(\nu_\alpha)_\alpha)$ of $(Q,d)$ over $\Spec A$ and 
a morphism $\Spec K\to\overline{G(d)}\times R^\textit{$\theta$-s}(Q,d)$ 
making the diagram arising from the above maps (the usual diagram when applying 
the valuative criterion) commutative corresponds to a pair $(g,V'')$ such that 
$V''=V_K$, $gV_K=V'_K$.
The element $g\in\overline{G(d)}(K)$ corresponds to a family of matrices 
with coefficients in $K$ up to a common factor in $K^*$,
so we can represent it by a family of matrices $(M_q)_q$ with coefficients 
in $A$ and choose bases of $V_q$ and $V'_q$ such that all $M_q$ are diagonal
with diagonal entries of the form $t^e$.
Let $l$ be the minimal integer such that $t^l$ occurs as a diagonal entry of some $M_q$.
Considering the matrices for the homomorphisms $\nu_\alpha,\nu'_\alpha$
with respect to the chosen bases, the $(i,j)$-entries for the two matrices for 
a given $\alpha$ differ by a factor $t^e\neq 1$ if the diagonal entries for the 
corresponding pairs of basis elements differ by such a factor. 
In this case modulo $\mathfrak m$ one of the two $(i,j)$-entries vanishes.
Let $W'_q\subseteq V'_q$ be the submodules generated by those basis elements 
which correspond to the diagonal entries of the form $t^l$, and similary $(U_q)_q$ 
the submodules of $(V_q)_q$ generated by the basis elements corresponding
to the remaining diagonal entries.
Then central fibres of the modules $W'_q\subseteq V'_q$ form a 
subrepresentation $W'_k\subseteq V'_k$ and the central fibres of the modules 
$U_q\subseteq V_q$ form a subrepresentation $U_k\subset V_k$. 
If these were proper nonzero subrepresentations, then, because $V_k$, $V'_k$ are 
$\theta$-stable, both subrepresentations would satisfy $\theta(W'_k)<0$, 
$\theta(U_k)<0$, which is impossible. 
Therefore $W'_k=V'_k$, $U_k=0$. It follows that the element $g\in\overline{G(d)}(K)$
can be represented by matrices with coefficients in $A\setminus\mathfrak m$.
These matrices are invertible as matrices over $A$ and form the required 
element in $\overline{G(d)}(A)$.

Being proper and quasifinite, the morphism $\Psi\colon\overline{G(d)}\times 
R^\textit{$\theta$-s}(Q,d)\to R^\textit{$\theta$-s}(Q,d)\times R^\textit{$\theta$-s}(Q,d)$ 
is finite by \cite[III (1), (4.4.2)]{EGA} or \cite[IV (3), (8.11.1)]{EGA}.
The morphism $\Psi$ is a monomorphism because the operation has trivial stabilisers.
As finite epimorphisms of rings are surjective (see \cite[Ch.\ IV, Prop.\ 1.7]{La69}), 
it follows that $\Psi$ is a closed embedding, i.e.\ the operation
$\overline{G(d)}\times R^\textit{$\theta$-s}(Q,d)\to R^\textit{$\theta$-s}(Q,d)$ is free.

The sheaf-quotient (fppf) $(R^\textit{$\theta$-s}(Q,d)/\overline{G(d)})_{\textit{fppf}}$,
i.e.\ the quotient in the category of fppf-sheaves, is the sheaf associated 
to the presheaf $Y\mapsto R^\textit{$\theta$-s}(Q,d)(Y)/\overline{G(d)}(Y)$ 
on the fppf-site of $S$-schemes (see e.g.\ \cite[3.4.1]{Br14}). As the operation 
$\overline{G(d)}\times R^\textit{$\theta$-s}(Q,d)\to R^\textit{$\theta$-s}(Q,d)$ 
is free and a geometric quotient
$\pi\colon R^\textit{$\theta$-s}(Q,d)\to R^\textit{$\theta$-s}(Q,d)/\overline{G(d)}$ 
exists, by \cite[Thm.\ 6]{An73} the canonical morphism
$(R^\textit{$\theta$-s}(Q,d)/\overline{G(d)})_{\textit{fppf}}\to 
R^\textit{$\theta$-s}(Q,d)/\overline{G(d)}$ 
is an isomorphism and the morphism $\pi$ is fppf. 
Since the sheaf-quotient 
$(R^\textit{$\theta$-s}(Q,d)/\overline{G(d)})_{\textit{fppf}}$
has the property that $\overline{G(d)}\times R^\textit{$\theta$-s}(Q,d)\to
R^\textit{$\theta$-s}(Q,d)\times_{(R^\textit{$\theta$-s}(Q,d)/
\overline{G(d)})_{\textit{fppf}}}R^\textit{$\theta$-s}(Q,d)$
is an isomorphism (see e.g.\ \cite[3.4.5]{Br14}), it follows that the morphism 
of schemes $R^\textit{$\theta$-s}(Q,d)\to R^\textit{$\theta$-s}(Q,d)/\overline{G(d)}$ 
is a $\overline{G(d)}$-torsor in the fppf topology.
\end{proof}

The tautological $G(d)$-equivariant representation $U_{R^\textit{$\theta$-s}(Q,d)}$
on $R^\textit{$\theta$-s}(Q,d)$ can be made into a $\overline{G(d)}$-equivariant 
representation if and only if the dimension vector $d$ is indivisible, 
i.e.\ the set of integers $d_q$ is coprime, see \cite[Proof of Prop.\ 5.3]{Ki94}.
In this case by descent along the $\overline{G(d)}$-torsor
$R^\textit{$\theta$-s}(Q,d)\to\mathcal M^\textit{$\theta$-s}(Q,d)$ one obtains
a representation $U_{\mathcal M^\textit{$\theta$-s}(Q,d)}$ 
on $\mathcal M^\textit{$\theta$-s}(Q,d)$.

\begin{thm}\label{thm:modulireprquiver}
For an indivisible dimension vector $d=(d_q)_{q\in Q_0}$ the scheme 
$\mathcal M^\textit{$\theta$-s}(Q,d)$ with universal representation 
$U_{\mathcal M^\textit{$\theta$-s}(Q,d)}$ is a fine moduli space of $\theta$-stable 
representations of $(Q,d)$ in the sense that it represents the contravariant 
functor on $S$-schemes
\[
Y\;\mapsto\;
\big\{\textit{$\,\theta$-stable representations of $(Q,d)$ over $Y\,$}\big\}\,/\!\sim
\]
of $\theta$-stable representations up to isomorphism over a Zariski cover.
\end{thm}
\begin{proof}
We denote the above functor $F^\textit{$\theta$-s}(Q,d)$ and 
$q\colon R^\textit{$\theta$-s}(Q,d)\to\mathcal M^\textit{$\theta$-s}(Q,d)$
the quotient morphism.
We have the morphisms $\mathcal M^\textit{$\theta$-s}(Q,d)\to 
F^\textit{$\theta$-s}(Q,d)$, $(f\colon Y\to\mathcal M^\textit{$\theta$-s}(Q,d))
\linebreak\mapsto f^*U_{\mathcal M^\textit{$\theta$-s}(Q,d)}$ 
and $F^\textit{$\theta$-s}(Q,d)\to\mathcal M^\textit{$\theta$-s}(Q,d)$,
$V\mapsto f_V$, where $f_V$ is locally defined as $f_V=q\circ\tilde{f}_V$ 
for free $V$ and $\tilde{f}_V\colon Y\to R^\textit{$\theta$-s}(Q,d)$ is determined
by $V$ and a choice of basis of $V$, using that $R^\textit{$\theta$-s}(Q,d)$ with
$U_{R^\textit{$\theta$-s}(Q,d)}$ represents the functor of $\theta$-stable 
free representations with fixed bases.

To show that the composition 
$F^\textit{$\theta$-s}(Q,d)\to\mathcal M^\textit{$\theta$-s}(Q,d)
\to F^\textit{$\theta$-s}(Q,d)$ is the identity, it is enough to show that 
$f_V^*U_{\mathcal M^\textit{$\theta$-s}(Q,d)}\cong V$ for free representations $V$.
It is $f_V^*U_{\mathcal M^\textit{$\theta$-s}(Q,d)}
=(q\circ \tilde{f}_V)^*U_{\mathcal M^\textit{$\theta$-s}(Q,d)}
\cong\tilde{f}_V^*U_{R^\textit{$\theta$-s}(Q,d)}\cong V$.

To show that the composition 
$\mathcal M^\textit{$\theta$-s}(Q,d)\to F^\textit{$\theta$-s}(Q,d)\to
\mathcal M^\textit{$\theta$-s}(Q,d)$ is the identity,
it suffices to show that $f=f_{f^*U_{\mathcal M^\textit{$\theta$-s}(Q,d)}}$
for $f\colon Y\to\mathcal M^\textit{$\theta$-s}(Q,d)$ such that
$f^*U_{\mathcal M^\textit{$\theta$-s}(Q,d)}$ is free.
Let $p_Y\colon Y'\to Y$ be an fppf morphism such that the $\overline{G(d)}$-torsor 
$R_Y=Y{}_f\!\!\times_q\,R^\textit{$\theta$-s}(Q,d)\to Y$ becomes a trivial torsor 
$R_{Y'}\to Y'$ after base extension by $p_Y$. A section $Y'\to R_{Y'}$ gives a 
morphism $\tilde{f}'\colon Y'\to R^\textit{$\theta$-s}(Q,d)$ which satisfies 
$q\circ\tilde{f}'=f\circ p_Y$.
We have $(\tilde{f}')^*U_{R^\textit{$\theta$-s}(Q,d)}\cong
p_Y^*f^*U_{\mathcal M^\textit{$\theta$-s}(Q,d)}\cong
p_Y^*(\tilde{f}_{f^*U_{\mathcal M^\textit{$\theta$-s}(Q,d)}})^*
U_{R^\textit{$\theta$-s}(Q,d)}$.
Therefore $\tilde{f}'$ coincides with 
$\tilde{f}_{f^*U_{\mathcal M^\textit{$\theta$-s}(Q,d)}}\!\circ\,p_Y$
up to an element of $\overline{G(d)}$ that arises from comparison of
the two bases coming from $U_{R^\textit{$\theta$-s}(Q,d)}$ of these isomorphic 
representations. Composing with $q$, it follows 
$f\circ p_Y=f_{f^*U_{\mathcal M^\textit{$\theta$-s}(Q,d)}}\circ p_Y$,
and thus $f=f_{f^*U_{\mathcal M^\textit{$\theta$-s}(Q,d)}}$ since 
$p_Y$ is an epimorphism.
\end{proof}

The equivalence relation on isomorphism classes of geometric points of 
$\mathcal X^\theta(Q,d)$ in (3) after theorem \ref{thm:adequatemodulispace},
or equivalently expressed as GIT-equivalence on $R^\theta(Q,d)(k)$
defined by $x_1\sim x_2$ if 
$\overline{G(d)_sx_1}\cap\overline{G(d)_sx_2}\neq\emptyset$
in $R^\theta(Q,d)_s$ for $x_1,x_2\in R^\theta(Q,d)(k)$ 
over $s\colon\Spec k\to S$, 
is by \cite{Ki94} equivalent to S-equivalence of semistable representations  
as in the following definition.

\begin{defi}
Two $\theta$-semistable representations $V,V'$ of $Q$ over an algebraically 
closed field are  called {\it S-equivalent} if there are filtrations in the 
category of $\theta$-semistable representations
$0=V_0\subsetneq\ldots\subsetneq V_l=V$ and $0=V'_0\subsetneq\ldots\subsetneq V'_l=V'$ 
with $\theta$-stable quotients such that 
$\bigoplus_iV_{i+1}/V_i\cong\bigoplus_iV'_{i+1}/V'_i$.
\end{defi}

In particular, $\theta$-stable representations are S-equivalent 
if and only if they are isomorphic. Note that the S in S-equivalence doesn't 
refer to the base scheme $S$.

\begin{cor}
Via the morphism $R^\theta(Q,d)\to\mathcal M^\theta(Q,d)$
any $\theta$-semistable representation of $(Q,d)$ over an $S$-scheme $Y$
determines a morphism $Y\to\mathcal M^\theta(Q,d)$ over $S$.
This way for a geometric point $s\colon\Spec k\to S$ the geometric points 
$\mathcal M^\theta(Q,d)_s(k)$ over $s$ are in bijection with 
the S-equivalence classes of $\theta$-semistable representations
of $(Q,d)$ over $k$.
\end{cor}

\subsection{Decomposition of the weight space and limits of moduli spaces of quiver representations}
\label{subsec:weightspace-limitfunctor}

In the case of GIT quotients of projective varieties $X$ over an algebraically closed 
field the chamber structure of the space $\NS^G(X)_\Q,$ where $\NS^G(X)$ is the 
N\'eron-Severi group of $G$-equivariant line bundles, has been described in 
\cite{Th96}, \cite{DH98}, \cite{Re00}, and in the case $X=R(Q,d)$ of quivers 
(without oriented cycles) in \cite{HP02}, \cite{Ch08}. 

\medskip

Let $Q$ be a finite connected quiver, $d=(d_q)_{q\in Q_0}$ a dimension vector, 
$S$ be a scheme and $R(Q,d)$ the representation space over $S$. We have the space 
$\Q^{Q_0}$ of fractional linearisations of the line bundle $\O_{R(Q,d)}$. 

\begin{defi} We define the {\it weight space}
\[\textstyle
H(Q,d)=\big\{\theta\in\Q^{Q_0}\:|\:\theta(d)=0\big\}
\]
and
\[
C(Q,d)=\big\{\theta\in H(Q,d)\:|\:R^\theta(Q,d)\neq\emptyset\big\}.
\]
\end{defi}

For the finitely many dimension vectors $0<d'<d$ we define
\[
\begin{array}{l}
H_{d'}=\{\theta\in H(Q,d)\:|\:\theta(d')=0\},\\
H^{<0}_{d'}=\{\theta\in H(Q,d)\:|\:\theta(d')<0\},\\
H^{\leq0}_{d'}=\{\theta\in H(Q,d)\:|\:\theta(d')\leq0\}.\\
\end{array}
\]

\begin{rem}\label{rem:CQd}
Let the base scheme be $S=\Spec k$, $k$ an algebraically closed field.\\
(1) Existence of subrepresentations of a given dimension vector is a closed condition 
on $R(Q,d)$, see \cite[Lemma 3.1]{Scho92}.
We call a dimension vector $0<d'<d$ a generic dimension vector if there is 
a non-empty open subscheme of $R(Q,d)$ of representations that have a 
subrepresentation of dimension vector $d'$, or equivalently, all representations
of dimension vector $d$ have a subrepresentation of dimension vector $d'$.
There is a dense open subset in $R(Q,d)$ of representations which only have
subrepresentations of generic dimension vectors.\\
(2) The set of weights that allow (semi)stable representations is determined by 
the generic dimension vectors:
the intersection $\bigcap_{\textit{$d'\!\!$ generic}}H^{<0}_{d'}$ is the set of 
weights $\theta\in H(Q,d)$ such that $\theta$-stable representations exist, and
\[
C(Q,d)\;\;=\bigcap_{\textit{$d'\!\!$ generic}}\!\!\!H^{\leq0}_{d'}
\] 
Thus $C(Q,d)$ is a convex polyhedral cone in $H(Q,d)$, and it is full-dimensional 
if there is a $\theta\in H(Q,d)$ such that $\theta$-stable representations exist.
If the quiver $Q$ has no oriented cycles then $C(Q,d)$ is strongly convex.\\
(3) By \cite[Cor.\ 1]{Cr96}, improving the results of \cite{Scho92} (see Thm.\ 3.3, 
Thm.\ 5.4 and the text following the proof of Thm.\ 5.4 in \cite{Scho92}), the 
generic dimension vectors, and thus the cone $C(Q,d)$, are combinatorially determined.   
\end{rem}

\begin{defi}
Weights $\theta,\theta'\in H(Q,d)$ are called {\it GIT equivalent} if 
\[R^\theta(Q,d)=R^{\theta'}(Q,d).\]
We denote the {\it GIT equivalence class} of $\theta$ by 
\[C_\theta=\{\theta'\in H(Q,d)\:|\:R^\theta(Q,d)=R^{\theta'}(Q,d)\}\]
and denote
\[\overline{C}_\theta=\{\theta'\in H(Q,d)\:|\:R^\theta(Q,d)\subseteq R^{\theta'}(Q,d)\}.\]
\end{defi}

\medskip

For the standard results in the following proposition see \cite{Ch08} (following
\cite{Re00}; proven for quivers without oriented cycles, but the results and 
methods of proof also apply to quivers with oriented cycles).

\begin{prop}\label{prop:GITclasses}
Let $S=\Spec k$, $k$ an algebraically closed field.
For $\theta\in C(Q,d)$ the set $\overline{C}_\theta$ is a convex polyhedral
cone and $\relint(\overline{C}_\theta)=C_\theta$. 
The cones $\overline{C}_\theta$ for $\theta\in C(Q,d)$ form a finite 
fan with support $C(Q,d)$.
\end{prop}

In the situation of proposition \ref{prop:GITclasses} the decomposition of $C(Q,d)$ 
into GIT equivalence classes is determined by segments of the hyperplanes $H_{d'}$. 
The relevant part $W_{d'}\subseteq H_{d'}$ consists of those $\theta$ such that 
there is a $\theta$-semistable representation $V$ with a subrepresentation 
$V'\subset V$ of dimension vector $d'$ such that $\theta(V')=0$. Since in this case 
both $V'$ and $V/V'$ are also $\theta$-semistable, we have the equivalent condition 
that there is a $\theta$-semistable $V$ of the form $V\cong V'\oplus V''$ with 
$V',V''$ $\theta$-semistable of dimension vector $d',d-d'$, 
thus:
\[
W_{d'}=C(Q,d')\cap C(Q,d-d')
\]
Using the fact from remark \ref{rem:CQd}.(3) that the sets $C(Q,d)$ for any $d$ are 
combinatorially determined, this implies the following result.

\begin{prop}
The set $C(Q,d)\subseteq H(Q,d)$ and its decomposition into GIT equivalence classes 
over an algebraically closed field $k$ do not depend on the base field $k$.
\end{prop}

It follows that the above results concerning the decomposition of $H(Q,d)$, in particular 
proposition \ref{prop:GITclasses}, also apply to the case of an arbitrary base scheme $S$. 
The weight space decomposition over $S$ can be determined by looking at an arbitrary fibre 
(or the situation over any algebraically closed field).

\medskip

Assume that stable representations for some weight $\theta$ exist. Then $C(Q,d)$ is 
full-dimensional. The sets $W_{d'}$ of codimension one are called walls. 
Walls whose intersections with $\interior(C(Q,d))$ are nonempty are called inner walls, 
otherwise outer walls. A weight $\theta\in H(Q,d)$ is said to be generic if 
$R^\theta(Q,d)=R^\textit{$\theta$-s}(Q,d)\neq\emptyset$. The GIT equivalence classes 
of generic weights are open sets called chambers, the connected components of weights
$\theta\in\interior C(Q,d)$ not contained in an inner wall. 

\medskip

By variation of geometric invariant theory quotients (VGIT), cf.\ \cite{Th96}, 
\cite{DH98}, quiver varieties of nearby weights are related. 
Let $C_{\theta_0},C_\theta\subset C(Q,d)$ be GIT equivalence 
classes such that $C_{\theta_0}\subseteq\overline{C}_\theta$.
Then the inclusion $R^\theta(Q,d)\subseteq R^{\theta_0}(Q,d)$
induces a morphism over $S$
\[\phi_{\theta,\theta_0}\colon\mathcal M^\theta(Q,d)\to\mathcal M^{\theta_0}(Q,d)\] 
The morphisms 
$\phi_{\theta,\theta_0}\colon\mathcal M^\theta(Q,d)\to\mathcal M^{\theta_0}(Q,d)$
for representatives of the GIT equivalence classes in $C(Q,d)$
form an inverse system and we can consider its inverse limit 
$\varprojlim_\theta\,\mathcal M^\theta(Q,d)$. 

\begin{prop}\label{prop:functorinvlim}
Let $Q$ be a quiver and $(d_q)_{q\in Q_0}$ be an indivisible dimension vector. 
Assume that stable representations for some weight $\theta$ exist. 
Then the functor $\varprojlim_\theta\,\mathcal M^\theta(Q,d)$ is isomorphic to the 
closed subfunctor of the product of $\mathcal M^\theta(Q,d)$ for generic $\theta$
\begin{equation}\label{eq:functorinvlim}
Y\;\mapsto\;
\left\{\left.
(\phi_\theta)_\theta\in\!\!\prod_{\textit{$\theta$ generic}}\!\!\!
\mathcal M^\theta(Q,d)(Y)
\;\;\right|\;
\begin{array}{l}
\forall\,\theta,\theta'\,\textit{generic}\;\;\:
\forall\,\tau\in\overline{C}_\theta\cap\overline{C}_{\theta'}\colon\\
\phi_{\theta,\tau}\circ\phi_\theta=\phi_{\theta',\tau}\circ\phi_{\theta'}
\end{array}
\right\}
\end{equation}
\end{prop}
\begin{proof}
The limit functor is isomorphic to this functor since each $\tau\in C(Q,d)$ 
is contained in some $\overline{C}_\theta$ for a generic $\theta$.
This functor is a closed subfunctor of the product because $\mathcal M^\tau(Q,d)$
is a scheme separated over $S$ (use \cite[(5.2.5)]{EGA1}).
\end{proof}

\section{Weight space decomposition and moduli spaces for some bipartite quivers}
\label{sec:weightspace-moduli-PnQn}

\subsection{The quivers $P_n,Q_n$ and the structure of the weight space}
\label{subsec:weightspacePnQn}

We define the quivers $P_n$ and $Q_n$ with fixed dimension vectors (cf.\ introduction).

\begin{defi}
{\rm (a)} Let
$Q_n=(\{p,q_1,\ldots,q_n\},\{\alpha_1,\ldots,\alpha_n\})$
where $s(\alpha_i)=q_i$, $t(\alpha_i)=p$ with dimension vector given by $d_p=2$ 
and $d_{q_i}=1$.\\
{\rm (b)} Let $P_n=(\{p_1,p_2,q_1,\ldots,q_n\},\{\alpha_{1,1},\alpha_{1,2},\ldots,
\alpha_{n,1},\alpha_{n,2}\})$
where $s(\alpha_{i,j})=q_i$, $t(\alpha_{i,j})=p_j$ with dimension 
vector $(1,\ldots,1)$.
\end{defi}

For these quivers with dimension vectors we determine the decomposition of the 
weight space as considered in subsection \ref{subsec:weightspace-limitfunctor}. 
In the following when studying the structure of GIT classes in the cone 
$C(Q)\subset H(Q)$ for $Q=Q_n,P_n$ we restrict to the intersection with 
a certain affine hyperplane. In the case $Q_n$ this yields the hypersimplex 
$\Delta(2,n)$, in the case $P_n$ the product of simplices $\Delta^1\times\Delta^{n-1}$.

\begin{rem}\label{rem:Qn-weightspace} Decomposition of the weight space for $Q_n$.\\
(1) A weight for the quiver $Q_n$ is given by a tuple  
$\theta=(\eta,\theta_1,\ldots,\theta_n)$ where $\theta_p=\eta$, $\theta_{q_i}=\theta_i$. 
The weight space is the $n$-dimensional hyperplane
$H(Q_n)=\{\theta\in\Q^{n+1}\:|\:2\eta=-\sum_i\theta_i\}\subset\Q^{n+1}$. 
Omitting the coordinate $\eta$ we have an isomorphism $H(Q_n)\cong\Q^n$, 
with basis $e_1,\ldots,e_n\in\Q^n$ and coordinates $\theta_1,\ldots,\theta_n$. 
A semistable representation exists if and only if 
$0\leq\theta_i\leq\frac{1}{2}\sum_{i=1}^n\theta_i$ for all $i$, 
this defines a full-dimensional strongly convex polyhedral cone $C(Q_n)\subset H(Q_n)$. 
The affine hyperplane $\{\sum_{i=1}^n\theta_i=2\}$ intersects 
this cone in the hypersimplex 
\[\textstyle
C(Q_n)\cap\big\{\sum_{i=1}^n\theta_i=2\big\}\;=\;
\conv\big\{e_i+e_j\:|\:i\neq j\big\}\;=\;\Delta(2,n).
\]
(2) For $\theta=(\theta_1,\ldots,\theta_n)\in\Delta(2,n)$ a $\theta$-semistable 
representation that is not $\theta$-stable exists if and only if $\theta$ is 
contained in one of the following walls. The outer walls are of the form 
$W_i=\{\theta_i=0\}\cap\Delta(2,n)=\conv\{e_j+e_k\:|\:j,k\neq i\}
\cong\Delta(2,n-1)$ and 
$W_{\{\{i\},\{1,\ldots,n\}\setminus \{i\}\}}=\linebreak\{\theta_i=1\}\cap\Delta(2,n)
=\conv\{e_i+e_j\:|\:j\in\{1\ldots,n\}\setminus\{i\}\}
\cong\Delta(1,n-1)=\Delta^{n-2}$.
The inner walls are 
\[\textstyle
W_{\{J,J^\complement\}}\;=\;\big\{\theta\in\Delta(2,n)\:\big|\:\sum_{i\in J}\theta_i=1\big\}
\;=\;\conv\big\{e_i+e_j\:|\:i\in J,j\in J^\complement\big\}
\]
for $J\subset\{1,\ldots,n\}$, $2\leq|J|\leq n-2$, 
$J^\complement=\{1,\ldots,n\}\setminus J$. It is 
$W_{\{J,J^\complement\}}\cong\Delta^{|J|-1}\times\Delta^{n-|J|-1}$.\\
(3) An inner wall $W_{\{J,J^\complement\}}$ intersects the boundary of $\Delta(2,n)$
as follows:
$W_{\{J,J^\complement\}}\cap W_i=
\conv\{e_j+e_k\:|\:j\in J\setminus\{i\},k\in J^\complement\setminus\{i\}\}$ 
(i.e.\ a wall $W_{\{J\setminus\{i\},J^\complement\setminus\{i\}\}}
\subset\Delta(2,\{1,\ldots,n\}\setminus\{i\})$)
and, assuming $i\in J$,
$W_{\{J,J^\complement\}}\cap W_{\{\{i\},\{i\}^\complement\}}
=\conv\{e_i+e_j\:|\:j\in J^\complement\}=\Delta^{|J^\complement|-1}$.\\
(4) Distinct inner walls $W_{\{J_1,J_1^\complement\}},W_{\{J_2,J_2^\complement\}}
\subset\Delta(2,n)$ do not intersect in the interior of $\Delta(2,n)$ 
if and only if the two partitions of $\{1,\ldots,n\}$ into two subsets 
are compatible in the sense that there is an inclusion between the sets
$J_1,J_1^\complement$ and $J_2,J_2^\complement$.
In this case, if we assume $J_1\subset J_2$, the intersection
$W_{\{J_1,J_1^\complement\}}\cap W_{\{J_2,J_2^\complement\}}
=\conv\{e_j+e_k\:|\:j\in J_1,k\in J_2^\complement\}
=\Delta^{|J_1|-1}\times\Delta^{|J_2^\complement|-1}$
is a wall of the face $\Delta(2,J_1\cup J_2^\complement)$
of $\Delta(2,n)$.
\end{rem}

We have the representation space $R(Q_n)=(\A^2)^n$ with the operation of 
$\GL(2)\times(\G_m)^n$. A free representation of $Q_n$ over a scheme $Y$ 
after a choice of a basis can be expressed as a tuple $V=(s_1,\ldots,s_n)$ 
of $n$ sections of $\A^2_Y\to Y$. 
We write $s_i\sim s_j$ if $s_i$ and $s_j$ are in the same $\G_m$-orbit.

\begin{lemma}\label{lemma:repr-pol-Qn}
Let $V=(s_1,\ldots,s_n)$ be a representation of $Q_n$ over an algebraically closed 
field that is semistable with respect to some weight in the interior of $\Delta(2,n)$,
in particular all $s_i$ are nonzero. Let the partition $\{1,\ldots,n\}=\bigsqcup_lJ_l$ 
be defined by $i,j\in J_l$ for some $l$ if and only if $s_i\sim s_j$.
The set $\Theta(V)=\{\theta\in\Delta(2,n)\:|\:\textit{$V$ is $\theta$-semistable}\}$ 
is a closed convex polytope in $\Delta(2,n)$ with the following properties:
\begin{enumerate}[\rm(i)]
\item We have 
$\Theta(V)=\bigcap_l\{\theta\in\Delta(2,n)\:|\:\sum_{i\in J_l}\theta_i\leq 1\}$,
i.e. $\Theta(V)$ is bounded within $\Delta(2,n)$ by the inner walls 
$W_{\{J_l,J_l^\complement\}}$ for $2\leq|J_l|\leq n-2$. Any two distinct of 
these walls do not intersect in the interior of $\Delta(2,n)$.
\item The sets of vertices and edges of $\Theta(V)$ are subsets of the sets 
of vertices and edges of $\Delta(2,n)$.
$\Theta(V)=\conv\{e_i+e_j\:|\:\textit{$i\in J_l$, $j\in J_{l'}$ for some $l\neq l'$}\}$.
\item $\Theta(V)$ is full dimensional if and only if there are $\geq3$
classes $J_l$, or equivalently, $V$ is stable with respect to some
weight. Otherwise there are two classes $J$ and $J^\complement$ and 
$\Theta(V)=W_{\{J,J^\complement\}}$.
\end{enumerate}
\end{lemma}
\begin{proof} 
For (i) note that the nontrivial subrepresentations of $V$ with one-dimensional 
subspace in $V_q$ are given by the sets $J_l$ (the spaces ($V_{q_j})_{j\in J_l}$ 
together with the one-dimensional subspace of $V_p$ determined by $(s_j)_{j\in J_l}$). 
That the walls do not intersect in the interior follows from remark 
\ref{rem:Qn-weightspace}.(4).
Concerning (ii), as the inner walls bounding $\Theta(V)$ do not intersect in the
interior of $\Delta(2,n)$, the first statement follows inductively using
remark \ref{rem:Qn-weightspace}.(3) and (4). Then $\Theta(V)$ is the convex hull
of the vertices of $\Delta(2,n)$ it contains.
The statements in (iii) are easy to show.
\end{proof}

The property that the sets of vertices and edges of $\Theta(V)$
are subsets of the sets of vertices and edges of $\Delta(2,n)$ means
that $\Theta(V)$ is a matroid polytope.
Matroid polytopes were defined in \cite{GGMS87} (see also \cite[(1.2.6)]{Ka93a}), 
and by \cite[Thm.\ 4.1]{GGMS87} matroid polytopes in the hypersimplex $\Delta(k,n)$
correspond to matroids of rank $k$ on $\{1,\ldots,n\}$.
In the theory of GIT $\Theta(V)$ corresponds to the stability set of $V$ 
denoted $\Omega(V)$ in \cite{Re00}.

\medskip

Considering the quiver $P_n$ we have the representation space $R(P_n)=(\A^1\times\A^1)^n$ 
with the operation of $(\G_m\times\G_m)\times(\G_m)^n$. We write a free representation 
of $P_n$ over a scheme $Y$ after a choice of a basis as a tuple $V=(s_1,\ldots,s_n)$ 
of sections of $(\A^1\times\A^1)_Y\to Y$ where $s_i$ corresponds to 
$(\alpha_{i,1},\alpha_{i,2})$.
We set $s_i\sim s_j$ if $s_i$ and $s_j$ are in the same $\G_m$-orbit and add 
the sections $s_0=(0,1)$, $s_\infty=(1,0)$ which are only used to compare with 
sections $s_i$ up to $\sim$.

\begin{rem}\label{rem:Pn-weightspace} 
Decomposition of the weight space for $P_n$.\\
(1) A weight for the quiver $P_n$ is an element
$\theta=(\eta_1,\eta_2,\theta_1,\ldots,\theta_n)$ where
$\theta_{p_i}=\eta_i$, $\theta_{q_i}=\theta_i$. We also write this as 
$\theta=\eta_1f_1+\eta_2f_2+\sum_i\theta_ie_i$ with 
$f_1,f_2,e_1,\ldots,e_n$ a basis of $\Q^{n+2}$.
The weight space is the $(n+1)$-dimensional hyperplane 
$H(P_n)=\{\sum_i\theta_i=-\eta_1-\eta_2\}\subset\Q^{n+2}$.
A semistable representation exists if and only if 
$\eta_1,\eta_2\leq0$ and $\theta_1,\ldots,\theta_n\geq 0$. 
This defines a full-dimensional cone $C(P_n)\subset H(P_n)$.
The affine hyperplane defined by $\eta_1+\eta_2=-1$ intersects this cone in
a product of simplices
\[
C(P_n)\cap\{\eta_1\!+\!\eta_2\!=\!-1\}
=\conv\{e_1\!-\!f_1,\ldots,e_n\!-\!f_1,e_1\!-\!f_2,\ldots,e_n\!-\!f_2\}
=\Delta^1\times\Delta^{n-1}.
\]
(2) The outer walls are of the form 
$W_{q_i}=\{\theta_i=0\}\cap(\Delta^1\times\Delta^{n-1})
\cong\Delta^1\times\Delta^{n-2}$ and 
$W_{p_j}=\{\eta_j=0\}\cap(\Delta^1\times\Delta^{n-1})\cong\Delta^{n-1}$.
The inner walls are 
\[\textstyle
W_J=\big\{\theta\in\Delta^1\times\Delta^{n-1}\:\big|\:\sum_{i\in J}\theta_i=-\eta_1\big\}
=\conv\big(\{e_i-f_1\:|\:i\in J\}\cup\{e_i-f_2\:|\:i\in J^\complement\}\big)
\]
for $J\subset\{1,\ldots,n\}$, $1\leq|J|\leq n-1$, $n\geq2$. 
It is $W_J\cong\Delta^{n-1}$.\\
(3) An inner wall $W_J$ intersects the boundary of $\Delta^1\times\Delta^{n-1}$
as follows:
$W_J\cap W_{q_i}=\conv\big(\{e_j-f_1\:|\:j\in J\setminus\{i\}\}\cup\{e_j-f_2\:|\:
j\in J^\complement\setminus\{i\}\}\big)$ 
(i.e.\ a wall $W_{J\setminus\{i\}}\subset\Delta^1\times\Delta^{n-2}$)
and
$W_J\cap W_{q_1}=\conv\big(\{e_i-f_2\:|\:i\in J^\complement\}\big)
=\Delta^{|J^\complement|-1}$, 
$W_J\cap W_{q_2}=\conv\big(\{e_i-f_1\:|\:i\in J\}\big)
=\Delta^{|J|-1}$.\\
(4) Two inner walls $W_{J_1},W_{J_2}$ do not intersect in the interior of 
$\Delta^1\times\Delta^{n-1}$ if and only if $J_1\subsetneq J_2$
or $J_2\subsetneq J_1$.
\end{rem}

\begin{rem}
The notation is such that for $\tau\in\relint(W_J)$ representations
$(s_1,\ldots,s_n)$ with $s_j\sim s_\infty=(1,0)$ for $j\in J$ and
$s_i\sim s_0=(0,1)$ for $i\in J^\complement$ are strictly $\tau$-semistable.
\end{rem}

The following lemma can be proven in the same way as lemma \ref{lemma:repr-pol-Qn} 
using remark \ref{rem:Pn-weightspace}.

\begin{lemma}\label{lemma:repr-pol-Pn}
Let $V=(s_1,\ldots,s_n)$ be a representation of $P_n$ over an algebraically closed 
field that is semistable with respect to some weight in the interior of 
$\Delta^1\times\Delta^{n-1}$, in particular all $s_i$ are nonzero.
Let $J_0,J_\infty\subset\{1,\ldots,n\}$ be defined by 
$J_0=\{i\:|\:s_i\sim s_0=(0,1)\}$, $J_\infty=\{i\:|\:s_i\sim s_\infty=(1,0)\}$. The set 
$\Theta(V)=\{\theta\in\Delta^1\times\Delta^{n-1}\:|\:\textit{$V$ is $\theta$-semistable}\}$
is a closed convex polytope in $\Delta^1\times\Delta^{n-1}$ with the following 
properties:
\begin{enumerate}[\rm(i)]
\item $\Theta(V)=\{\theta\in\Delta^1\times\Delta^{n-1}\:|\:
\sum_{i\in J_0}\theta_i\leq-\eta_2\}
\cap\{\theta\in\Delta^1\times\Delta^{n-1}\:|\:\sum_{i\in J_\infty}\theta_i\leq-\eta_1\}$,
i.e. $\Theta(V)$ is bounded within $\Delta^1\times\Delta^{n-1}$ by the walls 
$W_{\!J_0^\complement},W_{J_\infty}$. These walls, if distinct, do not intersect in 
the interior of $\Delta^1\times\Delta^{n-1}$.
\item The set of vertices of $\Theta(V)$ is a subset of the set of vertices of 
$\Delta^1\times\Delta^{n-1}$. 
It is $\Theta(V)=
\conv\big(\{e_i-f_2\:|\:i\in J_0\}\cup\{e_i-f_1\:|\:i\in J_\infty\}\big)$.
\item $\Theta(V)$ is full dimensional if and only if 
$(J_0\cup J_\infty)^\complement\neq\emptyset$, or equivalently, $V$ 
is stable with respect to some weight. Otherwise 
$\Theta(V)=W_{\!J_0^\complement}=W_{J_\infty}$.
\end{enumerate}
\end{lemma}

\subsection{Relation between representations of $Q_{n+2}$ and $P_n$}
\label{subsec:Qn+2-Pn}

We show that the structure of GIT equivalence classes in 
$\Delta^1\times\Delta^{n-1}\subset H(P_n)$ arises from the
structure of GIT equivalence classes in $\Delta(2,n+2)$ in the 
neighbourhood of some vertex $e_a+e_b$ and compare the corresponding
stacks and moduli spaces of representations.

\medskip

The subset $\Delta(2,n+2)_{a,b}=\Delta(2,n+2)\cap
\{\theta_a+\theta_b\geq\sum_{i\neq a,b}\theta_i\}\setminus\{e_a+e_b\}$ 
is part of a cone with apex $e_a+e_b$ and we can project it onto
$W_{\{\{a,b\},\{a,b\}^\complement\}}=\Delta(2,n+2)\cap\linebreak
\{\theta_a+\theta_b=\sum_{i\neq a,b}\theta_i\}$ by
\[
\begin{array}{rcl}
\Delta(2,n+2)_{a,b}&\to&W_{\{\{a,b\},\{a,b\}^\complement\}}\\
\theta=(\theta_a,\theta_b,(\theta_i)_{i\neq a,b})&\mapsto&
\bar{\theta}=(\lambda(\theta_a-1)+1,\lambda(\theta_b-1)+1,(\lambda\theta_i)_{i\neq a,b})\\
\end{array}
\]
where $\lambda=1/\sum_{i\neq a,b}\theta_i$.
The wall $W_{\{\{a,b\},\{a,b\}^\complement\}}=\conv(\{e_i+e_a\:|\:i\neq a,b\}\cup\linebreak
\{e_i+e_b\:|\:i\neq a,b\})\subset\Delta(2,n+2)$ can be identified with 
$\Delta^1\times\Delta^{n-1}=\conv(\{e_i-f_2\:|\:i\neq a,b\}\cup\linebreak
\{e_i-f_1\:|\:i\neq a,b\})\subset H(P_n)$ by 
\[
\begin{array}{rcl}
W_{\{\{a,b\},\{a,b\}^\complement\}}&\to&\Delta^1\times\Delta^{n-1}\subset H(P_n)\\[1mm]
\bar{\theta}=(\bar{\theta}_a,\bar{\theta}_b,(\bar{\theta}_i)_{i\neq a,b})&\mapsto&
\theta'=(\eta'_1=\bar{\theta}_a-1,\eta'_2=\bar{\theta}_b-1,
(\theta'_i=\bar{\theta}_i)_{i\neq a,b})\\
\end{array}
\]
We observe that the projections of the walls of $\Delta(2,n+2)$ near $e_a+e_b$
coincide with the walls of $\Delta^1\times\Delta^{n-1}\subset H(P_n)$. 
The walls meeting the interior of $\Delta(2,n+2)_{a,b}$ are the walls 
$W_{\{J,J^\complement\}}$ with $a\in J,b\in J^\complement$ or $a\in J^\complement,b\in J$. 
Their intersections with $W_{\{\{a,b\},\{a,b\}^\complement\}}$ are exactly the 
inner walls of $\Delta^1\times\Delta^{n-1}$: assume $a\in J, b\in J^\complement$, 
then the equation $\sum_{j\in J}\theta_j=1$ for $W_{\{J,J^\complement\}}$ 
is equivalent to the equation $\sum_{j\in J\setminus\{a\}}\theta'_j=-\eta'_1$
defining $W_{J\setminus\{a\}}\subset\Delta^1\times\Delta^{n-1}$.

\begin{prop}\label{prop:Qn+2-P_n}
Let $a,b\in\{1,\ldots,n+2\}$, $a\neq b$.\\[1mm]
{\rm (a)} The maps 
\begin{equation}\label{eq:maprepr-Qn+2-P_n}
\begin{array}{rcl}
\left\{\!
\begin{array}{l}
\textit{free representations with fixed bases}\\
\textit{$V=(s_1,\ldots,s_{n+2})$ of $Q_{n+2}$ over $Y$}\\
\textit{such that $s_a\cap 0=s_b\cap 0=\emptyset$}\\ 
\textit{and $\G_ms_a\!\cap\G_ms_b=\emptyset$}
\end{array}\!\right\}
&\rightarrow&
\left\{\!
\begin{array}{l}
\textit{free representations with}\\
\textit{fixed bases of $P_n$ over $Y$}
\end{array}\!\right\}
\\[9mm]
V=(s_1,\ldots,s_{n+2})&\mapsto&V'=(s'_i)_{i\neq a,b}
\end{array}
\end{equation}
for schemes $Y$ over a base scheme $S$, where $s'_i$ arises from $s_i$ by the 
base change that transforms $s_a$ to $(1,0)$ and $s_b$ to $(0,1)$, induce an 
isomorphism of stacks
\begin{equation}\label{eq:isom-Qn+2-Pn}
[R(Q_{n+2})_{a,b}/((\GL(2)\times(\G_m)^n)/\G_m)]\;\to\;
[R(P_n)/((\G_m\times\G_m)\times(\G_m)^n)/\G_m)]
\end{equation}
where $R(Q_{n+2})_{a,b}=R(Q_{n+2})\setminus(\{s_a=0\}\cup\{s_b=0\}\cup
\{\G_m s_a=\G_ms_b\})$.\\[1mm]
{\rm (b)} We have the map
\begin{equation}\label{eq:proj-Qn+2-Pn}
\begin{array}{rcl}
\Delta(2,n+2)_{a,b}&\to&\Delta^1\times\Delta^{n-1}\subset H(P_n)\\
\theta&\mapsto&\theta'=(\eta'_1=\lambda(\theta_a-1),\eta'_2=\lambda(\theta_b-1),
(\theta'_i=\lambda\theta_i)_{i\neq a,b})\\
\end{array}
\end{equation}
where $\lambda=1/\sum_{i\neq a,b}\theta_i$.
On the fibres of this map the $\theta$-(semi)stable locus in
$R(Q_{n+2})_{a,b}$ does not change.\\
{\rm (c)} 
Let $\theta\in W_{\{\{a,b\},\{a,b\}^\complement\}}\subset\Delta(2,n+2)$ and 
$\theta'\in\Delta^1\times\Delta^{n-1}\subset H(P_n)$ its image under the
map {\rm(\ref{eq:proj-Qn+2-Pn})}.
Let $l\in\N_{>0}$ such that $l\theta$ and $l\theta'$ are integral. Then
the equivariant line bundles $\O_{R(Q_{n+2})_{a,b}}$ with $l\theta$-linearisation 
and $\O_{R(P_n)}$ with $l\theta'$-linearisation define isomorphic line bundles 
on the quotient stacks under the identification {\rm(\ref{eq:isom-Qn+2-Pn})}.
\end{prop}
\begin{proof}
(a) The maps on the sets of morphisms are given by
composition with the base change maps that transform $s_a,s_b$ to $(1,0),(0,1)$
and forgetting the automorphisms of the spaces corresponding to the vertices $q_a,q_b$.
The inverse functor is given on the objects by maps $(s_i)_{i\neq a,b}\mapsto(s_i)_i$
setting $s_a=(1,0)$, $s_b=(0,1)$ and on the morphisms by taking those
automorphisms of the spaces corresponding to the vertices $q_a,q_b$ such that
$s_a,s_b$ remain fixed. One checks that both compositions of these
two functors are isomorphic to the identity functors.\\
(b) The fibres are the line segments $[\bar{\theta},e_a+e_b]\setminus\{e_a+e_b\}$
for $\bar{\theta}\in W_{\{\{a,b\},\{a,b\}^\complement\}}$.
All $\theta$ contained in this line segment are elements of the same GIT equivalence class 
except $\bar{\theta}$. But $R^{\bar{\theta}}(Q_{n+2})$ differs from 
$R^\theta(Q_{n+2})$ for the other $\theta$ only outside $R(Q_{n+2})_{a,b}$.\\
(c) The isomorphism (\ref{eq:isom-Qn+2-Pn}) identifies the structure sheaves of the 
two stacks. We have to verify that the additional multiplications by the corresponding 
characters coincide. It suffices to consider the groupoids of $Y$-valued points for 
$S$-schemes $Y$ of the two given groupoid schemes. The elements of 
$((\GL(2)\times(\G_m)^n)/\G_m)(Y)$ fixing $s_a=(1,0)$, $s_b=(0,1)$ are of the 
form $\Big(\left(\begin{smallmatrix}\alpha_a&0\\0&\alpha_b\end{smallmatrix}\right),
\alpha_a,\alpha_b,(\alpha_i)_{i\neq a,b}\Big)$,
the corresponding element
$(\beta_1,\beta_2,(\alpha_i)_{i\neq a,b})\in((\G_m\times\G_m)\times(\G_m)^n)/\G_m)(Y)$ 
satisfies $\beta_1=\alpha_a$, $\beta_2=\alpha_b$.
Applying the character corresponding to 
$l\theta=(l\eta,l\theta_1,\ldots,l\theta_{n+2})$ we obtain
$(\alpha_a\alpha_b)^{l\eta}\alpha_a^{l\theta_a}\alpha_b^{l\theta_b}
\prod_{i\neq a,b}\alpha_i^{l\theta_i}$ and applying $l\theta'$ we have 
$\beta_1^{l\eta'_1}\beta_2^{l\eta'_2}\prod_{i\neq a,b}\alpha_i^{l\theta'_i}$.
As $\eta=-1$, this coincides for $\theta'_i=\theta_i$, $\eta'_1=\theta_a-1$,
$\eta'_2=\theta_b-1$.
\end{proof}

\begin{cor}\label{cor:Qn+2-P_n}
Let $a,b\in\{1,\ldots,n+2\}$, $a\neq b$.\\
{\rm (a)} For $\theta\in\Delta(2,n+2)_{a,b}\setminus W_{\{\{a,b\},\{a,b\}^\complement\}}$
a free representation $V=(s_1,\ldots,s_{n+2})$ of $Q_{n+2}$ is $\theta$-(semi)stable 
if and only if $V$ satisfies $s_a\cap 0=\emptyset$, $s_b\cap 0=\emptyset$, 
$\G_m s_a\cap\G_ms_b=\emptyset$ and its image under {\rm(\ref{eq:maprepr-Qn+2-P_n})} 
is a $\theta'$-(semi)stable representation of $P_n$, where $\theta'$ is the image of 
$\theta$ under the map {\rm(\ref{eq:proj-Qn+2-Pn})}.\\
{\rm (b)} The map {\rm(\ref{eq:proj-Qn+2-Pn})} defines a bijection between
the GIT equivalence classes in $\Delta^1\times\Delta^{n-1}$ and in
$\Delta(2,n+2)_{a,b}\setminus W_{\{\{a,b\},\{a,b\}^\complement\}}$ 
(resp.\ $ W_{\{\{a,b\},\{a,b\}^\complement\}}$).\\
{\rm (c)} For  $\theta\in\Delta(2,n+2)_{a,b}\setminus W_{\{\{a,b\},\{a,b\}^\complement\}}$
the isomorphism {\rm(\ref{eq:isom-Qn+2-Pn})} induces an isomorphism of stacks
\[
[R^\theta(Q_{n+2})/((\GL(2)\times(\G_m)^n)/\G_m)]\;\to\;
[R^{\theta'}(P_n)/((\G_m\times\G_m)\times(\G_m)^n)/\G_m)]
\]
and thus an isomorphism of their moduli spaces
\[
\mathcal M^\theta(Q_{n+2})\;\to\;\mathcal M^{\theta'}(P_n).
\]
These isomorphisms determine an isomorphism of the inverse systems of the moduli spaces 
$\mathcal M^\theta(Q_{n+2})$ for $\theta\in\Delta(2,n+2)_{a,b}\setminus 
W_{\{\{a,b\},\{a,b\}^\complement\}}$ and $\mathcal M^{\theta'}(P_n)$ 
for $\theta'\in\Delta^1\times\Delta^{n-1}$. 
\end{cor}

\subsection{Moduli spaces for $P_n,Q_n$ and GIT quotients of products of $\P^1$'s}
\label{subsec:quivermoduli-GITP1}

We compare the moduli spaces of representations of $Q_n$ and $P_n$ to GIT quotients 
of $(\P^1)^n$ by the quotient groups 
$((\GL(2)\times(\G_m)^n)/\G_m)/((\G_m\times(\G_m)^n)/\G_m)\cong\PGL(2)$ and
$(((\G_m)^2\times(\G_m)^n)/\G_m)/((\G_m\times(\G_m)^n)/\G_m)\cong\G_m$.

\medskip

We have $\NS((\P^1)^n)=\Pic((\P^1)^n)\cong\Z^n$, where 
$(\theta_1,\ldots,\theta_n)\in\Z^n$ corresponds to the line bundle 
$\O_{(\P^1)^n}(\theta_1,\ldots,\theta_n)=\O_{\P^1}(\theta_1)\boxtimes\ldots\boxtimes
\O_{\P^1}(\theta_n)$. A line bundle 
$\O_{(\P^1)^n}(\theta_1,\ldots,\theta_n)$ with 
$\theta_i\in 2\Z$ for all $i$ has a $\PGL(2)$-linearisation ($\O_{\P^1}(-2)$
being the canonical sheaf of $\P^1$), and if an element of $\Pic((\P^1)^n)$ has a 
$\PGL(2)$-linearisation then this linearisation is unique. Thus we can identify 
$\NS^{\PGL(2)}((\P^1)^n)_\Q=\Pic^{\PGL(2)}((\P^1)^n)_\Q\cong\Q^n$ 
with $H(Q_n)$ by 
$(\theta_1,\ldots,\theta_n)\leftrightarrow(\eta=-\frac{1}{2}\sum_i\theta_i,\theta_1,
\ldots,\theta_n)$.

\begin{prop}\label{prop:Q_n-GITP1}
Let $R^*(Q_n)=R(Q_n)\setminus\bigcup_i\{s_i=0\}\cong(\A^2\setminus\{0\})^n$.\\
{\rm (a)} The morphism $R^*(Q_n)\to(\P^1)^n$ induces an isomorphism of stacks
\begin{equation}\label{eq:isom-Q_n-GITP1}
[R^*(Q_n)/((\GL(2)\times(\G_m)^n)/\G_m)]\;\to\;[(\P^1)^n/\PGL(2)].
\end{equation}
{\rm (b)} Let $\theta=(\eta,\theta_1,\ldots,\theta_n)\in H(Q_n)$ be integral. 
Then the line bundle $\O_{(\P^1)^n}(\theta_1,\ldots,\theta_n)$ has a 
(unique) $\PGL(2)$-linearisation and the equivariant line bundles $\O_{R^*(Q_n)}$ with 
$\theta$-linearisation and $\O_{(\P^1)^n}(\theta_1,\ldots,\theta_n)$ define 
isomorphic line bundles on the quotient stacks under the identification 
in {\rm(\ref{eq:isom-Q_n-GITP1})}.
\end{prop}
\begin{proof}
Consider the operation of the subgroup $(\G_m)^n\cong(\G_m\times(\G_m)^n)/\G_m
\subset(\GL(2)\times(\G_m)^n)/\G_m$ on $R^*(Q_n)\cong(\A^2\setminus\{0\})^n$.
The operation of $\G_m$ on $\A^2\setminus\{0\}$ is free and $\A^2\setminus\{0\}\to\P^1$
a $\G_m$-torsor. 
The structure sheaf on $\A^2\setminus\{0\}$ with $\G_m$-linearisation by $\theta$ 
descends to the line bundle $\O_{\P^1}(\theta)$. Thus $R^*(Q_n)\to(\P^1)^n$ is a 
$(\G_m)^n$-torsor and $\O_{R^*(Q_n)}$ with $(\G_m)^n$-linearisation by 
$(\theta_1,\ldots,\theta_n)$ descends to the line bundle 
$\O_{(\P^1)^n}(\theta_1,\ldots,\theta_n)$.\\
(a) A $Y$-valued point of $[R^*(Q_n)/((\GL(2)\times(\G_m)^n)/\G_m)]$ for some 
$S$-scheme $Y$ corresponds to a $(\GL(2)\times(\G_m)^n)/\G_m$-torsor
$E\to Y$ with an equivariant morphism $E\to R^*(Q_n)$.
Taking quotients by $(\G_m)^n\subset (\GL(2)\times(\G_m)^n)/\G_m$ 
we obtain a $\PGL(2)$-equivariant morphism $E/(\G_m)^n\to(\P^1)^n$
where $E/(\G_m)^n\to Y$ is a $\PGL(2)$-torsor. The diagram
\[
\begin{array}{ccc}
E&\longrightarrow&(\A^2\setminus\{0\})^n\\
\downarrow&&\downarrow\\
E/(\G_m)^n&\longrightarrow&(\P^1)^n\\
\end{array}
\]
is cartesian and its vertical arrows are $(\G_m)^n$-torsors. For a $Y$-valued point of 
\linebreak
$[(\P^1)^n/\PGL(2)]$, a $\PGL(2)$-torsor $\overline{E}\to Y$ with equivariant morphism
to $(\P^1)^n$, we construct a $(\GL(2)\times(\G_m)^n)/\G_m$-torsor over $Y$ with an 
equivariant morphism to $(\A^2\setminus\{0\})^n$ by taking the fibred product 
$\overline{E}\times_{(\P^1)^n}(\A^2\setminus\{0\})^n$.
One verifies that this way two functors between
$[R^*(Q_n)/((\GL(2)\times(\G_m)^n)/\G_m)]$ and $[(\P^1)^n/\PGL(2)]$
are defined whose compositions are isomorphic to the identity functors.\\
(b) As the categories of sheaves do not change under stackification we may work with 
the prestack $[R^*(Q_n)/((\GL(2)\times(\G_m)^n)/\G_m)]^\textit{pre}$ of trivial 
$(\GL(2)\times(\G_m)^n)/\G_m$-torsors $E\to Y$ with an equivariant morphism 
$E\to R^*(Q_n)$, and its image in $[(\P^1)^n/\PGL(2)]$ under the isomorphism 
(\ref{eq:isom-Q_n-GITP1}). 
Objects in $[R^*(Q_n)/((\GL(2)\times(\G_m)^n)/\G_m)]^\textit{pre}$ over an 
$S$-scheme $Y$ correspond to morphisms $Y\to R^*(Q_n)$, a morphism 
$(g,y)\colon Y\to (\GL(2)\times(\G_m)^n)/\G_m\times R^*(Q_n)$ gives an arrow 
$y\to gy$ over $\id_Y$.

Let $\mathscr L$ be the line bundle on 
$[R^*(Q_n)/((\GL(2)\times(\G_m)^n)/\G_m)]^\textit{pre}$ defined by the 
equivariant line bundle $\O_{R^*(Q_n)}$ with linearisation by $\theta$, in 
particular $\mathscr L(y)=\Gamma(Y,y^*\O_{R^*(Q_n)})$ for $y\colon Y\to R^*(Q_n)$. 
We describe the push-forward $\overline{\mathscr L}$ with respect to the isomorphism 
(\ref{eq:isom-Q_n-GITP1}), which by definition is given by 
$\overline{\mathscr L}(\overline{y})
=\varprojlim_{\overline{y'}\to\overline{y}}\mathscr L(y')$ where $\overline{y'}$ is the 
image of the object $y'$ under (\ref{eq:isom-Q_n-GITP1}).
For open embeddings $\overline{y}\colon\overline{U}\hookrightarrow(\P^1)^n$ and 
$y\colon U\hookrightarrow R^*(Q_n)$ such that $U\subseteq R^*(Q_n)$ is the preimage of
$\overline{U}\subseteq(\P^1)^n$ we have
$\overline{\mathscr L}(\overline{y})
=\varprojlim_{\overline{y'}\to\overline{y}}\mathscr L(y')
=\varprojlim_{h\in(\G_m)^n(U),\overline{hy}\to\overline{y}}\mathscr L(hy)
=\Gamma(U,\O_{R^*(Q_n)})^{(\G_m)^n}
=\Gamma(\overline{U},\O_{(\P^1)^n}(\theta_1,\ldots,\theta_n))$.
Further we have the natural restriction maps. Thus the line bundle 
$\overline{\mathscr L}$ comes from the line bundle 
$\O_{(\P^1)^n}(\theta_1,\ldots,\theta_n)$ on $(\P^1)^n$ with some $\PGL(2)$-linearisation.
\end{proof}

\begin{cor}\label{cor:reprQn-P1n-GIT}
{\rm (a)} We have an identification
$\Pic^{\PGL(2)}((\P^1)^n)\cong H(Q_n)\cap\Z^{n+1}$ by 
$(\theta_1,\ldots,\theta_n)\leftrightarrow(\eta=-\frac{1}{2}\sum_i\theta_i,\theta_1,
\ldots,\theta_n)$.\\
{\rm (b)} For $(\eta,\theta_1,\ldots,\theta_n)\in H(Q_n)$ such that $\theta_i>0$ 
for all $i$ a free representation $V=(s_1,\ldots,s_n)$ of $Q_n$ is 
$(\eta,\theta_1,\ldots,\theta_n)$-(semi)stable if and only if 
$s_i\cap 0=\emptyset$ for all $i$ and its image under 
{\rm(\ref{eq:isom-Q_n-GITP1})} is a $(\theta_1,\ldots,\theta_n)$-(semi)stable 
section of $(\P^1)^n$.\\
{\rm (c)} The image of $\interior C(Q_n)\subset H(Q_n)$ under the isomorphism 
$\Pic^{\PGL(2)}((\P^1)^n)_\Q\cong H(Q_n)$ is the $\PGL(2)$-ample cone in 
$\Pic^{\PGL(2)}((\P^1)^n)_\Q$. The GIT classes in these open cones coincide.\\
{\rm (d)} For $\theta\in H(Q_n)\cong\Pic^{\PGL(2)}((\P^1)^n)_\Q$ such that 
$\theta_i>0$ for all $i$ the isomorphism {\rm(\ref{eq:isom-Q_n-GITP1})} induces 
isomorphisms of the stacks of $\theta$-(semi)stable points, of their moduli spaces 
and their inverse systems.
\end{cor}

One can also compare these stacks with $[G(2,n)/((\G_m)^n/\G_m)]$,
cf.\ \cite[2.4]{Ka93a}.

\medskip

Similarly, one can show the analogous results for the quiver $P_n$.
In the case of the quiver $P_n$ we have
$\NS^{\G_m}((\P^1)^n)=\Pic^{\G_m}((\P^1)^n)\cong\Z^{n+1}$, 
where $(\theta_1,\dots,\theta_n)\in\Z^n$ defines an element
$\O_{(\P^1)^n}(\theta_1,\ldots,\theta_n)\in\Pic((\P^1)^n)$ and
for each $\O_{(\P^1)^n}(\theta_1,\ldots,\theta_n)$ the set of $\G_m$-linearisations
is a principal homogeneous space under the character group $\Z$ of $\G_m$.
Below we identify $\Pic^{\G_m}((\P^1)^n)$ with $H(P_n)\cap\Z^{n+2}$, where
$(\eta_1,\eta_2,\theta_1,\ldots,\theta_n)\in H(P_n)\cap\Z^{n+2}$ corresponds 
to the line bundle $\O_{(\P^1)^n}(\theta_1,\ldots,\theta_n)$ with a certain
$\G_m$-linearisation, and the $\G_m$-linearisations on 
$\O_{(\P^1)^n}(\theta_1,\ldots,\theta_n)$ corresponding 
to $(\eta_1,\eta_2,\theta_1,\ldots,\theta_n)$ and 
$(\eta_1',\eta_2',\theta_1,\ldots,\theta_n)$ differ by the character
$(\eta,-\eta)=(\eta_1-\eta_1',\eta_2-\eta_2')$ of $(\G_m)^2/\G_m\cong\G^m$.

\begin{prop}\label{prop:P_n-GITP1}
Let $R^*(P_n)=R(P_n)\setminus\bigcup_i\{s_i=0\}\cong(\A^2\setminus\{0\})^n$.\\
{\rm (a)} The morphism $R^*(P_n)\to(\P^1)^n$ induces an isomorphism of stacks
\begin{equation}\label{eq:isom-P_n-GITP1}
[R^*(P_n)/(((\G_m)^2\times(\G_m)^n)/\G_m)]\;\to\;[(\P^1)^n/\G_m].
\end{equation}
{\rm (b)} Let $(\eta_1,\eta_2,\theta_1,\ldots,\theta_n)\in H(P_n)$ be integral. 
Then the line bundle $\O_{R^*(P_n)}$ with\linebreak
$((\G_m)^2\times(\G_m)^n)/\G_m$-linearisation
given by $(\eta_1,\eta_2,\theta_1,\ldots,\theta_n)$ descends to the line bundle 
$\O_{(\P^1)^n}(\theta_1,\ldots,\theta_n)$ with a certain $\G_m$-linearisation, and both 
define isomorphic line bundles on the quotient stacks under the identification 
in {\rm(\ref{eq:isom-P_n-GITP1})}.
\end{prop}

\begin{cor}\label{cor:reprPn-P1n-GIT}
{\rm (a)} For $\theta=(\eta_1,\eta_2,\theta_1,\ldots,\theta_n)\in H(P_n)$ such that 
$\theta_i>0$ for all $i$ a free representation $V=(s_1,\ldots,s_n)$ of $P_n$ is 
$\theta$-(semi)stable if and only if $s_i\cap 0=\emptyset$ for all $i$ and its 
image under {\rm(\ref{eq:isom-Q_n-GITP1})} is a (semi)stable section of $(\P^1)^n$
with respect to $\O_{(\P^1)^n}(\theta_1,\ldots,\theta_n)$ with the corresponding 
$\G_m$-linearisation.\\
{\rm (b)} The image of $\interior C(P_n)\subset H(P_n)$ under the isomorphism 
$\Pic^{\G_m}((\P^1)^n)_\Q\cong H(P_n)$ is the $\G_m$-ample cone in 
$\Pic^{\G_m}((\P^1)^n)_\Q$. The GIT classes in these open cones coincide.\\
{\rm (c)} For $\theta\in H(P_n)\cong\Pic^{\G_m}((\P^1)^n)_\Q$ such that 
$\theta_i>0$ for all $i$ the isomorphism {\rm(\ref{eq:isom-Q_n-GITP1})} induces 
isomorphisms of the stacks of $\theta$-(semi)stable points, of their moduli spaces 
and their inverse systems.
\end{cor}

In the following, when considering free representations $V=(s_1,\ldots,s_n)$
of $Q_n$ and $P_n$ such that $s_i\cap 0=\emptyset$ up to isomorphism,
we will often write them as tuples of sections of $(\P^1)^n$.

\subsection{The functor of inverse limits of quiver varieties for $P_n$ and $Q_n$}
\label{subsec:functor-limPnQn}

Let $Q\!=\!Q_n$ or $Q\!=\!P_n$. Because all chambers are connected via the 
relative interiors of GIT equivalence classes of codimension one, we can rewrite 
the contravariant functor on $S$-schemes (\ref{eq:functorinvlim}) in 
proposition \ref{prop:functorinvlim} as
\begin{equation}\label{eq:functorinvlim-QnPn}
Y\mapsto
\left\{\left.
(\phi_{V^\theta})_\theta\in\!\!\!\!\!\prod_{\textit{$\theta$ generic}}\!\!\!\!\!
\mathcal M^\theta(Q)(Y)
\;\right|\!
\begin{array}{l}
\forall\,\theta,\theta'\,\textit{generic}\;\;
\forall\:C_\tau\subset\interior C(Q)\;\textit{GIT}\\
\textit{equiv.\ class of codimension one such that}\\
\overline{C}_\tau\subseteq\overline{C}_\theta\cap\overline{C}_{\theta'}\colon
\;\phi_{\theta,\tau}\circ\phi_{V^{\theta}}=\phi_{\theta',\tau}\circ\phi_{V^{\theta'}}
\end{array}\!
\right\}
\end{equation}
where the product is over representatives $\theta$ of the generic GIT equivalence classes
and $\phi_{V^\theta}\colon Y\to\mathcal M^\theta(Q)$ is the morphism determined by a
$\theta$-stable representation $V^\theta$ over $Y$.

\medskip

For a representation $V$ of $Q$ over an $S$-scheme $Y$ stable with respect 
to some generic $\theta$ the polytopes  $\Theta(V(y))$ for the fibres $V(y)$ 
over the geometric points $y$ of $Y$ defined in subsection \ref{subsec:weightspacePnQn} 
have the property that $\{y'\in Y\:|\:\Theta(V(y))\subseteq \Theta(V(y'))\}$ are 
the points of an open subscheme of $Y$. Therefore for two representations $V,V'$ 
we have an open subscheme
\[
U(V,V')\;=\;\big\{\,y\in Y\;\big|\:\interior\big(\Theta(V(y))\cap\Theta(V'(y))\big)
\neq\emptyset\,\big\}\;\subseteq \;Y.
\]
If $V^\theta,V^{\theta'}\!$ are part of a family over $Y$ satisfying the conditions
in (\ref{eq:functorinvlim-QnPn}), then 
\[
U(V^\theta,V^{\theta'})
\;=\;\big\{\,y\in Y\;\big|\:\textit{$V^\theta(y)$ $\theta'$-stable}\,\big\}
\;=\;\big\{\,y\in Y\;\big|\:\textit{$V^{\theta'}\!(y)$ $\theta$-stable}\,\big\}
\] 
and $V^\theta|_{U(V^\theta,V^{\theta'})}\cong V^{\theta'}|_{U(V^\theta,V^{\theta'})}$
by the conditions coming from the walls meeting the interior of the polytopes
$\Theta(V^\theta(y))$, $\Theta(V^{\theta'}\!(y))$.

\medskip

To study the conditions coming from walls that separate polytopes $\Theta(V(y))$ 
we consider the schemes of representations of $Q$ for weights in a GIT equivalence
class $C_\tau$ of codimension one with adjacent chambers $C_\theta,C_{\theta'}$. 
For $Q=Q_n$, $C_\tau\cap\Delta(2,n)$ is contained in a wall of the form 
$W_{\{J,J^\complement\}}$ for some $J\subset\{1,\ldots,n\}$, $2\leq|J|\leq n-2$. 
For $Q=P_n$, $C_\tau\cap\Delta^1\times\Delta^{n-1}$ is contained in a 
wall of the form $W_J$ for some $J\subset\{1,\ldots,n\}$, $1\leq|J|\leq n-1$. 
In case $Q=Q_n$ let $i\in J^\complement$, $j\in J$, in case $Q=P_n$ let $i=0,j=\infty$.
Let $Z_\tau\subset R^\tau(Q)$ be the closed subscheme of strictly 
$\tau$-semistable points and $V^\tau$ its tautological representation
considered as tuple of sections $(s^\tau_1,\ldots,s^\tau_n)$ of $\P^1_{Z_\tau}\to Z_\tau$.
The geometric fibres of $Z_\tau$ over $S$ consist of the following three strata: 
the subset where both conditions $s^\tau_l=s^\tau_j\Leftrightarrow l\in J$ and 
$s^\tau_k=s^\tau_i\Leftrightarrow k\in J^\complement$ 
are satisfied, and the two subsets $Z_{\tau,J},Z_{\tau,J^\complement}$ where only 
the first (resp.\ the second) condition is satisfied. 
In case $Q=Q_n$ assume $\sum_{j\in J}\theta_j>1$, then 
$\sum_{i\in J^\complement}\theta'_i>1$ and $Z_{\tau,J}=Z_{\tau}\cap R^{\theta'}(Q)$, 
$Z_{\tau,J^\complement}=Z_{\tau}\cap R^\theta(Q)$.
The image of $Z_{\tau,J}$ forms a projective space 
$\P^{|J^\complement|-2}_S\subset\mathcal M^{\theta'}(Q)$, the image of 
$Z_{\tau,J^\complement}$ forms a projective space 
$\P^{|J|-2}_S\subset\mathcal M^\theta(Q)$ and the image of $Z_\tau$ forms 
a subscheme $Z^\tau\subset\mathcal M^\tau(Q)$ isomorphic to the base scheme $S$. 
The morphisms
$\mathcal M^\theta(Q)\rightarrow\mathcal M^\tau(Q)\leftarrow\mathcal M^{\theta'}(Q)$
restrict to morphisms $\P^{|J|-2}_S\rightarrow Z^\tau\leftarrow\P^{|J^\complement|-2}_S$
and are isomorphisms elsewhere.
The case $Q=P_n$ is similar with projective spaces
$\P^{|J|-1}_S\subset\mathcal M^\theta(Q)$, 
$\P^{|J^\complement|-1}_S\subset\mathcal M^{\theta'}(Q)$. 

\medskip

Coordinate functions on $R^*(Q)$ are given by the tautological family on $R^*(Q)$, 
considered as sections 
$s=(s_1\!=\!(s_{1,0}\!:\!s_{1,1}),\ldots,s_n\!=\!(s_{n,0}\!:\!s_{n,1}))$ 
of $\P^1_{R^*(Q)}\to R^*(Q)$, as follows. 
In case $Q=Q_n$ for $i,j\in\{1,\ldots,n\}$ over the invariant open subscheme 
$\{y\:|\:s_i(y)\neq s_j(y)\}\subset R^*(Q)$ there is the section
$\left(\begin{smallmatrix} s_{i,1}&-s_{i,0}\\ -s_{j,1}&s_{j,0}\end{smallmatrix}\right)$
in $\PGL(2)$ that transforms the tautological family $s$ to the family $\tilde{s}$ 
with $\tilde{s}_i=(0:1),\tilde{s}_j=(1:0)$.
In case $Q=P_n$ we have the additional sections $s_i=s_0=(0:1)$, $s_j=s_\infty=(1:0)$ 
and we set $\tilde{s}=s$.
The invariant open subscheme 
\[
U_\tau=\big\{\,y\:\big|\:s_i(y)\neq s_j(y),\:
\forall\,k\!\in\!J^\complement\colon s_k(y)\neq s_j(y),\:
\forall\,l\!\in\!J\colon s_l(y)\neq s_i(y)\big\}\subset R^*(Q)
\]
contains $Z_\tau$. The algebra of invariant functions on $U_\tau$ corresponds to 
the subalgebra of torus invariant functions in the $\O_S$-algebra generated by 
$\frac{\tilde{s}_{l,1}}{\tilde{s}_{l,0}}$ for $l\in J$ and
$\frac{\tilde{s}_{k,0}}{\tilde{s}_{k,1}}$ for $k\in J^\complement$. 
This $\O_S$-algebra of torus invariants is generated by 
$f^{i,j}_{k,l}=\frac{\tilde{s}_{k,0}\tilde{s}_{l,1}}{\tilde{s}_{k,1}\tilde{s}_{l,0}}$
for $l\in J$, $k\in J^\complement$.
The function $f^{i,j}_{k,l}$ can be written in terms of $s$ as
\[
f^{i,j}_{k,l}=\frac{(s_{j,0}s_{l,1}-s_{j,1}s_{l,0})(s_{i,1}s_{k,0}-s_{i,0}s_{k,1})}
{(s_{i,1}s_{l,0}-s_{i,0}s_{l,1})(s_{j,0}s_{k,1}-s_{j,1}s_{k,0})}.
\]
These invariant functions are sometimes called the cross-ratios of four sections.
Since invariant regular functions on the representation space correspond to regular 
functions on the moduli spaces we obtain the following result.

\begin{lemma}\label{le:fijlk}
{\rm(a)} The invariant regular functions $f^{i,j}_{k,l}$ on $U_\tau$ define 
regular functions on the image of $U_\tau\cap R^\tau(Q)$ (resp.\ $U_\tau\cap R^\theta(Q)$,
$U_\tau\cap R^{\theta'}(Q)$) in  $\mathcal M^\tau(Q)$ (resp.\ $\mathcal M^\theta(Q)$, 
$\mathcal M^{\theta'}(Q)$) which we also denote $f^{i,j}_{k,l}$.
These satisfy $\phi_{\theta,\tau}^*(f^{i,j}_{k,l})=f^{i,j}_{k,l}$,
$\phi_{\theta',\tau}^*(f^{i,j}_{k,l})=f^{i,j}_{k,l}$.\\
{\rm(b)} For $z\in Z^\tau$ the local ring $\O_{\mathcal M^\tau(Q),z}$ is the 
localisation of an $\O_S$-algebra generated by the functions $f^{i,j}_{k,l}$ for 
$l\in J$, $k\in J^\complement$.
\end{lemma}

We express the conditions in \rm(\ref{eq:functorinvlim-QnPn}) for a family 
$(V^\theta)_\theta$ locally in terms of equations.

\begin{lemma}\label{le:eq-invlimit}
Let $V^\theta,V^{\theta'}$ be free representations of $Q=Q_n$ or $Q=P_n$ over 
an $S$-scheme $Y$ stable with respect to generic $\theta,\theta'$. 
Assume $V^\theta,V^{\theta'}$ are part of a collection as in 
{\rm(\ref{eq:functorinvlim-QnPn})}. 
We consider $V^\theta,V^{\theta'}$ as tuples of sections 
$(s^\theta_i)_i=(s^\theta_{i,0}:s^\theta_{i,1})_i$, 
$(s^{\theta'}_i)_i=(s^{\theta'}_{i,0}:s^{\theta'}_{i,1})_i$ of $\P^1_Y$, 
in case $Q=P_n$ we add the sections $s^\theta_0,s^{\theta'}_0=(0:1)$,
$s^\theta_\infty,s^{\theta'}_\infty=(1:0)$.
In case $Q=Q_n$ let $i,j\in\{1,\ldots,n\}$, in case $Q=P_n$ let $i=0,j=\infty$.
Let $U_{i,j}=\{y\:|\:s^\theta_i(y)\neq s^\theta_j(y),
s^{\theta'}_i(y)\neq s^{\theta'}_j(y)\}\subseteq Y$.
We choose homogeneous coordinates $x^\theta_0,x^\theta_1$ and 
$x^{\theta'}_0,x^{\theta'}_1$ of $\P^1_{U_{i,j}}$ 
such that $s^\theta_i=(0:1),s^\theta_j=(1:0)$ with respect to $x^\theta_0,x^\theta_1$
and $s^{\theta'}_i=(0:1), s^{\theta'}_j=(1:0)$ with respect to $x^{\theta'}_0,x^{\theta'}_1$. 
Then $V^\theta$ and $V^{\theta'}$ are over $U_{i,j}$ related by the equations
\begin{equation}\label{eq:theta-theta'}
s^\theta_{k,0}s^\theta_{l,1}s^{\theta'}_{k,1}s^{\theta'}_{l,0}
=s^\theta_{k,1}s^\theta_{l,0}s^{\theta'}_{k,0}s^{\theta'}_{l,1}
\end{equation}
for all $k,l\in\{1,\ldots,n\}$.
\end{lemma}
\begin{proof}
Over $U(V^\theta,V^{\theta'})$ the equations hold because
$V^\theta|_{U(V^\theta,V^{\theta'})}\cong V^{\theta'}|_{U(V^\theta,V^{\theta'})}$.

Consider the local situation around a point 
$y\in U_{i,j}\setminus U(V^\theta,V^{\theta'})$. 
Assume first that the polytopes $\Theta(V^\theta(y))$ and $\Theta(V^{\theta'}(y))$ 
meet in an inner wall $W$. 
In case $Q=P_n$ let $W=W_J$ and we can assume that 
$s^\theta_h(y)=s^\theta_j(y)=s^\theta_\infty(y)=(1:0)$ for $h\in J$ and 
$s^{\theta'}_h(y)=s^{\theta'}_i(y)=s^{\theta'}_0(y)=(0:1)$ for $h\in J^\complement$. 
In case $Q=Q_n$ let $W=W_{\{J,J^\complement\}}$ and we can assume 
$j\in J$, $i\in J^\complement$ and $s^\theta_h(y)=s^\theta_j(y)=(1:0)$ for $h\in J$, 
$s^{\theta'}_h(y)=s^{\theta'}_i(y)=(0:1)$ for $h\in J^\complement$.
For $k,l$ there are the cases: $k,l\in J$ (similar: $k,l\in J^\complement$)
or $k\in J,l\in J^\complement$ (similar: $k\in J^\complement,l\in J$).

Case $k\in J^\complement,l\in J$:
We have $s^\theta_{l,0}(y),s^{\theta'}_{l,0}(y)\neq0$ and 
$s^\theta_{k,1}(y),s^{\theta'}_{k,1}(y)\neq0$.
In a neighbourhood of $y$ equation (\ref{eq:theta-theta'}) is equivalent to the 
equation $f^{\theta,i,j}_{k,l}=f^{\theta',i,j}_{k,l}$ with 
$f^{\theta,i,j}_{k,l}=\frac{s^\theta_{k,0}s^\theta_{l,1}}{s^\theta_{k,1}s^\theta_{l,0}}$,
$f^{\theta',i,j}_{k,l}=\frac{s^{\theta'}_{k,0}s^{\theta'}_{l,1}}
{s^{\theta'}_{k,1}s^{\theta'}_{l,0}}$.
The regular functions $f^{\theta,i,j}_{k,l}$ (resp.\ $f^{\theta',i,j}_{k,l}$) 
are pullbacks of the regular functions $f^{i,j}_{k,l}$ on
$\mathcal M^\theta(Q)$ (resp.\ $\mathcal M^{\theta'}(Q)$) via 
$\phi_{V^\theta}$ (resp.\ $\phi_{V^{\theta'}}$). 
Because $\phi_{\theta,\tau}\circ\phi_{V^\theta}=\phi_{\theta',\tau}\circ\phi_{V^{\theta'}}$
and using lemma \ref{le:fijlk}.(a) it follows $f^{\theta,i,j}_{k,l}=f^{\theta',i,j}_{k,l}$.

Case $k,l\in J$:
We have $s^\theta_{k,0}(y),s^{\theta'}_{k,0}(y)\neq0$ and 
$s^\theta_{l,0}(y),s^{\theta'}_{l,0}(y)\neq0$.
In a neighbourhood of $y$ equation (\ref{eq:theta-theta'}) is equivalent to the equation
$\frac{s^\theta_{l,1}}{s^\theta_{l,0}}\frac{s^{\theta'}_{k,1}}{s^{\theta'}_{k,0}}
=\frac{s^\theta_{k,1}}{s^\theta_{k,0}}\frac{s^{\theta'}_{l,1}}{s^{\theta'}_{l,0}}$.
Multiplying with $\frac{s^\theta_{g,0}s^{\theta'}_{h,0}}{s^\theta_{g,1}s^{\theta'}_{h,1}}$
for $g\in J^\complement$ such that $s^\theta_g(y)\neq(0:1)$ and $h\in J$ such that 
$s^{\theta'}_h(y)\neq(1:0)$, we obtain the equivalent equation 
$f^{\theta,i,j}_{g,l}f^{\theta',i,j}_{h,k}=f^{\theta,i,j}_{g,k}f^{\theta',i,j}_{h,l}$.
This equation holds because by the first case we have 
$f^{\theta,i,j}_{g,l}=f^{\theta',i,j}_{g,l}$, $f^{\theta,i,j}_{g,k}=f^{\theta',i,j}_{g,k}$.

This shows that equations (\ref{eq:theta-theta'}) hold for $V^\theta,V^{\theta'}$
in a neighbourhood of $y$ if the polytopes $\Theta(V^\theta(y))$ and 
$\Theta(V^{\theta'}(y))$ have overlapping interiors or meet in an inner wall.

To show the general case we show that if the union of the sets of walls separating
$\Theta(V^\theta(y))$, $\Theta(V^{\theta'}(y))$ and 
$\Theta(V^{\theta'}(y))$, $\Theta(V^{\theta''}(y))$ are exactly the walls that separate 
$\Theta(V^\theta(y)),\Theta(V^{\theta''}(y))$ and if $V^\theta,V^{\theta'}$ and
$V^{\theta'},V^{\theta''}$ in a neighbourhood of $y$ are related by equations of 
the form (\ref{eq:theta-theta'}), then so are $V^\theta,V^{\theta''}$.
Assume that $\Theta(V^\theta(y))$ and $\Theta(V^{\theta'}(y))$ are separated by an
inner wall $W_J$ (resp.\ $W_{\{J,J^\complement\}}$) and that $\Theta(V^{\theta'}(y))$ 
and $\Theta(V^{\theta''}(y))$ are separated by an inner wall $W_{J'}$ (resp.\ 
$W_{\{J',(J')^\complement\}}$) such that $J\subset J'$, these walls are part 
of the boundary of $\Theta(V^{\theta'}(y))$ and both walls separate 
$\Theta(V^\theta(y))$ and $\Theta(V^{\theta''}(y))$. 
If $s_i^\theta(y)\neq s_j^\theta(y)$,  $s_i^{\theta''}(y)\neq s_j^{\theta''}(y)$
then  either $i\in J$, $j\in(J')^\complement$ or $j\in J$, $i\in(J')^\complement$ 
and thus $s_i^{\theta'}(y)\neq s_j^{\theta'}(y)$. 
We assume $j\in J$, $i\in(J')^\complement$.

Case $k\in J',l\in J^\complement$ (similar: $k\in (J')^\complement,l\in J$).
It is $s^{\theta'}_k(y)\neq (0:1)$, $s^{\theta'}_l(y)\neq(1:0)$ and in 
a neighbourhood of $y$ we have  
\[
s^\theta_{k,0}s^\theta_{l,1}
s^{\theta''}_{k,1}s^{\theta''}_{l,0}s^{\theta'}_{k,0}s^{\theta'}_{l,1}
=s^\theta_{k,0}s^\theta_{l,1}
s^{\theta'}_{k,1}s^{\theta'}_{l,0}s^{\theta''}_{k,0}s^{\theta''}_{l,1}
=s^{\theta'}_{k,0}s^{\theta'}_{l,1}s^\theta_{k,1}s^\theta_{l,0}
s^{\theta''}_{k,0}s^{\theta''}_{l,1},
\]
thus 
$s^\theta_{k,0}s^\theta_{l,1}s^{\theta''}_{k,1}s^{\theta''}_{l,0}
=s^\theta_{k,1}s^\theta_{l,0}s^{\theta''}_{k,0}s^{\theta''}_{l,1}$.

Case $k\in J',l\in J$ (similar: $k\in(J')^\complement, l\in J^\complement$).
Choose $h\in J^\complement\cap J'$.
Then $s^{\theta'}_k(y),s^{\theta'}_l(y)\neq (0:1)$ and
$s^\theta_h(y)=(0:1),s^{\theta''}_h(y)=(1:0)$, 
$s^{\theta'}_h(y)\neq (1:0),(0:1)$.
In a neighbourhood of $y$ we have
\[
\begin{array}{l}
s^\theta_{k,0}s^\theta_{l,1}s^{\theta''}_{k,1}s^{\theta''}_{l,0}
s^{\theta'}_{k,0}s^{\theta'}_{l,0}
s^{\theta'}_{h,1}s^\theta_{h,1}s^{\theta''}_{h,0}
=s^\theta_{k,0}s^\theta_{l,1}s^{\theta'}_{k,1}s^{\theta''}_{l,0}
s^{\theta''}_{k,0}s^{\theta'}_{l,0}
s^{\theta''}_{h,1}s^\theta_{h,1}s^{\theta'}_{h,0}\\
=s^{\theta'}_{k,0}s^\theta_{l,1}s^\theta_{k,1}s^{\theta''}_{l,0}
s^{\theta''}_{k,0}s^{\theta'}_{l,0}
s^{\theta''}_{h,1}s^{\theta'}_{h,1}s^\theta_{h,0}
=s^{\theta'}_{k,0}s^{\theta'}_{l,1}s^\theta_{k,1}s^{\theta''}_{l,0}
s^{\theta''}_{k,0}s^\theta_{l,0}
s^{\theta''}_{h,1}s^\theta_{h,1}s^{\theta'}_{h,0}\\
=s^{\theta'}_{k,0}s^{\theta''}_{l,1}s^\theta_{k,1}s^{\theta'}_{l,0}
s^{\theta''}_{k,0}s^\theta_{l,0}
s^{\theta'}_{h,1}s^\theta_{h,1}s^{\theta''}_{h,0}
\end{array}
\]
thus 
$s^\theta_{k,0}s^\theta_{l,1}s^{\theta''}_{k,1}s^{\theta''}_{l,0}
=s^\theta_{k,1}s^\theta_{l,0}s^{\theta''}_{k,0}s^{\theta''}_{l,1}$.
\end{proof}

\begin{prop}\label{prop:functorinvlimQnPn}
For $Q=Q_n$ or $Q=P_n$ the functor of the inverse limit of moduli spaces
of representations is isomorphic to the functor $\Lim(Q)$ defined by
\begin{equation}\label{eq:functorinvlimQnPn}
Y\;\mapsto\;
\Bigg\{
(\phi_{V^\theta})_\theta\,\in\!\!\!\prod_{\textit{$\theta$ generic}}\!\!\!
\mathcal M^\theta(Q)(Y)
\;\;\Bigg|\;
\textit{locally equations {\rm(\ref{eq:theta-theta'})} hold for $(V^\theta)_\theta$}\;
\Bigg\}
\end{equation}
where the product is over representatives $\theta$ of the generic GIT equivalence classes
and $\phi_{V^\theta}\colon Y\to\mathcal M^\theta(Q)$ is the morphism determined by a
$\theta$-stable representation $V^\theta$ over $Y$.
\end{prop}
\begin{proof}
By lemma \ref{le:eq-invlimit} for a family $(V^\theta)_\theta$ 
satisfying the conditions of (\ref{eq:functorinvlim-QnPn}) the equations 
(\ref{eq:theta-theta'}) hold locally. We show the opposite implication.

Let $(V^\theta)_\theta$ be a family of representations over an $S$-scheme $Y$
as in (\ref{eq:functorinvlimQnPn}). Let $y\in Y$. 
Let $C_\tau$ be an equivalence class of codimension $1$ in the interior  
and $\theta,\theta'$ generic such that
$\overline{C}_\tau=\overline{C}_\theta\cap\overline{C}_{\theta'}$.
We can consider $V^\theta,V^{\theta'}$ in a neighbourhood $U$ of $y$ as tuples of 
sections $(s^\theta_i)_i,(s^{\theta'}_i)_i$ of $\P^1_U\to U$.
In case $Q=Q_n$ let $i,j\in\{1,\ldots,n\}$ such that $s^\theta_i(y)\neq s^\theta_j(y)$ 
and $s^{\theta'}_i(y)\neq s^{\theta'}_j(y)$ and choose coordinates such that 
$s^\theta_i,s^{\theta'}_i=(0:1)$ and $s^\theta_j,s^{\theta'}_j=(1:0)$. 
In case $Q=P_n$ let $i=0,j=\infty$, then $s^\theta_i=s^\theta_0=(0:1),
s^\theta_j=s^\theta_\infty=(1:0)$ and the same for $\theta'$.

If there exists $k\in\{1,\ldots,n\}$ such that $s^\theta_k(y)\neq(0:1),(1:0)$,
$s^{\theta'}_k(y)\neq(0:1),(1:0)$ then the equations 
$s^\theta_{k,0}s^\theta_{l,1}s^{\theta'}_{k,1}s^{\theta'}_{l,0}
=s^\theta_{k,1}s^\theta_{l,0}s^{\theta'}_{k,0}s^{\theta'}_{l,1}$
show that in a neighbourhood of $y$ $V^\theta\cong V^{\theta'}$ and thus
$\phi_{\theta,\tau}\circ\phi_{V^\theta}=\phi_{\theta',\tau}\circ\phi_{V^{\theta'}}$.

Otherwise there is $J\subset\{1,\ldots,n\}$, where we can assume that $j\in J$,
$i\in J^\complement$, such that 
$\forall\,k\in J^\complement\colon s^\theta_k(y)=(0:1),s^{\theta'}_k(y)\neq(1:0)$ 
and $\forall\,l\in J\colon s^\theta_l(y)\neq(0:1),s^{\theta'}_l(y)=(1:0)$.
Let $x\in S$ be the image of $y\in Y$ and $z\in Z^\tau$ the point over $x$, then 
$\phi_{\theta,\tau}\circ\phi_{V^\theta}(y)=z=\phi_{\theta',\tau}
\circ\phi_{V^{\theta'}}(y)$.
As equations (\ref{eq:theta-theta'}) hold, the functions
$f^{\theta,i,j}_{k,l}=\frac{s^\theta_{k,0}s^\theta_{l,1}}{s^\theta_{k,1}s^\theta_{l,0}}$ 
and $f^{\theta'\!,i,j}_{k,l}=\frac{s^{\theta'}_{k,0}s^{\theta'}_{l,1}}
{s^{\theta'}_{k,1}s^{\theta'}_{l,0}}$ coincide. We may consider these functions 
as elements of the local ring $\O_{Y,y}$.
Using lemma \ref{le:fijlk}, these are the pullbacks of the functions $f^{i,j}_{k,l}$ 
in $\O_{\mathcal M^\tau(Q),z}$ under $\phi_{\theta,\tau}\circ\phi_{V^\theta}$ and
$\phi_{\theta',\tau}\circ\phi_{V^{\theta'}}$, and it follows that the homomorphisms 
of local rings $\O_{\mathcal M^\tau(Q),z}\to\O_{Y,y}$ agree.

Since $\phi_{\theta,\tau}\circ\phi_{V^\theta}$ and 
$\phi_{\theta',\tau}\circ\phi_{V^{\theta'}}$ coincide as maps of sets of points and 
in each point the induced homomorphisms of local rings coincide, it follows
$\phi_{\theta,\tau}\circ\phi_{V^\theta}=\phi_{\theta',\tau}\circ\phi_{V^{\theta'}}$.
\end{proof}

Let $(V^\theta)_\theta$ be a family of representations over $Y$ satisfying the 
conditions of (\ref{eq:functorinvlim-QnPn}) or equivalently (\ref{eq:functorinvlimQnPn}). 
For $V^\theta, V^{\theta'}$ we have the natural structure of a closed subscheme 
$Z(V^\theta,V^{\theta'})\subseteq Y$ supported on the closed subset 
$Y\setminus U(V^\theta,V^{\theta'})$. 
Under the assumptions and notations of lemma \ref{le:eq-invlimit}
let $y\in U_{i,j}$ and assume $k,l\in\{1,\ldots,n\}$ such that
$s^\theta_k\neq(0:1),(1:0)$, $s^{\theta'}_k\neq(0:1)$
and $s^{\theta'}_l\neq(0:1),(1:0)$, $s^\theta_l\neq(1:0)$.
Because equations (\ref{eq:theta-theta'}) hold in a neighbourhood of $y$, the 
scheme defined by the equation $s^{\theta'}_k=(1:0)$ coincides with the scheme 
defined by the equation $s^\theta_l=(0:1)$ and does not depend on the choice 
of $k,l$.
The scheme $Z(V^\theta,V^{\theta'})$ can be defined in a neighbourhood of $y$
by each of these equations.

\medskip

The following proposition gives a geometric interpretation of the
condition in (\ref{eq:functorinvlimQnPn}) that two representations 
are related by equations (\ref{eq:theta-theta'}), 
and indicates the relation between inverse limits 
of moduli spaces of representations for the quivers $P_n,Q_n$ and
moduli spaces of chains and trees of $\P^1$'s with marked points.

\begin{prop}\label{prop:curveP1P1}
Let $V^\theta, V^{\theta'}$ be free representations of $Q=Q_n$ or $Q=P_n$ over 
an $S$-scheme $Y$ stable with respect to generic $\theta,\theta'$. 
We use the notations of lemma \ref{le:eq-invlimit} and assume that 
$V^\theta,V^{\theta'}$ over $U_{i,j}$ are related by equations {\rm(\ref{eq:theta-theta'})}.
Consider the equations
\begin{equation}\label{eq:curveP1P1}
s^\theta_{k,0}s^{\theta'}_{k,1}x^\theta_1x^{\theta'}_0
=s^\theta_{k,1}s^{\theta'}_{k,0}x^\theta_0x^{\theta'}_1
\end{equation}
over $U_{i,j}$.\\
{\rm(a)} Let $C_{i,j}\subset(\P^1)^2_{U_{i,j}}$ be the closed subscheme defined
by the equations {\rm(\ref{eq:curveP1P1})} for all $k$. 
For given $k$ let 
\[
U_{i,j,k}\;=\;\big\{\,y\in U_{i,j}\,\:\big|\:
\textit{$s^\theta_i(y),s^\theta_j(y),s^\theta_k(y)$ distinct \ or \  
$s^{\theta'}_i(y),s^{\theta'}_j(y),s^{\theta'}_k(y)$ distinct}\,\big\}.
\]
Then over $U_{i,j,k}$ the subscheme $C_{i,j,k}=C_{i,j}\times_{U_{i,j}}U_{i,j,k}
\subset(\P^1)^2_{U_{i,j,k}}$ is given by the single equation 
{\rm(\ref{eq:curveP1P1})} for this $k$.\\
{\rm(b)} The curves $C_{i,j}\subset(\P^1)^2_{U_{i,j}}$ glue (with the appropriate
coordinate changes) to a reduced curve $C\subset(\P^1)^2_Y$ flat over $Y$, 
which contains all pairs of sections $(s^\theta_l,s^{\theta'}_l)$, is isomorphic 
to $\,\P^1_{U(V^\theta,V^{\theta'})}$ over $U(V^\theta,V^{\theta'})\subseteq Y$ 
via its two projections, and degenerates exactly over the subscheme 
$Z(V^\theta,V^{\theta'})\subseteq Y$ to a chain of two $\,\P^1$'s intersecting 
transversally such that each projection defines an isomorphism on one component 
and contracts the other to a reduced point.\\
{\rm(c)} The curve $C$ induces morphisms $\,\P^1_Y\leftrightarrows\,\P^1_Y$ 
which degenerate to $\,\P^1\to\textit{pt.}$ over $Z(V^\theta,V^{\theta'})$ 
and which restrict to mutually inverse isomorphisms 
$\,\P^1_{U(V^\theta,V^{\theta'})}\!\leftrightarrow\P^1_{U(V^\theta,V^{\theta'})}$ 
that give rise to the isomorphism 
$V^\theta|_{U(V^\theta,V^{\theta'})}\cong V^{\theta'}|_{U(V^\theta,V^{\theta'})}$.
\end{prop}
\begin{proof}
(a) Assume that $s^\theta_k(y)\neq(0:1),(1:0)$. We can further assume 
$s^{\theta'}_k(y)\neq(1:0)$ (similar: $s^{\theta'}_k(y)\neq(0:1)$). 
Then in a neighbourhood of $y$ for all $l$ we have
\[\textstyle
s^\theta_{l,0}s^{\theta'}_{l,1}x^\theta_1x^{\theta'}_0
=s^\theta_{l,0}s^{\theta'}_{l,1}
\frac{s^\theta_{k,1}s^{\theta'}_{k,0}}{s^\theta_{k,0}s^{\theta'}_{k,1}}
x^\theta_0x^{\theta'}_1
=s^\theta_{l,1}s^{\theta'}_{l,0}x^\theta_0x^{\theta'}_1
\]
using equation (\ref{eq:curveP1P1}) for $k$ and equation (\ref{eq:theta-theta'}) 
for $k,l$.\\
(b) To show that the curves $C_{i,j}$ glue, it suffices to show that
the curves $C_{i,j}$ and $C_{i,j'}$ (and similar: $C_{i,j}$ and $C_{i',j}$)
coincide over $U_{i,j}\cap U_{i.j'}$ after the base changes\\[2mm]
\centerline{
$\left(\begin{smallmatrix}s^{\theta,i,j'}_{k,0}\\s^{\theta,i,j'}_{k,1}
\end{smallmatrix}\right)
=\left(\begin{smallmatrix}s^{\theta,i,j'}_{j,0}&\;0\\s^{\theta,i,j'}_{j,1}&\;a
\end{smallmatrix}\right)
\left(\begin{smallmatrix}s^{\theta,i,j}_{k,0}\\s^{\theta,i,j}_{k,1}
\end{smallmatrix}\right)\quad$
and
$\quad\left(\begin{smallmatrix}s^{\theta,i,j}_{k,0}\\s^{\theta,i,j}_{k,1}
\end{smallmatrix}\right)
=\left(\begin{smallmatrix}a&0\\[1mm]-s^{\theta,i,j'}_{j,1}&s^{\theta,i,j'}_{j,0}
\end{smallmatrix}\right)
\left(\begin{smallmatrix}s^{\theta,i,j'}_{k,0}\\s^{\theta,i,j'}_{k,1}
\end{smallmatrix}\right)$,
}\\[2mm]
where $a$ satisfies the relation 
$s^{\theta,i,j}_{j',0}s^{\theta,i,j'}_{j,1}+as^{\theta,i,j}_{j',1}=0$
and $s^\theta_k=(s^{\theta,i,j}_{k,0}:s^{\theta,i,j}_{k,1})$
with respect to coordinates $x_0^{\theta,i,j},x_1^{\theta,i,j}$
such that $(s^{\theta,i,j}_{i,0}:s^{\theta,i,j}_{i,1})=(0:1)$,
$(s^{\theta,i,j}_{j,0}:s^{\theta,i,j}_{j,1})=(1:0)$ (similar for $\theta'$, $j'$).
Applying these base changes one verifies that the equations 
(\ref{eq:curveP1P1}) hold for all $k$ with respect to $i,j$ 
if and only if they hold for all $k$ with respect to $i,j'$. 
The properties of $C$ follow from the properties of $C_{i,j,k}$ 
given by the single equation {\rm(\ref{eq:curveP1P1})} for $k$, 
which are easy to verify.\\
(c) follows from (b).
\end{proof}

\begin{rem}\label{rem:restr-eq-theta-theta'}
Proposition \ref{prop:curveP1P1} allows to restrict the sets of equations 
required to hold in functors like (\ref{eq:functorinvlimQnPn}) in 
proposition \ref{prop:functorinvlimQnPn}.
All equations (\ref{eq:theta-theta'}) relating $V^\theta,V^{\theta'}$ hold 
if the equations with respect to choices of $i,j$ hold such that the corresponding sets 
$U_{i,j}$ cover $Y$.
Also, with respect to fixed $i,j$, all equations relating $V^\theta,V^{\theta'}$ 
hold over $U_{i,j}$ if those for certain $k$ hold over $U_{i,j,k}$
and these sets cover $U_{i,j}$.
\end{rem}

\section{Losev-Manin and Grothendieck-Knudsen moduli spaces and root systems of type $A$}
\label{sec:Ln-M0n}

\subsection{Losev-Manin and Grothendieck-Knudsen moduli spaces}
\label{subsec:Ln-M0n}

The Losev-Manin moduli spaces of stable $n$-pointed chains of $\P^1$'s 
were introduced in \cite{LM00}.

\begin{defi}
A {\it stable $n$-pointed chain of $\,\P^1\!$} over an algebraically closed 
field is a tuple $(C,s_0,s_\infty,s_1,\ldots,s_n)$, where $(C,s_0,s_\infty)$ 
is a chain of $\,\P^1\!$, i.e.\ its irreducible components are isomorphic to 
projective lines with two distinct closed points $(\P^1,0,\infty)$ which
intersect transversally such that the point $0$ of one component 
meets the point $\infty$ of another component and the remaining points $0,\infty$
are denoted $s_0,s_\infty$, further $s_1,\ldots,s_n\in C$ are closed regular 
points different from $s_0,s_\infty$, and each component contains at least 
one of the marked points $s_1,\ldots,s_n$.
\end{defi}

\begin{defithm} {\rm(\cite{LM00}).}
Let $n\in\Z_{\geq1}$. The Losev-Manin moduli space $\overline{L}_n$ is 
the fine moduli space of stable $n$-pointed chains of $\,\P^1\!$, 
i.e.\ $\overline{L}_n$ represents the moduli functor (denoted by the same symbol)
\[ 
Y\;\mapsto\;
\Big\{\textit{isomorphism classes of stable $n$-pointed chains of $\,\P^1\!$ over $Y$}\Big\}
\]
where a stable $n$-pointed chain of $\,\P^1$ over a scheme $Y$ is a flat
proper morphism\linebreak
$C\to Y$ with sections $s_0,s_\infty,s_1,\ldots,s_n\colon Y\to C$ 
such that the geometric fibres \linebreak
$(C_y,s_0(y),s_\infty(y),s_1(y),\ldots,s_n(y))$ 
are stable $n$-pointed chains of $\,\P^1\!$ over algebraically\linebreak
closed fields.
\end{defithm}

$\overline{L}_n$ is toric, it compactifies the algebraic torus 
$L_n=(\G_m)^n/\,\G_m$, the moduli space of $n$ points in $\P^1\setminus\{0,\infty\}$.
It has been shown in \cite{LM00} that the moduli functor of 
stable $n$-pointed chains of $\P^1$ is represented by a projective scheme 
of relative dimension $n-1$, using an inductive construction of 
$\overline{L}_n$ together with the universal curve which is isomorphic to
$\overline{L}_{n+1}\to\overline{L}_n$.
This morphism $\overline{L}_{n+1}\to\overline{L}_n$, studied in a similar 
setting in \cite{Kn83}, is a special case of the following morphisms that 
arise by forgetting sets of sections, see \cite[Construction 3.15]{BB11}.

\begin{prop} 
Let $\emptyset\neq I\subseteq\{1,\ldots,n\}$. We write $\overline{L}_I$ 
for $\overline{L}_{|I|}$ with sections indexed by $I$. Then there is a morphism
\[
\gamma_I\colon\:\overline{L}_n\:\to\:\overline{L}_I
\]
such that a stable $n$-pointed chain is transformed to a stable $I$-pointed
chain by forgetting the sections $\{s_i\:|\:i\not\in I\}$ and contracting 
components which have become unstable.
\end{prop}

The Grothendieck-Knudsen moduli space $\overline{M}_{0,n}$ is the 
moduli space of stable $n$-pointed curves of genus $0$, i.e.\ of 
stable $n$-pointed trees of $\P^1$.
More generally, stable $n$-pointed curves of genus $g$ occur already 
in \cite[I.5]{SGA7(1)} and their moduli spaces and stacks 
have been systematically studied in \cite{Kn83}.

\begin{defi}
A {\it stable $n$-pointed curve of genus $0$} over an algebraically closed 
field is a tuple $(C,s_1,\ldots,s_n)$ where $C$ is a complete 
connected reduced curve $C$ of genus $0$ with at most ordinary double points, 
i.e.\ a tree of $\,\P^1\!$, $s_1,\ldots,s_n$ are closed points of $C$ 
such that each $s_i$ is a regular point of $C$, $s_i\neq s_j$ for $i\neq j$ 
and each component of $C$ has a least $3$ special points (i.e.\ singular points
and marked points $s_i$).
\end{defi}

\begin{defithm} {\rm(\cite{Kn83}).}
Let $n\in\Z_{\geq3}$. The Grothendieck-Knudsen moduli space $\overline{M}_{0,n}$ 
is the fine moduli space of stable $n$-pointed curves of genus $0$, 
i.e.\ $\overline{M}_{0,n}$ represents the moduli functor (denoted by the same symbol)
\[ 
Y\;\mapsto\;\Big\{\textit{isomorphism classes of stable $n$-pointed curves 
of genus $0$ over $Y$}\Big\}
\]
where a stable $n$-pointed curve of genus $0$ over a scheme $Y$ is a 
flat proper morphism $C\to Y$ with sections $s_1,\ldots,s_n\colon Y\to C$ 
such that the geometric fibres $(C_y,s_1(y),\ldots,s_n(y))$ are 
stable $n$-pointed curves of genus $0$. 
\end{defithm}

$\overline{M}_{0,n}$ compactifies the moduli space $M_{0,n}$ of $n$ distinct 
points in $\P^1$.
The fact that the moduli functor of stable $n$-pointed curves of genus $0$
is represented by a projective scheme of relative dimension $n-3$ has been shown
in \cite{Kn83} using an inductive argument on $n$, showing that the 
universal family over $\overline{M}_{0,n}$ is formed by the morphism
$\overline{M}_{0,n+1}\to\overline{M}_{0,n}$.
This morphism is a special case of the following morphisms for inclusions 
$I\subset\{1,\ldots,n\}$ which can be defined by mapping a stable $n$-pointed 
tree $(C\to Y,s_1,\ldots,s_n)$ to its image under the morphism defined by the
sheaf $\omega_{C/Y}(\sum_{i\in I}s_i)$, where $\omega_{C/Y}$ is the relative
dualising sheaf (see \cite{Kn83}).

\begin{prop} 
Let $I\subseteq\{1,\ldots,n\}$, $|I|\geq 3$. We write $\overline{M}_{0,I}$ 
for $\overline{M}_{0,|I|}$ with the sections of poined trees indexed by $I$. 
Then there is a morphism
\[
\gamma_I\colon\:\overline{M}_{0,n}\:\to\:\overline{M}_{0,I}
\]
such that a stable $n$-pointed tree is transformed to a stable $I$-pointed
tree by forgetting the sections $\{s_i\:|\:i\not\in I\}$ and contracting
components which have become unstable.
\end{prop}

\subsection{Embeddings into products of $\P^1\!$ and relation to root systems of type $A$}

We consider the cross-ratio varieties for root subsystems of type $A_3$ in $A_{n-1}$ 
defined and studied in \cite{Sek94}, \cite{Sek96}. It was already observed in 
\cite{Sek96} that the cross-ratio variety for root subsystems of type $A_3$ in $A_{n-1}$ 
should be isomorphic to the Grothendieck-Knudsen moduli space $\overline{M}_{0,n}$. 

\medskip

The root lattice of $A_n$ is the sublattice $M(A_{n-1})\subset L(A_{n-1})=\Z^n$ 
generated by the roots $\alpha_{i,j}=e_i-e_j$ where $i,j\in\{1,\ldots,n\}$ 
and $e_1,\ldots,e_n$ are the standard basis vectors of $\Z^n$. We have the 
lattice $L(A_{n-1})^*$ dual to $L(A_{n-1})$ and the lattice 
$N(A_{n-1})=L(A_{n-1})^*/(1,\ldots,1)\Z$ dual to $M(A_{n-1})$.

\medskip

We consider cross-ratio varieties for root systems of type $A_3$ in $A_{n-1}$.
Let $\A(L(A_{n-1}))$ be the $n$-dimensional affine space with coordinates 
$t_1,\ldots,t_n$ corresponding to the basis $e_1,\ldots,e_n$. 
We have the open subscheme
$U(A_{n-1})=\A(L(A_{n-1}))\setminus\bigcup_{i\neq j}\{t_i=t_j\}$.
For root subsystems $A_3\cong\Delta\subseteq A_{n-1}$ consisting of roots
$\alpha_{i,j}$ for $i,j\in I=\{i_1,\ldots,i_4\}\subseteq\{1,\ldots,n\}$, $|I|=4$
there is the morphism
\begin{equation}\label{eq:cAn1Delta}
c_{A_{n-1},\Delta}^{i_1,i_2,i_3,i_4}=
((t_{i_1}-t_{i_4})(t_{i_2}-t_{i_3}):(t_{i_1}-t_{i_3})(t_{i_2}-t_{i_4}))
\colon\;U(A_{n-1})\;\to\;\P^1
\end{equation}
depending on the choice of an ordering $I=\{i_1,i_2,i_3,i_4\}$ or equivalently 
a simple system $\alpha_{i_1,i_2},\alpha_{i_2,i_3},\alpha_{i_3,i_4}$ in 
$\Delta\cong A_3$.
The morphisms $c_{A_{n-1},\Delta}^{i_1,i_2,i_3,i_4}$ for different orderings 
$I=\{i_1,\ldots,i_4\}$ are related by isomorphisms of $\P^1$: for a permutation 
$\sigma$ there is an isomorphism $\phi_\sigma$ of $\P^1$ such that 
$c_{A_{n-1},\Delta}^{\sigma(i_1),\sigma(i_2),\sigma(i_3),\sigma(i_4)}=
\phi_\sigma\circ c_{A_{n-1},\Delta}^{i_1,i_2,i_3,i_4}$, where $\phi_\sigma$ 
for generators of the permutation group is given by the following matrices $M_\sigma$: 
\begin{equation}\label{eq:matrices-permutations}
M_{(i_1i_2)}=\left(\begin{smallmatrix}0&1\\1&0\end{smallmatrix}\right),\qquad
M_{(i_1i_3)}=\left(\begin{smallmatrix}-1&1\\0&1\end{smallmatrix}\right),\qquad
M_{(i_3i_4)}=\left(\begin{smallmatrix}0&1\\1&0\end{smallmatrix}\right)
\end{equation}

\begin{exa}\label{ex:xA3A3} $\overline{X}_{A_3,A_3}\cong\P^1$.\\
For each ordering $\{1,\ldots,4\}=\{i_1,i_2,i_3,i_4\}$ we have a copy of $\P^1$
with homogeneous coordinates $x^{i_1,i_2,i_3}_{i_4,0},x^{i_1,i_2,i_3}_{i_4,1}$
and the morphism $c_{A_3,A_3}^{i_1,i_2,i_3,i_4}\colon U(A_3)\to\P^1$ defined 
in (\ref{eq:cAn1Delta}). Let $X_{A_3,A_3}\subset\prod_{i\in S_4}\P^1$ 
the image of $c_{A_3,A_3}=\prod_{i\in S_4}c_{A_3,A_3}^{i_1,i_2,i_3,i_4}$, 
then $X_{A_3,A_3}\cong\P^1\setminus\{(0:1),\linebreak(1:0)\}$ and its closure 
$\overline{X}_{A_3,A_3}\cong\P^1$ via projections to the factors.
Using (\ref{eq:matrices-permutations}), $\overline{X}_{A_3,A_3}$
can be described as the subscheme of $\prod_{i\in S_4}\P^1$ defined by  
equations of the form
\begin{equation}\label{eq:M04-x}
\begin{array}{l} 
x^{i_2,i_1,i_3}_{i_4,0}x^{i_1,i_2,i_3}_{i_4,0}
=x^{i_2,i_1,i_3}_{i_4,1}x^{i_1,i_2,i_3}_{i_4,1}\\
x^{i_3,i_2,i_1}_{i_4,0}x^{i_1,i_2,i_3}_{i_4,1}=
x^{i_3,i_2,i_1}_{i_4,1}\big(x^{i_1,i_2,i_3}_{i_4,1}-x^{i_1,i_2,i_3}_{i_4,0}\big)\\
x^{i_1,i_2,i_4}_{i_3,0}x^{i_1,i_2,i_3}_{i_4,0}
=x^{i_1,i_2,i_4}_{i_3,1}x^{i_1,i_2,i_3}_{i_4,1}\\
\end{array}
\end{equation}
There are the factorisations 
$c_{A_3,A_3}^{i_1,i_2,i_3,i_4}=\phi_{A_3}^{i_1,i_2,i_3,i_4}\circ c_{A_3,A_3}\colon 
U(A_3)\to\overline{X}_{A_3,A_3}\stackrel{\sim}{\to}\P^1$ where 
$\phi_{A_3}^{i_1,i_2,i_3,i_4}\colon\overline{X}_{A_3,A_3}\stackrel{\sim}{\to}\P^1$ 
is induced by the projection onto the corresponding factor.
\end{exa}

In general for root systems $\Delta\cong A_3$ in $A_{n-1}$ we have a morphism 
$c_{A_{n-1},\Delta}\colon U(A_{n-1})\to U(\Delta)
\stackrel{c_{\Delta,\Delta}}{\longrightarrow}\overline{X}_{\Delta,\Delta}$ 
such that
\[
c^{i_1,i_2,i_3,i_4}_{A_{n-1},\Delta}
=\phi_\Delta^{i_1,i_2,i_3,i_4}\circ c_{A_{n-1},\Delta}\colon
\;U(A_{n-1})\;\to\;U(\Delta)\;\to\;\overline{X}_{\Delta,\Delta}\;\stackrel{\sim}{\to}\;\P^1.
\]

\begin{defi}
We denote the image $X_{A_{n-1},A_3}=\big(\prod_{A_3\cong\Delta\subseteq A_{n-1}}
c_{A_{n-1},\Delta}\big)\big(U(A_{n-1})\big)$ and define the {\it cross-ratio variety} 
$\overline{X}_{A_{n-1},A_3}$ for root systems of type $A_3$ in $A_{n-1}$ as the 
closure
\[\textstyle
\overline{X}_{A_{n-1},A_3}\;=\;\overline{X_{A_{n-1},A_3}}
\;\subseteq\;\prod_{A_3\cong\Delta\subseteq A_{n-1}}\overline{X}_{\Delta,\Delta}
\;\cong\;\prod_{A_3\cong\Delta\subseteq A_{n-1}}\P^1.
\]
Let $\overline{c}_{A_{n-1},\Delta}\colon\overline{X}_{A_{n-1},A_3}\to
\overline{X}_{\Delta,\Delta}$ be the morphisms induced by the projection on the 
factors of the product.
\end{defi}

In a similar way we can construct a scheme from the set of root subsystems 
of type $A_1$ in $A_{n-1}$. We have the $n$-dimensional projective space 
$\P(L(A_{n-1}))$ with homogeneous coordinates $t_1,\ldots,t_n$ corresponding 
to the basis $e_1,\ldots,e_n$. The open subscheme 
$T(A_{n-1})=\P(L(A_{n-1}))\setminus\bigcup_i\{t_i=0\}\subset \P(L(A_{n-1}))$ 
is the algebraic torus with character lattice $M(A_{n-1})$. For root subsystems 
$A_1\cong\{\pm\alpha_{i,j}\}\subseteq A_{n-1}$ we have morphisms
\[
c_{A_{n-1},\{\pm\alpha_{i,j}\}}^{i,j}=(t_i:t_j)\colon\;T(A_{n-1})\;\to\;\P^1
\]
depending on the choice of an ordering of $\{i,j\}$ or equivalently on the choice 
of a simple system $\alpha_{i,j}$ in $\{\pm\alpha_{i,j}\}$.
The morphisms $c_{A_{n-1},\{\pm\alpha_{i,j}\}}^{i,j}$ and 
$c_{A_{n-1},\{\pm\alpha_{i,j}\}}^{j,i}$ for the two different orderings of
$\{i,j\}$ are related by the isomorphism of $\P^1$ given by the matrix
$\left(\begin{smallmatrix}0&1\\1&0\end{smallmatrix}\right)$.

\begin{exa}\label{ex:xA1A1} $\overline{X}_{A_1,A_1}\cong\P^1$.\\
In the case of $A_1=\{\pm\alpha_{i,j}\}$ for each odering of $\{i,j\}$ 
we have a copy of $\P^1$ with homogeneous coordinates $x^j_{i,0},x^j_{i,1}$ 
and a morphism $c_{A_1,A_1}^{i,j}\colon U(A_3)\to\P^1$. 
Let $X_{A_1,A_1}\subset\P^1\times\P^1$ be the image of 
$c_{A_1,A_1}=c_{A_1,A_1}^{i,j}\times c_{A_1,A_1}^{j,i}$, 
then $X_{A_1,A_1}\cong\P^1\setminus\{(0:1),(1:0)\}$ and its closure 
$\overline{X}_{A_1,A_1}\cong\P^1$ via projections to the factors.
$\overline{X}_{A_1,A_1}$ can be described as the subscheme of $\P^1\times\P^1$ 
defined by the equation $x^i_{j,0},x^j_{i,0}=x^i_{j,1}x^j_{i,1}$.
The morphism $c_{A_1,A_1}^{i,j}$ factors as 
$T(A_1)\stackrel{c_{A_1,A_1}}{\longrightarrow}\overline{X}_{A_1,A_1}
\stackrel{\sim}{\to}\P^1$, where the last morphism is induced by the projection 
onto the corresponding factor.
\end{exa}

For root subsystems $A_1\cong\{\pm\alpha_{i,j}\}\subseteq A_{n-1}$ 
the projection $T(A_{n-1})\to T(\{\pm\alpha_{i,j}\})$ composed with
$c_{\{\pm\alpha_{i,j}\},\{\pm\alpha_{i,j}\}}$ gives the morphism
\[
c_{A_{n-1},\{\pm\alpha_{i,j}\}}\colon\;T(A_{n-1})\;\to\;T(\{\pm\alpha_{i,j}\})
\to\overline{X}_{\{\pm\alpha_{i,j}\},\{\pm\alpha_{i,j}\}}
\]
that satisfies $c_{\{\pm\alpha_{i,j}\},\{\pm\alpha_{i,j}\}}^{i,j}\circ 
c_{A_{n-1},\{\pm\alpha_{i,j}\}}=c_{A_{n-1},\{\pm\alpha_{i,j}\}}^{i,j}
\colon T(A_{n-1})\to\P^1$. 
 
\medskip

The morphisms $c_{A_{n-1},\{\pm\alpha_{i,j}\}}\colon T(A_{n-1})\to
X_{\{\pm\alpha_{i,j}\},\{\pm\alpha_{i,j}\}}\cong\G_m$ are homomorphisms of algebraic tori.
Their product, the homomorphism of tori 
$\prod_{A_1\cong\{\pm\alpha_{i,j}\}\subseteq A_{n-1}}c_{A_{n-1},\{\pm\alpha_{i,j}\}}$, 
embeds $T(A_{n-1})$ into the open dense torus of
$\prod_{A_1\cong\{\pm\alpha_{i,j}\}\subseteq A_{n-1}}
\overline{X}_{\{\pm\alpha_{i,j}\},\{\pm\alpha_{i,j}\}}$, we denote its image 
$X_{A_{n-1},A_1}$. The closure
\[\textstyle
\overline{X}_{A_{n-1},A_1}\:=\;\overline{X_{A_{n-1},A_1}}
\;\subseteq\;\prod_{A_3\cong\{\pm\alpha_{i,j}\}\subseteq A_{n-1}}
\overline{X}_{\{\pm\alpha_{i,j}\},\{\pm\alpha_{i,j}\}}
\:\cong\;\prod_{A_1\cong\{\pm\alpha_{i,j}\}\subseteq A_{n-1}}\P^1
\]
coincides with the toric variety $X(A_{n-1})$ associated with the root system $A_{n-1}$, 
i.e.\ its fan is the fan of Weyl chambers of the root system $A_{n-1}$ in the 
lattice $N(A_{n-1})$, because on the open dense torus $T(A_{n-1})\subset X(A_{n-1})$ the 
morphisms $c_{A_{n-1},\{\pm\alpha_{i,j}\}}$ coincide with the morphisms 
\[
\gamma_{\{\pm\alpha_{ij}\}}\colon\;X(A_{n-1})\;\to\;X(A_1)\cong\P^1
\]
for inclusions of root subsystems $\{\pm\alpha_{ij}\}\subseteq A_{n-1}$ of 
type $A_1$ and the morphism\linebreak
$\prod_{\{\pm\alpha_{ij}\}}\gamma_{\{\pm\alpha_{ij}\}}$
is a closed embedding (see \cite{BB11} or theorem \ref{thm:Ln-ProdP1-functor} below).

\medskip

Already in \cite[(2.6.3)]{LM00} the Losev-Manin moduli space $\overline{L}_n$ 
has been identified as the smooth projective toric variety associated with 
the $(n-1)$-dimensional permutohedron, which coincides with $X(A_{n-1})$.
In \cite{BB11} an alternative proof that the functor $\overline{L}_n$ is 
representable and the fine moduli space $\overline{L}_n$ is isomorphic to 
the smooth projective toric variety $X(A_{n-1})$ is given after a systematic study of 
toric varieties $X(R)$ associated with root systems $R$, their functorial properties 
with respect to maps of root systems and their functor of points, and it is shown that 
the morphisms $\gamma_R\colon X(A_{n-1})\to X(R)\cong X(A_{k-1})$ corresponding to 
inclusions of root systems $A_{k-1}\cong R\subset A_{n-1}$ can be identified with the 
contraction morphisms $\gamma_I\colon\overline{L}_n\to\overline{L}_I\cong\overline{L}_{k}$ 
determined by subsets $I\subset\{1,\ldots,n\}$, $|I|=k$.

\begin{exa}\label{ex:L2} $\overline{L}_2\cong\P^1$.\\
Each $i\in\{1,2\}$ defines via the line bundle $\O_C(s_i)$ a contraction of 
any stable $2$-pointed chain $(C\to Y,s_0,s_\infty,s_1,s_2)$ over a scheme $Y$ 
to a trivial $\P^1$-bundle over $Y$ with three pairwise disjoint sections 
$s^i_0,s^i_\infty,s^i_i$ and a fourth section $s^i_j$.
We can choose homogeneous coordinates $x^i_{j,0},x^i_{j,1}$ such that 
$s^i_0=(0:1)$, $s^i_\infty=(1:0)$, $s^i_i=(1:1)$.
The remaining section $s^i_j=(s^i_{j,0}:s^i_{j,1})$ of $\P^1_Y$ determines an 
isomorphism $\overline{L}_2\cong\P^1$.
Since $(s^i_{j,0}:s^i_{j,1})=(s^j_{i,1}:s^j_{i,0})$, $\overline{M}_{0,4}$ can be 
embedded in $\P^1\times\P^1$ as the subscheme defined by the same equation 
$x^i_{j,0},x^j_{i,0}=x^i_{j,1}x^j_{i,1}$ as in example \ref{ex:xA1A1}, and the 
two projections induce the above isomorphisms $\overline{M}_{0,4}\cong\P^1$ 
corresponding to the two choices of an ordering of $\{1,2\}$.
\end{exa}

\begin{thm}\label{thm:Ln-ProdP1-functor} {\rm(\cite{BB11}).}
The morphism
\[
\prod\limits_{\{i,j\}}\!\gamma_{\{i,j\}}\colon\overline{L}_n\to
\prod\limits_{\{i,j\}}\overline{L}_{\{i,j\}}
\;\:\textit{or equivalently}
\prod\limits_{\{\pm\alpha_{i,j}\}}\!\!\!\!\gamma_{\{\pm\alpha_{i,j}\}}\colon
X(A_{n-1})\to\!\!\!
\prod\limits_{\{\pm\alpha_{i,j}\}}\!\!\!\!X(\{\pm\alpha_{i,j}\}),
\]
where the product is over all subsets $I=\{i,j\}\subset\{1,\ldots,n\}$,
$|I|=2$ or equivalently over all root subsystems 
$A_1\cong\{\pm\alpha_{i,j}\}\subseteq A_{n-1}$, is a closed embedding. 
Its image in $\prod_{\{i,j\}}\overline{L}_{\{i,j\}}
=\prod_{\{\pm\alpha_{i,j}\}}X({\{\pm\alpha_{i,j}\}})$ 
is given by the equations of the form
\[
x^j_{i,0}x^i_{k,0}x^j_{k,1}=x^j_{i,1}x^i_{k,1}x^j_{k,0}
\]
for all subsets $J=\{i,j,k\}\subseteq\{1,\ldots,n\}$, $|J|=3$
or equivalently all root subsystems
$A_2\cong\{\pm\alpha_{i,j},\pm\alpha_{j,k},\pm\alpha_{i,k}\}\subseteq A_{n-1}$,
where $x^i_{j,0},x^i_{j,1}$ are the homogeneous coordinates of 
$\overline{L}_{\{i,j\}}=X(\{\pm\alpha_{i,j}\})$ as defined in examples \ref{ex:xA1A1}
and \ref{ex:L2}. The functor of points of $\overline{L}_n=X(A_{n-1})$ is isomorphic 
to the contravariant functor
\begin{equation}\label{eq:Ln-functor} 
\textstyle
Y\;\mapsto\;
\left\{
\begin{array}{l}
\textit{families of sections}\\
(s^i_j\colon Y\to\P^1_Y)_{i,j\in\{1,\ldots,n\}}\\
\end{array}
\left|
\begin{array}{l}
\forall\,i\colon s^i_{i,0}=s^i_{i,1}\\
\forall\,i,j,k\colon s^j_{i,0}s^i_{k,0}s^j_{k,1}=s^j_{i,1}s^i_{k,1}s^j_{k,0}
\end{array}
\right.\right\}
\end{equation}
\end{thm}

Similarly, we can describe $\overline{M}_{0,n}$ and its embedding into 
a product of $\P^1$'s in terms of the root system $A_{n-1}$. First we consider the 
case of $\overline{M}_{0,4}$.

\begin{exa}\label{ex:M04} $\overline{M}_{0,4}\cong\P^1$.\\ 
We choose a permutation $I=\{1,\ldots,4\}=\{i_1,\ldots,i_4\}$.
Three sections $s_{i_1},s_{i_2},s_{i_3}$ define a contraction of any stable 
$4$-pointed curve $(C\to Y,s_{i_1},s_{i_2},s_{i_3},s_{i_4})$ to a trivial 
$\P^1$-bundle over $Y$ via the line bundle $\omega_{C/Y}(s_{i_1}+s_{i_2}+s_{i_3})$
(cf.\ \cite{Kn83}).
Choosing homogeneous coordinates $x^{i_1,i_2,i_3}_{i_4,0},x^{i_1,i_2,i_3}_{i_4,1}$
of $\P^1_Y$ such that $s_{i_1}=(0\!:\!1)$, $s_{i_2}=(1\!:\!0)$, $s_{i_3}=(1\!:\!1)$,
the remaining section $s_{i_4}=(s^{i_1,i_2,i_3}_{i_4,0}:s^{i_1,i_2,i_3}_{i_4,1})$ 
of $\P^1_Y$ yields an isomorphism $\overline{M}_{0,4}\cong\P^1$.
Under permutations of $I$ this fourth section is transformed by isomorphisms of $\P^1$ 
as in (\ref{eq:matrices-permutations}), so the sections in these coordinates are related by
\begin{equation}\label{eq:M04-s3}
\begin{array}{l} 
s^{i_2,i_1,i_3}_{i_4,0}s^{i_1,i_2,i_3}_{i_4,0}
=s^{i_2,i_1,i_3}_{i_4,1}s^{i_1,i_2,i_3}_{i_4,1}\\
s^{i_3,i_2,i_1}_{i_4,0}s^{i_1,i_2,i_3}_{i_4,1}=
s^{i_3,i_2,i_1}_{i_4,1}\big(s^{i_1,i_2,i_3}_{i_4,1}-s^{i_1,i_2,i_3}_{i_4,0}\big)\\
\end{array}
\end{equation}
for permutations of the upper indices and by
\begin{equation}\label{eq:M04-s4}
s^{i_1,i_2,i_3}_{i_4,0}s^{i_1,i_2,i_4}_{i_3,0}
=s^{i_1,i_2,i_3}_{i_4,1}s^{i_1,i_2,i_4}_{i_3,1}.\hspace{1.84cm}
\end{equation}
Therefore the moduli space $\overline{M}_{0,4}$ can be described as the subscheme in
$\prod_{i\in S_4}\P^1$ by equations of the form (\ref{eq:M04-x}), where each projection
induces the above isomorphism $\overline{M}_{0,4}\cong\P^1$ corresponding to a choice 
of an ordering of $I$.
\end{exa}

The space $U(A_{n-1})=\A(L(A_{n-1}))\setminus\bigcup_{i\neq j}\{t_i=t_j\}$
parametrises $n$ distinct points in $\A^1=\P^1\setminus\{(1:0)\}$, so its
tautological family determines a morphism $\psi\colon U(A_{n-1})\to M_{0,n}$.
A point $(a_1,a_2,a_3,a_4)\in U(A_3)$, considered as tuple
$((a_1:1),(a_2:1),(a_3:1),(a_4:1)))$ of points in $\P^1$, can be transformed by the
matrix $\left(\begin{smallmatrix}-(a_2-a_3)&\:a_1(a_2-a_3)\\-(a_1-a_3)&\:a_2(a_1-a_3)
\end{smallmatrix}\right)$ into the tuple $((0:1),(1:0),(1:1),
((a_1-a_4)(a_2-a_3):(a_2-a_4)(a_1-a_3))))$ (compare to (\ref{eq:cAn1Delta})). 
We obtain the following result:

\begin{lemma}\label{le:cDelta-gammaI}
Let $I=\{i_1,i_2,i_3,i_4\}\subseteq\{1,\ldots,n\}$, $|I|=4$ and 
$A_3\cong\Delta=\{\alpha_{i,j}\:|\:i,j\in I\}\linebreak\subseteq A_{n-1}$ 
the corresponding root subsystem. We identify 
$\overline{M}_{0,I}\cong\P^1$ as in example \ref{ex:M04} by 
$(C\to Y,(s_i)_i)\mapsto(s_{i_4}\colon Y\to\P^1_Y)$ after contracting a curve $C\to Y$ 
over $Y$ to $\,\P^1_Y$ and choosing a basis such that the sections 
$s_{i_1},s_{i_2},s_{i_3}$ become $(0:1),(1:0),(1:1)$, and denote the composition
$\gamma_I^{i_1,i_2,i_3,i_4}\colon\overline{M}_{0,n}\stackrel{\!\!\gamma_I\:}{\to}
\overline{M}_{0,I}\stackrel{\sim}{\to}\P^1$.
This relates to the morphism $c_{A_{n-1},\Delta}^{i_1,i_2,i_3,i_4}\colon 
U(A_{n-1})\to U(\Delta)\to\overline{X}_{\Delta,\Delta}\stackrel{\sim}{\to}\P^1$ 
of {\rm(\ref{eq:cAn1Delta})} as follows:
\[
c_{A_{n-1},\Delta}^{i_1,i_2,i_3,i_4}=\gamma_I^{i_1,i_2,i_3,i_4}\circ\psi.
\]
\end{lemma}

The following theorem is based on the closed embeddings of $\overline{M}_{0,n}$
and its universal curve in products of $\P^1$'s described in \cite[Prop.\ 4]{GHP88}. 
We use this to deduce a characterisation of the functor of points of $\overline{M}_{0,n}$, 
add the interpretation in terms of root systems and the link to the cross-ratio 
varieties using lemma \ref{le:cDelta-gammaI}.

\begin{thm}\label{thm:M0n-ProdP1-functor}
The morphism
\[
\prod\limits_{|I|=4}\!\gamma_I\colon\:\overline{M}_{0,n}\;\to\;
\prod\limits_{|I|=4}\overline{M}_{0,I},
\]
where the product is over all subsets $I\subseteq\{1,\ldots,n\}$,
$|I|=4$, is a closed embedding. This embedding induces an isomorphism 
$\overline{M}_{0,n}\to\overline{X}_{A_{n-1},A_3}$ and can be identified 
with the closed embedding
\[
\prod\limits_{A_3\cong\Delta\subseteq A_{n-1}\!\!\!\!\!\!\!}\!\!\!\!\!
\overline{c}_{A_{n-1},\Delta}\colon\;
\overline{X}_{A_{n-1},A_3}\;\:\to\!
\prod\limits_{A_3\cong\Delta\subseteq A_{n-1}\!\!\!\!\!\!\!}\!\!\!\!\!
\overline{X}_{\Delta,\Delta},
\]
where the product is over all root subsystems $\Delta$ isomorphic to $A_3$.
The image of $\overline{M}_{0,n}\cong\overline{X}_{A_{n-1},A_3}$ in 
$\prod_I\overline{M}_{0,I}\cong\prod_{\Delta}\overline{X}_{\Delta,\Delta}$ 
is given by the equations of the form
\[
x^{i_1,i_2,i_3}_{i_4,0}x^{i_1,i_2,i_3}_{i_5,1}x^{i_1,i_2,i_4}_{i_5,0}
=x^{i_1,i_2,i_3}_{i_4,1}x^{i_1,i_2,i_3}_{i_5,0}x^{i_1,i_2,i_4}_{i_5,1}
\]
for all subsets $J\subseteq\{1,\ldots,n\}$, $|J|=5$ 
or equivalently all root subsystems $A_4\cong\Gamma\subseteq A_{n-1}$,
where the homogeneous coordinates $x^{i_1,i_2,i_3}_{i_4,0},x^{i_1,i_2,i_3}_{i_4,1}$
are defined as in examples \ref{ex:xA3A3} and \ref{ex:M04}, and homogeneous coordinates
corresponding to different permutations of a set $I\subset\{1,\ldots,n\}$, $|I|=4$ 
for the same $\overline{M}_{0,I}=\overline{X}_{\Delta,\Delta}$ are related by equations 
{\rm(\ref{eq:M04-x})}.
The functor of points of $\overline{M}_{0,n}$ is isomorphic to the contravariant functor
\begin{equation}\label{eq:M0n-functor} 
Y\!\mapsto\!
\left\{\!\!
\begin{array}{l}
\textit{families of sections}\\
s^{i_1,i_2,i_3}_{i_4}\colon Y\to\P^1_Y\\
\textit{for $i_1,i_2,i_3,i_4\in\{1,\ldots,n\}$}\\
\textit{such that $|\{i_1,i_2,i_3\}|=3$}\\
\end{array}
\!\!\!\left|\!
\begin{array}{l}
\forall\,T=\{i_1,i_2,i_3\}, |T|=3\colon\\
s^{i_1,i_2,i_3}_{i_1}\!\!=\!(0\!:\!1),s^{i_1,i_2,i_3}_{i_2}\!\!=\!(1\!:\!0), 
s^{i_1,i_2,i_3}_{i_3}\!\!=\!(1\!:\!1)\\
\forall\,I=\{i_1,i_2,i_3,i_4\}, |I|=4\colon\\
\textit{equations {\rm(\ref{eq:M04-s3})} and {\rm(\ref{eq:M04-s4})} hold}\\
\forall\,J=\{i_1,i_2,i_3,i_4,i_5\}, |J|=5\colon\\
s^{i_1,i_2,i_3}_{i_4,0}s^{i_1,i_2,i_3}_{i_5,1}s^{i_1,i_2,i_4}_{i_5,0}
=s^{i_1,i_2,i_3}_{i_4,1}s^{i_1,i_2,i_3}_{i_5,0}s^{i_1,i_2,i_4}_{i_5,1}
\end{array}\!\!\!\!
\right.\right\}
\end{equation}
\end{thm}

\medskip

\begin{rem}\label{rem:M0n-functor} 
The stable $n$-pointed tree can be reconstructed from data as in (\ref{eq:M0n-functor})
over an $S$-scheme $Y$ as the curve $C\subseteq(\prod_T\P^1)_Y$, where the product is
over subsets $T\subseteq\{1,\ldots,n\}$ such that $|T|=3$, defined by the equations
\[
s^{i_1,i_2,i_3}_{i_4,0}x^{i_1,i_2,i_3}_1x^{i_1,i_2,i_4}_0
=s^{i_1,i_2,i_3}_{i_4,1}x^{i_1,i_2,i_3}_0x^{i_1,i_2,i_4}_1,
\] 
where $x^{i_1,i_2,i_3}_0,x^{i_1,i_2,i_3}_1$ are the homogeneous coordinates 
of the factor corresponding to\linebreak $T=\{i_1,i_2,i_3\}$ and related under 
permutations of elements of $T$ by equations as the first two in (\ref{eq:M04-x}) 
leaving out the lower index, and the sections are given as\linebreak
$s_k=(s^T_{k,0}:s^T_{k,1})_T\colon Y\to C\subseteq(\prod_T\P^1)_Y$,
where $(s^T_{k,0}:s^T_{k,1})=(s^{i_1,i_2,i_3}_{k,0}:s^{i_1,i_2,i_3}_{k,1})$ for 
$T=\{i_1,i_2,i_3\}$ with respect to the coordinates $x^{i_1,i_2,i_3}_0,x^{i_1,i_2,i_3}_1$.
This follows easily from theorem \ref{thm:M0n-ProdP1-functor} and the fact that
$\overline{M}_{0,n+1}\to\overline{M}_{0,n}$ forms the universal $n$-pointed tree over 
$\overline{M}_{0,n}$. One may also compare this to proposition \ref{prop:curveP1P1}.
\end{rem}

\subsection{Losev-Manin and Grothendieck-Knudsen moduli spaces as limits of moduli spaces of quiver representations}

\smallskip

We relate moduli spaces of representations of $P_n$ 
to the Losev-Manin moduli spaces $\overline{L}_n$.

\medskip

For $i\in\{1,\ldots,n\}$ we define the weight 
$\theta^i\in\Delta^1\times\Delta^{n-1}\subset H(P_n)$ by 
$\eta^i_1,\eta^i_2=-\frac{1}{2}$, $\theta^i_i=1-\varepsilon$ and 
$\theta^i_j=\frac{\varepsilon}{n-1}$ for $j\neq i$, where $1\gg\varepsilon>0$.
The weights $\theta^i$ are generic and any free representation $V^{\theta^i}$ that is 
$\theta^i$-stable has the property that, considered as tuple $(s^{\theta^i}_j)_j$
of sections of $\P^1$, $s^{\theta^i}_i$ does not meet $(0:1)$ and $(1:0)$. 
The moduli spaces $\mathcal M^{\theta^i}(P_n)$ are isomorphic to $(\P^1)^{n-1}$.

\begin{thm}\label{thm:Ln-limPn}
There is an isomorphism 
\[\overline{L}_n\;\,\cong\;\;\varprojlim_\theta\,\mathcal M^\theta(P_n).\]
\end{thm}
\begin{proof}
By proposition \ref{prop:functorinvlimQnPn} 
$\varprojlim_\theta\mathcal M^\theta(P_n)\cong\Lim(P_n)$.
Consider the functor $\Lim'(P_n)$ defined by
\[\textstyle
Y\;\mapsto\;
\left\{
(\phi_{V^{\theta^i}})_i\in\prod_i\mathcal M^{\theta^i}(P_n)(Y)
\left|\,
\forall\,i,j,k,l\colon
s^{\theta^i}_{k,0}s^{\theta^i}_{l,1}s^{\theta^j}_{k,1}s^{\theta^j}_{l,0}
=s^{\theta^i}_{k,1}s^{\theta^i}_{l,0}s^{\theta^j}_{k,0}s^{\theta^j}_{l,1}
\right.\right\}
\]
where $\phi_{V^{\theta^i}}\colon Y\to\mathcal M^\theta(Q)$ is the morphism determined by a
$\theta^i$-stable representation $V^{\theta^i}$ over $Y$, and the representations
$V^{\theta^i}$ are considered as tuples of sections $(s^{\theta^i}_k)_k$ of $\P^1_Y\to Y$.

We claim that the natural morphism of functors $\Lim(P_n)\to\Lim'(P_n)$
that arises by restricting elements and relations is an isomorphism. 
The property 
\[
\interior\big(\Theta(V^{\theta}(y))\cap \Theta(V^{\theta'}(y))\big)\neq\emptyset
\;\Longleftrightarrow\;V^\theta\cong V^{\theta'}\;\textit{in a neighbourhood of $y$}
\]
of $(V^\theta)_\theta\in\Lim(P_n)(Y)$ follows from the equations relating 
$V^\theta,V^{\theta'}$ and thus also holds for elements of $\Lim'(P_n)(Y)$:
by proposition \ref{prop:curveP1P1}.(c) we have an isomorphism $V^\theta\cong V^{\theta'}$ 
on an open set, and for $y$ in its complement there are $J,J'$ such that 
$(s^\theta_j(y))_{j\in J}$ coincide, $(s^{\theta'}_j(y))_{j\in J'}$ coincide
and $(J\cup J')^\complement=\emptyset$, which implies that 
$\Theta(V^{\theta}(y))$ and $\Theta(V^{\theta'}(y))$ are separated by a wall.
By this property for elements of $\Lim(P_n)(Y)$ (resp.\ $\Lim'(P_n)(Y)$) the polytopes
$\Theta(V^\theta(y))$ (resp.\ $\Theta(V^{\theta^i}(y))$) have either disjoint 
interiors or coincide.
We construct the inverse morphism by showing that there is a unique 
way to extend an element $(V^{\theta^i})_i\in\Lim'(P_n)(Y)$ to an element 
$(V^\theta)_\theta\in\Lim(P_n)(Y)$. Such a $(V^\theta)_\theta\in\Lim(P_n)(Y)$ 
is uniquely determined by its restriction to $\Lim'(P_n)(Y)$
because for any generic $\theta$ we have 
$\bigcup_i\{y\:|\:\textit{$V^\theta(y)$ $\theta^i$-stable}\}=Y$.
On the other hand we can extend a given $(V^{\theta^i})_i\in\Lim'(P_n)(Y)$
to an element of $\Lim(P_n)(Y)$ if for any generic $\theta$ we have 
$\bigcup_i\{y\:|\:\textit{$V^{\theta^i}(y)$ $\theta$-stable}\}=Y$ or
equivalently for all $y\in Y$ the polytopes $\Theta(V^{\theta^i}(y))$ cover 
$\Delta^1\times\Delta^{n-1}$.
The fibres $(V^{\theta^i}(y))_i$ over a geometric point $y\in Y$ define a 
decomposition of $\Delta^1\times\Delta^{n-1}$ into polytopes:
the polytopes $\Theta(V^{\theta^i}(y))$ and possibly remaining parts. 
By remark \ref{rem:Pn-weightspace} and lemma \ref{lemma:repr-pol-Pn} 
each of the remaining parts is of the form 
\[\textstyle
P(J,J')=\{\theta\in\Delta^1\times\Delta^{n-1}\:|\:\sum_{i\in J}\theta_i\leq-\eta_1\}
\cap\{\theta\in\Delta^1\times\Delta^{n-1}\:|\:\sum_{i\in J'}\theta_i\leq-\eta_2\},
\]
i.e.\ bounded by walls $W_J$ and $W_{(J')^\complement}$,
for some $J,J'$ such that $J\subsetneq(J')^\complement$. 
But any such $P(J,J')$ contains $\theta^i$ for $i\in\{1,\ldots,n\}\setminus(J\cup J')$.
So already the polytopes $\Theta(V^{\theta^i}(y))$ cover $\Delta^1\times\Delta^{n-1}$. 
Thus we have shown $\Lim(P_n)\cong\Lim'(P_n)$.

By remark \ref{rem:restr-eq-theta-theta'} we can restrict the equations 
$s^{\theta^i}_{k,0}s^{\theta^i}_{l,1}s^{\theta^j}_{k,1}s^{\theta^j}_{l,0}
=s^{\theta^i}_{k,1}s^{\theta^i}_{l,0}s^{\theta^j}_{k,0}s^{\theta^j}_{l,1}$
in the functor $\Lim'(P_n)$ relating $V^{\theta^i},V^{\theta^j}$ to equations 
with $l=i$. 
Thus $\Lim'(P_n)$ is isomorphic to the functor $\Lim''(P_n)$ defined by 
\[\textstyle
Y\;\mapsto\;
\left\{\,
(\phi_{V^{\theta^i}})_i\in\prod_i\mathcal M^{\theta^i}(P_n)(Y)
\,\left|
\begin{array}{l}
\forall\,i,j,k\colon
s^{\theta^i}_{k,0}s^{\theta^i}_{i,1}s^{\theta^j}_{k,1}s^{\theta^j}_{i,0}
=s^{\theta^i}_{k,1}s^{\theta^i}_{i,0}s^{\theta^j}_{k,0}s^{\theta^j}_{i,1}
\end{array}\!
\right.\right\}
\]

Each $\phi_{V^{\theta^i}}$ is given by a family $(s^{\theta^i}_j)_{j=1,\ldots,n}$ of 
sections $s^{\theta^i}_j\colon Y\to\P^1_Y$ up to operation of $\G_m$ on $\P^1_Y$ 
fixing $(0:1),(1:0)$.
Within its isomorphism class we can choose this family such that $s_i^{\theta^i}=(1:1)$, 
then we denote $s^i_j=s^{\theta^i}_j$. The functor $\Lim''(P_n)$ is isomorphic to 
\[\textstyle
Y\;\mapsto\;
\left\{
\begin{array}{l}
\textit{families $(s^i_j)_{i,j\in\{1,\ldots,n\}}$}\\
\textit{of sections $s^i_j\colon Y\to\P^1_Y$}\\
\end{array}
\left|
\begin{array}{l}
\forall\,i\colon s^i_{i,0}=s^i_{i,1}\\
\forall\,i,j,k\colon s^i_{k,0}s^j_{k,1}s^j_{i,0}=s^i_{k,1}s^j_{k,0}s^j_{i,1}
\end{array}
\right.\right\}
\]
This functor is the same as the functor (\ref{eq:Ln-functor}) in 
theorem \ref{thm:Ln-ProdP1-functor}.
\end{proof}

In the same way we study the relation of moduli spaces of representations 
of $Q_n$ to the Grothendieck-Knudsen moduli spaces $\overline{M}_{0,n}$.

\medskip

For $T\subseteq\{1,\ldots,n\}$, $|T|=3$ we define the weight $\theta^T$ 
for the quiver $Q_n$ as $\theta^T_i=\frac{2}{3}(1-\varepsilon)$ for 
$i\in T$ and $\theta^T_i=\frac{2}{n-3}\varepsilon$ for $i\not\in T$, 
where $1\gg\varepsilon>0$.
The weights $\theta^T$ are generic and for any $\theta^T$-stable free representation 
$V^{\theta^T}\!$, considered as a tuple $(s^{\theta^T}_j)_j$ of sections of $\P^1$,
the sections $(s^{\theta^T}_i)_{i\in T}$ are disjoint.
The moduli spaces $\mathcal M^{\theta^T}(Q_n)$ are isomorphic to $(\P^1)^{n-3}$.

\begin{thm}\label{thm:M0n-limQn}
There is an isomorphism
\[\overline{M}_{0,n}\;\cong\;\;\varprojlim_\theta\,\mathcal M^\theta(Q_n).\]
\end{thm}
\begin{proof}
By proposition \ref{prop:functorinvlimQnPn} 
$\varprojlim_\theta\mathcal M^\theta(Q_n)\cong\Lim(Q_n)$.
Consider the functor $\Lim'(Q_n)$ defined by
\[\textstyle
Y\;\mapsto\;
\left\{\left.
(\phi_{V^{\theta^T}})_T\in\prod_T\mathcal M^{\theta^T}(Q_n)(Y)
\:\right|
\textit{equations {\rm(\ref{eq:theta-theta'})} hold}\;
\right\}
\]
where $\phi_{V^{\theta^T}}\colon Y\to\mathcal M^\theta(Q)$ is the morphism determined by a
$\theta^T$-stable representation $V^{\theta^T}$ over $Y$, and the representations
$V^{\theta^T}$ are considered as tuples of sections $(s^{\theta^T}_k)_k$ of $\P^1_Y\to Y$.

In the same way as in the proof in the case of $P_n$ we show that the 
natural morphism $\Lim(Q_n)\to\Lim'(Q_n)$ is an isomorphism by showing that 
$\bigcup_T\{y\:|\:\textit{$V^{\theta^T}\!(y)$ $\theta$-stable}\}=Y$ for all
$(V^{\theta^T})_T$ defining an element of $\Lim'(Q_n)(Y)$ and all generic $\theta$.
For such $(V^{\theta^T})_T$ and fixed $y\in Y$ again the polytopes 
$\Theta(V^{\theta^T}\!(y))$ either coincide or their interiors are disjoint, 
and $\Delta(2,n)$ decomposes into the polytopes $\Theta(V^{\theta^T}\!(y))$ 
and remaining parts.
Using remark \ref{rem:Qn-weightspace} and lemma \ref{lemma:repr-pol-Qn} we see that
the remaining parts consist of polytopes bounded by walls $W_{\{J_l,J_l^\complement\}}$
which do not intersect in the interior of $\Delta(2,n)$ and thus are of the form
$\bigcap_l\{\theta\in\Delta(2,n)\:|\:\sum_{j\in J_l}\theta_j\leq 1\}$ for 
disjoint subsets $J_l\subset\{1,\ldots,n\}$, $2\leq|J_l|\leq n-2$. 
Any such subset which is full-dimensional contains a $\theta^T$, 
thus $\bigcup_T\Theta(V^{\theta^T}\!(y))=\Delta(2,n)$ for fixed $y\in Y$, 
which is equivalent to the above statement.

In the functor $\Lim'(Q_n)$ the statement that equations \ref{eq:theta-theta'} 
hold means that for all $T,T'$ and for all $i,j$ locally over 
$\big\{y\:|\:s^{\theta^T}_i\!(y)\neq s^{\theta^T}_j\!(y),
s^{\theta^{T'}}_i\!\!(y)\neq s^{\theta^{T'}}_j\!\!(y)\big\}\subseteq Y$, 
after a choice of coordinates such that $s^{\theta^T}_i\!\!,s^{\theta^{T'}}_i\!\!=(0:1)$
and $s^{\theta^T}_j\!\!,s^{\theta^{T'}}_j\!\!=(1:0)$, for all $k,l$ the equations
\begin{equation}\label{eq:thetaT-thetaT'}
s^{\theta^T}_{k,0}s^{\theta^T}_{l,1}s^{\theta^{T'}}_{k,1}s^{\theta^{T'}}_{l,0}
=s^{\theta^T}_{k,1}s^{\theta^T}_{l,0}s^{\theta^{T'}}_{k,0}s^{\theta^{T'}}_{l,1}
\end{equation}
hold. 

By the last part of the proof of lemma \ref{le:eq-invlimit}, if the union of the walls 
separating $\Theta(V^{\theta^T}\!(y)),\Theta(V^{\theta^{T'}}\!\!(y))$ and 
$\Theta(V^{\theta^{T'}}\!\!(y)),\Theta(V^{\theta^{T''}}\!\!(y))$ are exactly 
the walls that separate\linebreak $\Theta(V^{\theta^T}\!(y)),\Theta(V^{\theta^{T''}}\!\!(y))$ 
and if $V^{\theta^T}\!\!,V^{\theta^{{T'}}}\!\!$ and 
$V^{\theta^{{T'}}}\!\!\!,V^{\theta^{T''}}\!\!$ in a neighbourhood of $y$ 
are related by equations of the form (\ref{eq:thetaT-thetaT'}), 
then so are $V^{\theta^T}\!\!,V^{\theta^{T''}}\!\!\!$. 
It follows that we can restrict to equations (\ref{eq:thetaT-thetaT'}) for which 
$\Theta(V^{\theta^T}\!(y)),\Theta(V^{\theta^{T'}}\!\!(y))$ either coincide or 
are separated by a wall 
$W_{\{J,J^\complement\}}=\Theta(V^{\theta^T}\!(y))\cap \Theta(V^{\theta^{T'}}\!\!(y))$.

We claim that we can restrict this set of equations further to the set of equations 
for $T,T'$ with $|T\cap T'|=2$. 
Given $T,T'$ with $|T\cap T'|<2$ we show that we can replace $T$ by $\tilde{T}$ 
(and similarly $T'$ by $\tilde{T}'$), possibly multiple times, such that 
in each step $\theta^{\tilde{T}}\!\in \Theta(V^{\theta^T}\!(y))$, 
$|T\cap\tilde{T}|=2$ and finally $|\tilde{T}\cap T'|=2$.
As the case that $\Theta(V^{\theta^T}\!(y))$ and $\Theta(V^{\theta^{T'}}\!\!(y))$ 
coincide is trivial, we can assume 
$\Theta(V^{\theta^T}\!(y))\cap \Theta(V^{\theta^{T'}}\!\!(y))=W_{\{J,J^\complement\}}$ 
for some $J$, so that $(s^{\theta^T}_i\!(y))_{i\in J^\complement}$ coincide and
$(s^{\theta^{T'}}_j\!\!(y))_{j\in J}$ coincide.
If $T\subset J$ then there is $\tilde T$ such that 
$\theta^{\tilde T}\!\!\in \Theta(V^{\theta^T}\!(y))$, 
$|T\cap\tilde T|=2$, $\tilde T\cap J^\complement\neq\emptyset$, thus we can assume
that both $T$ and $T'$ have nonempty intersection with $J$ and $J^\complement$.
If $T\cap T'=\emptyset$ there exists $\tilde{T}$ such that 
$\theta^{\tilde{T}}\!\in \Theta(V^{\theta^T}\!(y))$, $|T\cap\tilde{T}|=2$, 
$|\tilde{T}\cap T'|=1$: to define $\tilde{T}=\{i_1,i_2,i_3\}$ let 
$i_1\in T'\cap J^\complement$ and then $i_2,i_3\in T$ such that 
$s^{\theta^T}_{i_1}\!(y)\neq s^{\theta^T}_{i_2}\!(y),s^{\theta^T}_{i_3}\!(y)$. 
If $|T\cap T'|=1$ there is $\tilde{T}$ such that $\theta^{\tilde{T}}\!\in
\Theta(V^{\theta^T}\!(y))$, $|T\cap\tilde{T}|=2=|\tilde{T}\cap T'|$: 
to define $\tilde{T}=\{i_1,i_2,i_3\}$ let $\{i_1\}=T\cap T'$, 
then $i_2\in T'\cap J$ if $i_1\in J^\complement$ and 
$i_2\in T'\cap J^\complement$ if $i_1\in J$, and then $i_3\in T$ such that 
$s^{\theta^T}_{i_3}\!(y)\neq s^{\theta^T}_{i_1}\!(y),s^{\theta^T}_{i_2}\!(y)$.

Having shown that in the functor $\Lim'(Q_n)$ the equations (\ref{eq:thetaT-thetaT'})
with $|T\cap T'|=2$ are sufficient, consider $T=\{i_1,i_2,i_3\}$, $T'=\{i_1,i_2,i_4\}$.
By remark \ref{rem:restr-eq-theta-theta'} we can restrict to the equations with 
$i=i_1$, $j=i_2$, further the equations with $k=i_4$ suffice, and we write $i_5$ for $l$. 
There are two cases $i_3=i_5$ and $i_3\neq i_5$ giving rise to two sets of equations. 
This shows that $\Lim'(Q_n)$ is isomorphic to the functor $\Lim''(Q_n)$ defined by 
\[\textstyle
Y\mapsto\left\{\!
(\phi_{V^{\theta^T}})_T
\!\left|\!
\begin{array}{l}
\forall\,i_1,i_2,i_3,i_4\;\textit{pairwise distinct}\colon\\
s^{\{i_1,i_2,i_3\}}_{i_4,0}\!s^{\{i_1,i_2,i_3\}}_{i_3,1}\!
s^{\{i_1,i_2,i_4\}}_{i_4,1}\!s^{\{i_1,i_2,i_4\}}_{i_3,0}
\!\!=\!s^{\{i_1,i_2,i_3\}}_{i_4,1}\!s^{\{i_1,i_2,i_3\}}_{i_3,0}\!
s^{\{i_1,i_2,i_4\}}_{i_4,0}\!s^{\{i_1,i_2,i_4\}}_{i_3,1}\\
\forall\,i_1,i_2,i_3,i_4,i_5\;\textit{pairwise distinct}\colon\\
s^{\{i_1,i_2,i_3\}}_{i_4,0}\!s^{\{i_1,i_2,i_3\}}_{i_5,1}\!
s^{\{i_1,i_2,i_4\}}_{i_4,1}\!s^{\{i_1,i_2,i_4\}}_{i_5,0}
\!\!=\!s^{\{i_1,i_2,i_3\}}_{i_4,1}\!s^{\{i_1,i_2,i_3\}}_{i_5,0}\!
s^{\{i_1,i_2,i_4\}}_{i_4,0}\!s^{\{i_1,i_2,i_4\}}_{i_5,1}
\end{array}
\!\!\!\!
\right.\right\}
\]
where we choose coordinates such that 
$s^{\theta^{\{i_1,i_2,i_3\}}}_{i_1}\!,s^{\theta^{\{i_1,i_2,i_4\}}}_{i_1}\!=(0:1)$, 
$s^{\theta^{\{i_1,i_2,i_3\}}}_{i_2}\!,s^{\theta^{\{i_1,i_2,i_4\}}}_{i_2}\!=(1:0)$ 
and use the abbreviations of the form
$s^{\{i_1,i_2,i_3\}}\!=s^{\theta^{\{i_1,i_2,i_3\}}}$.

Each $V^{\theta^T}$ defining an element $\mathcal M^{\theta^T}(Q_n)(Y)$ consists 
of a family $(s^{\theta^T}_k)_{k=1,\ldots,n}$ of sections $s^{\theta^T}_k\colon Y\to\P^1_Y$
up to operation of $\PGL(2)$ on $\P^1_Y$.
Within its isomorphism class, if $T=\{i_1,i_2,i_3\}$, we can choose this family such that 
$s^{\theta^T}_{i_1}\!=(0:1)$, $s^{\theta^T}_{i_2}\!=(1:0)$, $s^{\theta^T}_{i_3}\!=(1:1)$, 
then we denote $s^{i_1,i_2,i_3}_{i_4}=s^{\theta^T}_{i_4}\colon Y\to\P^1_Y$. 
The $3!$ sections corresponding to $s^{\theta^T}_{i_4}$ for $T=\{i_1,i_2,i_3\}$ 
are related by equations (\ref{eq:M04-s3}).
The first set of equations of $\Lim''(Q_n)$ gives equations (\ref{eq:M04-s4}).
The second set of equations of $\Lim''(Q_n)$ gives the equations of the form
\[
s^{i_1,i_2,i_3}_{i_4,0}s^{i_1,i_2,i_3}_{i_5,1}s^{i_1,i_2,i_4}_{i_5,0}
=s^{i_1,i_2,i_3}_{i_4,1}s^{i_1,i_2,i_3}_{i_5,0}s^{i_1,i_2,i_4}_{i_5,1}
\]
The functor that associates to an $S$-scheme $Y$ families of sections 
$s^{i_1,i_2,i_3}_{i_4}\colon Y\to\P^1_Y$ for $i_1,i_2,i_3,i_4\in\{1,\ldots,n\}$,
$|\{i_1,i_2,i_3\}|=3$ that satisfy these equations is functor (\ref{eq:M0n-functor}).
\end{proof}

\section{Hassett moduli spaces and moduli spaces of quiver representations}
\label{sec:M0a}

\subsection{Hassett moduli spaces $\overline{M}_{0,a}$ of weighted pointed curves of genus zero}

The Hassett moduli spaces $\overline{M}_{g,a}$ of $a$-stable $n$-pointed curves
of genus $g$ for a weight $a$ were introduced in \cite{Ha03}. Here we consider 
the Hassett moduli spaces $\overline{M}_{0,a}$ of curves of genus $0$. The weight
$a=(a_1,\ldots,a_n)\in\Q^n$ satisfies $0<a_i\leq 1$ for all $i$ and $|a|=\sum_ia_i>2$.

\begin{defi}
An {\it $a$-stable $n$-pointed curve of genus $0$} over an algebraically closed 
field is a tuple $(C,s_1,\ldots,s_n)$ where $C$ is a complete 
connected reduced curve $C$ of genus $0$ with at most ordinary double 
points (i.e.\ a tree of $\P^1$'s), $s_1,\ldots,s_n$ are closed points of $C$ 
such that each $s_i$ is a regular point of $C$, if $(s_i)_{i\in I}$ coincide
then $\sum_{i\in I}a_i\leq 1$, and the $\Q$-divisor $K_C+a_1s_1+\ldots+a_ns_n$ 
is ample, i.e.\ for each irreducible component $C^k$ of $C$ we have 
$(\textit{number of intersection points with other components})+
\sum_{s_i\in C^k}a_i>2$.
\end{defi}

\begin{defithm} {\rm(\cite{Ha03})}.
The Hassett moduli space $\overline{M}_{0,a}$ is 
the fine moduli space of $a$-stable $n$-pointed curves of genus $0$, 
i.e.\ $\overline{M}_{0,a}$ represents the moduli functor (denoted by the same symbol)
\[ 
Y\;\mapsto\;\Big\{\textit{isomorphism classes of $a$-stable $n$-pointed curves 
of genus $0$ over $Y\,$}\Big\}
\]
where an $a$-stable $n$-pointed curve of genus $0$ over a scheme $Y$ is a 
flat proper morphism $C\!\to\!Y$ with sections 
$s_1,\ldots,s_n\colon Y\!\to C$ such that the geometric fibres 
$(C_y,s_1(y),\ldots,s_n(y))$ are $a$-stable $n$-pointed curves of genus $0$. 
\end{defithm}

For $a=(1,\ldots,1)$ the Hassett moduli space $\overline{M}_{0,a}$
coincides with the Grothendieck-Knudsen moduli space $\overline{M}_{0,n}$,
for $a=(1,1,\varepsilon,\ldots,\varepsilon)$, $0<\varepsilon\ll 1$ the moduli space
$\overline{M}_{0,a}$ coincides with the Losev-Manin moduli space $\overline{L}_{n}$
(cf.\ \cite[6.4 and 7.1]{Ha03}).

\subsection{The functor of points of $\overline{M}_{0,a}$}

We give a description of the functor of $\overline{M}_{0,a}$ similar to 
theorem \ref{thm:Ln-ProdP1-functor} and \ref{thm:M0n-ProdP1-functor} 
based on natural embeddings of $a$-stable $n$-pointed curves of genus $0$ 
in a product of $\P^1$'s.

\medskip

Let $(C\to Y,s_1,\ldots,s_n)$ be an $a$-stable $n$-pointed curve of genus $0$ over $Y$. 
Using again the methods of \cite{Kn83}, for $T\subseteq\{1,\ldots,n\}$, $|T|=3$ 
the sections $(s_i)_{i\in T}$ define a contraction morphism to a $\P^1$-bundle over $Y$.
Restricting to the open subscheme $Y^T\subseteq Y$ where $(s_i)_{i\in T}$ are 
pairwise disjoint, we obtain a trivial $\P^1$-bundle $\P^1_{Y^T}\to Y^T$ 
with sections $s^T_1,\ldots,s^T_n$ such that $(s^T_i)_{i\in T}$ are pairwise disjoint.
As in example \ref{ex:M04} we can introduce homogeneous coordinates 
$x^{i_1,i_2,i_3}_0,x^{i_1,i_2,i_3}_1$ on $\P^1_{Y^T}$ such that 
$s^T_{i_1}=(0:1)$, $s^T_{i_2}=(1:0)$, $s^T_{i_3}=(1:1)$. We write the sections 
$s^T_j$ respect to these coordinates as 
$s^{i_1,i_2,i_3}_j=(s^{i_1,i_2,i_3}_{j,0}:s^{i_1,i_2,i_3}_{j,1})$.

Under permutations of the elements of $T$ these sections are related by the 
formulae (\ref{eq:M04-s3}).
For  $I=\{i_1,i_2,i_3,i_4\}\subseteq\{1,\ldots,n\}$, $|I|=4$ and
$T,T'\subset I$, $T=\{i_1,i_2,i_3\}$, $T'=\{i_1,i_2,i_4\}$ over 
$Y^T\cap Y^{T'}$ we have the relation (\ref{eq:M04-s4}).
Let $J=\{i_1,i_2,i_3,i_4,i_5\}\subseteq\{1,\ldots,n\}$, $|J|=5$ 
and $I=\{i_1,i_2,i_3,i_4\}$, $I'=\{i_1,i_2,i_3,i_5\}$, $I''=\{i_1,i_2,i_4,i_5\}$, 
$T=\{i_1,i_2,i_3\}$, $T'=\{i_1,i_2,i_4\}$. Then over $Y^T\cap Y^{T'}$ we have the relation
\begin{equation}\label{eq:M0a-3}
s^{i_1,i_2,i_3}_{i_4,0}s^{i_1,i_2,i_3}_{i_5,1}s^{i_1,i_2,i_4}_{i_5,0}
=s^{i_1,i_2,i_3}_{i_4,1}s^{i_1,i_2,i_3}_{i_5,0}s^{i_1,i_2,i_4}_{i_5,1}
\end{equation}
because over the open set $\{y\:|\:s_{i_3}(y)\neq s_{i_4}(y)\}\subseteq Y^T\cap Y^{T'}$ 
we have the same situation as in the case $\overline{M}_{0,n}$ (see theorem 
\ref{thm:M0n-ProdP1-functor}), over 
$\{y\:|\:\textit{$s_{i_3}(y),s_{i_4}(y)$ in same component}\}\subseteq Y^T\cap Y^{T'}$
these equations correspond to usual base changes in $\P^1$.

\begin{thm}\label{thm:functorM0a}
The functor of $\overline{M}_{0,a}$ is isomorphic to the contravariant functor
\begin{equation}\label{eq:M0a-functor} 
Y\mapsto
\left\{\!\!
\begin{array}{l}
\textit{families of sections}\\
s^{i_1,i_2,i_3}_{i_4}\colon Y^T\to\P^1_{Y^T}\\
\textit{for $i_1,i_2,i_3,i_4\in\{1,\ldots,n\}$,}\\
T=\{i_1,i_2,i_3\}, |T|=3,\\
\textit{$Y^T\subseteq Y$ open}
\end{array}\!\!
\!\left|\!
\begin{array}{l}
(0)\;\forall\,T=\{i_1,i_2,i_3\}, |T|=3\colon\\
s^{i_1,i_2,i_3}_{i_1}\!\!=\!(0\!:\!1),s^{i_1,i_2,i_3}_{i_2}\!\!=\!(1\!:\!0), 
s^{i_1,i_2,i_3}_{i_3}\!\!=\!(1\!:\!1)\\
(1)\;\forall\,I=\{i_1,i_2,i_3,i_4\}, |I|=4\colon\\
\textit{equations {\rm(\ref{eq:M04-s3})} and {\rm(\ref{eq:M04-s4})} hold}\\
(2)\;\forall\,J=\{i_1,i_2,i_3,i_4,i_5\}, |J|=5\colon\\
\textit{equations {\rm(\ref{eq:M0a-3})} hold}\\
(3)\;\forall\,T\:\forall\,y\in Y^T\colon w(y,T,a)>2\\
(4)\;\forall\,y\in Y\;\forall\,a'\in\Q_{>0}^n,|a'|>2, a'\leq a\;\exists\,T\colon\\ 
y\in Y^T\!\!,\,w(y,T,a')>2\\
(5)\;\forall\,T,T'\;\forall\,y\in Y^{T'}\colon\\
\textit{$(s^{T'}_j(y))_{j\in T}$ pairwise distinct}\Rightarrow y\in Y^T
\end{array}\!\!\!\!
\right.\right\}
\end{equation}
where sets $T$ (same for $T'$) are always subsets $T\subseteq\{1,\ldots,n\}$ with $|T|=3$,
\[\textstyle
w(y,T,a)\;=\;\sum_l\,\min\big\{1,\sum_{i\in J_l(y,T)}a_i\big\}
\]
for geometric points $y\in Y^T$ and the partitions $\{1,\ldots,n\}=\bigsqcup_lJ_l(y,T)$ 
defined by $i,j\in J_l(y,T)$ for some $l$ if and only if $s_i^T(y)=s_j^T(y)$.
We write $s^T_i$ if we work with some fixed ordering of $T$ that is not 
relevant.
\end{thm}
\begin{proof}
We show that the above construction gives a morphism from  $\overline{M}_{0,a}$ to 
the functor (\ref{eq:M0a-functor}). 
Of the remaining conditions (3) is satisfied because of $a$-stability.
Concerning condition (4), for an $a$-stable $n$-pointed tree $C$ over 
an algebraically closed field each intersection point $p$ of components 
divides the tree into two components, where for a given 
$a'\in\Q_{>0}^n$, $|a'|>2$, $a'\leq a$ the weights $a'_i$ of 
the marked points of at least one component sum up to $>1$. 
Let $C^p$ be this component if it is unique, otherwise $C^p=C$. 
The intersection $\bigcap_p C^p$ is nonempty and $w(y,T,a')>2$ 
for each $T$ which singles out an irreducible component of $\bigcap_p C^p$.
Condition (5) is satisfied because our construction defines $Y^T$ as 
the maximal set where the sections $(s_i)_{i\in T}$ are disjoint.

We define the morphism in the opposite direction: for an $S$-scheme $Y$ and 
a collection of sections $s^{i_1,i_2,i_3}_{i_4}\colon Y^T\to\P^1_{Y^T}$ 
as in (\ref{eq:M0a-functor}) we construct an $a$-stable $n$-pointed tree over $Y$.
For a geometric point $y\in Y$ let $\mathcal T^y=\{T\:|\:y\in Y^T\}$ and 
$U_y=\{y'\in\bigcap_{T\in\mathcal T^y}Y^T\:|\:\forall\,T\in\mathcal T^y\colon 
s^T_i(y')\!=\!s^T_j(y')\Rightarrow s^T_i(y)\!=\!s^T_j(y)\}$.
Define $C^y\subseteq\prod_{T\in\mathcal T^y}(\P^1_{Y^T})_{U_y}$ as the subscheme 
defined by the equations
 \[
s^{i_1,i_2,i_3}_{i_4,0}x^{i_1,i_2,i_3}_1x^{i_1,i_2,i_4}_0
=s^{i_1,i_2,i_3}_{i_4,1}x^{i_1,i_2,i_3}_0x^{i_1,i_2,i_4}_1
\] 
where $x^{i_1,i_2,i_3}_0,x^{i_1,i_2,i_3}_1$ are homogeneous coordinates 
of $\P^1_{Y^T}$, $T=\{i_1,i_2,i_3\}$, as in remark \ref{rem:M0n-functor}.
We define sections 
$s^y_i=\prod_{T\in\mathcal T^y}s^T_i|_{U_y}\colon U_y\to C^y\subseteq
\prod_{T\in\mathcal T^y}(\P^1_{Y^T})_{U_y}$ for $i\in\{1,\ldots,n\}$.

Let $I\subseteq\{1,\ldots,n\}$ such $s^y_i(y)\neq s^y_{i'}(y)$ for $i,i'\in I$ 
and for any $j\in\{1,\ldots,n\}$ there is $i\in I$ such that $s^y_j(y)=s^y_i(y)$.
Then the same construction but restricted to $T\in\mathcal T^y$ with $T\subseteq I$ 
yields a curve $\bar{C}^y\subseteq\big(\prod_{T\in\mathcal T^y,T\subseteq I}
\P^1_{Y^T}\big)_{U_y}$ with sections $(\bar{s}^y_i)_i$ over $U_y$ isomorphic to 
$(C^y,(s^y_i)_i)$. By remark \ref{rem:M0n-functor} $(\bar{C}^y,(\bar{s}^y_i)_{i\in I})$ 
is a stable $|I|$-pointed tree over $U_y$.

Using conditions (3) and (4) we show that the fibres $(C^y_y,(s^y_i(y))_i)$ 
are $a$-stable. The property that for each irreducible component 
$C^{y,k}_y$ of $C^y_y$ we have
$(\textit{number of intersection}$\linebreak $\textit{points with other components})
+\sum_{s^y_i(y)\in C^{y,k}_y}a_i>2$ follows from condition (3).
The property that if $(s^y_j(y))_{j\in J}$ coincide then $a(J)=\sum_{j\in J}a_j\leq 1$
follows from condition (4): if $(s^y_j(y))_{j\in J}$ coincide, then for  
$0<\varepsilon<\min\{|a|-a(J),|a|-2\}$ there is $a'<a$ such that $a'_j=a_j$ for 
$j\in J$ and $\sum_{i\in J^\complement}a'_i=\max\{0,2-a(J)\}+\varepsilon$, 
by (4) there exists $T\in\mathcal T^y$ such that $w(y,T,a')>2$,
and $a(J)\leq 1$ follows, because choosing $\varepsilon<\min\{1,a(J)-1\}$ in case 
$a(J)>1$ would imply $w(y,T,a')<2$.

Thus we have shown that $(C^y,(s^y_i)_i)$ is an $a$-stable $n$-pointed tree over $U_y$.

We show that the data for $T\in\mathcal T^y$ are sufficient for constructing
the pointed curve over $U_y$: for $y'\in U_y$ it is 
$\mathcal T^y\subseteq\mathcal T^{y'}$ and we show that the
projection $\prod_{T'\in\mathcal T^{y'}}(\P^1_{Y^{T'}})_{U_y\cap U_{y'}}\to
\prod_{T\in\mathcal T^y}(\P^1_{Y^T})_{U_y\cap U_{y'}}$ induces an isomorphism
$(C^{y'},(s^{y'}_i)_i)|_{U_y\cap U_{y'}}\to(C^y,(s^y_i)_i)|_{U_y\cap U_{y'}}$.
A set $T'\in T^{y'}$ is not contained in $\mathcal T^y$ if and only if 
$(s^y_i(y))_{i\in T'}$ are not pairwise distinct. Assume $T'=\{i_1,i_2,i'_3\}$ 
such that $s^y_{i_1}(y)\neq s^y_{i_2}(y)=s^y_{i'_3}(y)$.
Let $T=\{i_1,i_2,i_3\}\in\mathcal T^y$ such that the projection of $C^y$ onto the factor 
$(\P^1_{Y^T})_{U_y}$ corresponding to $T$ induces an isomorphism of the irreducible
component of $C^y_y$ that contains $s^y_{i_2}$ and $(\P^1_{Y^T})_y$.
Then the coordinates for $T$ and $T'$ are related over $U_y\cap U_{y'}$ by
$s^{i_1,i_2,i_3}_{i'_3,0}x^{i_1,i_2,i_3}_1x^{i_1,i_2,i'_3}_0
=s^{i_1,i_2,i_3}_{i'_3,1}x^{i_1,i_2,i_3}_0x^{i_1,i_2,i'_3}_1$
where $s^{i_1,i_2,i_3}_{i'_3}$ does not meet $(0:1),(1:0)$ since 
$s^T_{i'_3}$ does not meet $s^T_{i_1},s^T_{i_2}$ over $U_{y'}$.
In the case that three sections $(s^y_i)_{i\in T''}$, $T''\in\mathcal T^{y'}$, 
coincide over $y$, there is a $T'\in\mathcal T^{y'}$ with $|T'\cap T''|=2$
such that the coordinates for $T''$ and $T'$ are related in a similar way and
only two of the sections $(s^y_i(y))_{i\in T'}$ coincide.

It follows that the $a$-stable $n$-pointed trees $(C^y\to U_y,(s^y_i)_i)$  glue to an
$a$-stable $n$-pointed tree $(C\to Y,(s_i)_i)$ over $Y$. This construction defines 
a morphism from the functor (\ref{eq:M0a-functor}) to $\overline{M}_{0,a}$.

One verifies that these two morphisms are mutually inverse.
\end{proof}

\subsection{Hassett moduli spaces as limits of moduli spaces of quiver representations}

In theorem \ref{thm:Ln-limPn} we have described the Losev-Manin moduli space
as the limit over moduli spaces of quiver representations $\mathcal M^\theta(P_n)$.
By corollary \ref{cor:Qn+2-P_n} this is the same as the limit over 
moduli spaces of representations of the quiver $Q_{n+2}$ for weights 
in a neighbourhood of a vertex of $\Delta(2,n+2)$: we have 
$\overline{L}_n=\varprojlim_{\theta<a}\mathcal M^\theta(Q_{n+2})$
for $a=(1,1,\varepsilon,\ldots,\varepsilon)\in\Q^{n+2}$, $0<\varepsilon\ll1$.
On the other hand for the same weight $a$ we have $\overline{L}_n\cong\overline{M}_{0,a}$.
We show that a similar statement relating Hassett moduli spaces of 
weighted $n$-pointed curves of genus zero to an inverse limit of 
moduli spaces of representations of $Q_n$ holds in general.

\begin{thm}\label{thm:M0a-limQna}
Let $a=(a_1,\ldots,a_n)\in\Q^n$ such that $0<a_i\leq1$, $|a|=\sum_ia_i>2$.
The Hassett moduli space $\overline{M}_{0,a}$ is the limit over
the moduli spaces of representations $\mathcal M^\theta(Q_n)$ for
$\theta\in P(a)=\{\theta\in\Delta(2,n)\:|\:\theta<a\}$, i.e.\
\[\overline{M}_{0,a}\;\cong\;\:\varprojlim_{\theta<a}\mathcal M^\theta(Q_n).\]
\end{thm}
\begin{proof}
We verify that the results in subsection \ref{subsec:functor-limPnQn} remain valid
if we consider the limit over weights in the convex polytope $P(a)$ and obtain the 
following analogue of proposition \ref{prop:functorinvlimQnPn}:
The functor $\varprojlim_{\theta<a}\mathcal M^\theta(Q_n)$ is isomorphic to the 
functor $\Lim(Q_n,a)$ defined by
\[\textstyle
Y\;\mapsto\;
\Big\{
(\phi_{V^\theta})_\theta\in\prod_{\,\textit{$\theta\!<\!a$ generic}}
\mathcal M^\theta(Q)(Y)
\;\Big|\;
\textit{locally equations {\rm(\ref{eq:theta-theta'})} hold for $(V^\theta)_\theta$}\;
\Big\}
\]
where the product is over representatives $\theta$ of the generic GIT equivalence classes
that meet $P(a)$ and $\phi_{V^\theta}\colon Y\to\mathcal M^\theta(Q)$ is the morphism 
determined by a $\theta$-stable representation $V^\theta$ over $Y$.

As in the proof of theorem \ref{thm:M0n-limQn} there is a description in terms of 
the weights $\theta^T$, but one has to take into consideration that for some
$(V^\theta)_\theta$ defining an element of $\Lim(Q_n,a)(Y)$ there might 
be points $y\in Y$ such that the polytopes $\Theta(V^\theta(y))$ do not cover 
$\Delta(2,n)$ and leave out some of the $\theta^T$. So $\Lim(Q_n,a)$ 
is isomorphic to the functor $\Lim'(Q_n,a)$ defined by
\[\textstyle
Y\!\mapsto\!
\left\{\!\!
\begin{array}{l}\!
(\phi_{V^{\theta^T}})_T\in\prod_T\mathcal M^{\theta^T}(Q_n)(Y^T)\\[1mm]
\textit{where $Y^T\subseteq Y$ open}
\end{array}
\!\!\!\left|\!
\begin{array}{l}
\textit{equations {\rm(\ref{eq:theta-theta'})} hold for $(V^{\theta^T})_T$}\\
\forall\,T\;\forall\,y\in Y^T\colon\interior\big(P(a)\cap \Theta(V^{\theta^T}(y))\big)
\neq\emptyset\\
\forall\,y\in Y\colon P(a)\subseteq\bigcup_{y\in Y^T}\Theta(V^{\theta^T}(y))\\
\forall\,T,T'\;\forall\,y\in Y^T\colon \theta^{T'}\!\!\in\Theta(V^{\theta^T}(y))
\Rightarrow y\in Y^{T'}
\end{array}\!\!\!\!
\right.\right\}
\]
with the product over $T\subseteq\{1,\ldots,n\}$, $|T|=3$.

\pagebreak

As in the proof of theorem \ref{thm:M0n-limQn} the equations relating
$V^{\theta^T}\!\!,V^{\theta^{T'}}\!$ for $|T\cap T'|=2$ are sufficient and we
introduce the sections $s^{i_1,i_2,i_3}_{i_4}$. We obtain the functor
$\Lim''(Q_n,a)$ defined by
\[\textstyle
Y\!\mapsto\!
\left\{
\begin{array}{l}\!
\textit{families of sections}\\
s^{i_1,i_2,i_3}_{i_4}\colon Y^T\to\P^1_{Y^T}\\
\textit{for $i_1,i_2,i_3,i_4\in\{1,\ldots,n\}$}\\
\textit{such that $|\{i_1,i_2,i_3\}|=3$,}\\
\textit{$Y^T\subseteq Y$ open}
\end{array}\!
\!\!\left|\!
\begin{array}{l}
(0)\;\forall\,T=\{i_1,i_2,i_3\}, |T|=3\colon\\
s^{i_1,i_2,i_3}_{i_1}\!=(0\!:\!1),s^{i_1,i_2,i_3}_{i_2}\!=(1\!:\!0), 
s^{i_1,i_2,i_3}_{i_3}\!=(1\!:\!1)\\
(1)\;\forall\,I=\{i_1,i_2,i_3,i_4\}, |I|=4\colon\\
\textit{equations {\rm(\ref{eq:M04-s3})} and {\rm(\ref{eq:M04-s4})} hold}\\
(2)\;\forall\,J=\{i_1,i_2,i_3,i_4,i_5\}, |J|=5\colon\\
s^{i_1,i_2,i_4}_{i_5,0}s^{i_1,i_2,i_3}_{i_5,1}s^{i_1,i_2,i_3}_{i_4,0}
=s^{i_1,i_2,i_4}_{i_5,1}s^{i_1,i_2,i_3}_{i_5,0}s^{i_1,i_2,i_3}_{i_4,1}\\
(3)\;\forall\,T\;\forall\,y\in Y^T\colon\interior\big(P(a)\cap \Theta(V^T\!(y))\big)
\neq\emptyset\\
(4)\;\forall\,y\in Y\colon P(a)\subseteq\bigcup_{y\in Y^T}\Theta(V^T\!(y))\\
(5)\;\forall\,T,T'\;\forall\,y\in Y^{T'}\!\colon \theta^T\!\!\in \Theta(V^{T'}\!(y))
\Rightarrow y\in Y^T
\end{array}
\!\!\!\!\right.\right\}
\]
where we write $s^T_j$ for $s^{i_1,i_2,i_3}_j$ in case the order of elements of 
$T=\{i_1,i_2,i_3\}$ is irrelevant and $V^T$ for the representation (up to isomorphism) 
corresponding to $(s^T_j)_j$.

To show that this functor coincides with functor (\ref{eq:M0a-functor})
we compare conditions (3),(4),(5) in both. 
Let $(J_l(y,T))_l$ and $w(y,T,a)$ be defined as 
in theorem \ref{thm:functorM0a}, $\theta,a'$ will always 
be elements of $\Q_{>0}^n$.
Concerning condition (3) for a given $T$ and $y\in Y^T$ we have the equivalences
\begin{center}$
\begin{array}{rcl}
\interior(P(a)\cap \Theta(V^T(y)))\neq\emptyset
&\Longleftrightarrow& 
\exists\,\theta<a,|\theta|=2\;
\forall\,l\colon\sum_{i\in J_l(y,T)}\theta_i<1\\
&\Longleftrightarrow& 
\exists\,a'<a\colon\:w(y,T,a')=|a'|>2\\
&\Longleftrightarrow& 
w(y,T,a)>2.\\
\end{array}$
\end{center}
Concerning condition (4) for fixed $y\in Y$ we have the equivalences
\begin{center}$
\begin{array}{rcl}
P(a)\subseteq\bigcup_{y\in Y^T}\Theta(V^T(y))
&\Longleftrightarrow&
\forall\,\theta<a,|\theta|=2\;\exists\,T\colon y\in Y^T,\;
\textit{$V^T(y)$ $\theta$-semistable}\\
&\Longleftrightarrow&
\forall\,\theta<a,|\theta|=2\;\exists\,T\colon y\in Y^T,\; 
\forall\, l\colon\sum_{i\in J_l(y,T)}\theta_i\leq 1\\
&\Longleftrightarrow&
\forall\,\theta<a,|\theta|=2\;\exists\,T\colon y\in Y^T,w(y,T,\theta)=2\\
&\Longleftrightarrow&
\forall\,a'\leq a,|a'|>2\;\exists\,T\colon y\in Y^T,w(y,T,a')>2.\\
\end{array}$
\end{center}
where for the implication ``$\Longrightarrow$''{} of the last equivalence
use that for $a'\leq a,|a'|>2$ and $\theta\in\interior(P(a'))$ we have
$a'=\theta+b$ with $b\in\Q_{>0}^n$.
Concerning condition (5) we have $\theta^T\!\!\in \Theta(V^{T'}\!(y))$
$\Longleftrightarrow$ $\textit{$V^{T'}(y)$ $\theta^T$-stable}$
$\Longleftrightarrow$ $\textit{$(s^{T'}_j(y))_{j\in T}$ pairwise distinct}$.

Thus $\Lim''(Q_n,a)$ is isomorphic to the functor of $\overline{M}_{0,a}$ 
in (\ref{eq:M0a-functor}).
\end{proof}

\begin{rem}
Our results fit well with the chamber decomposition in \cite[section 5]{Ha03} 
and the observation in \cite[section 8]{Ha03} that the corresponding GIT quotients
can be interpreted as the moduli spaces $\overline{M}_{0,a}$ for $|a|=2$.
In particular, we recover \cite[Theorem 8.2, Theorem 8.3]{Ha03}.
\end{rem}

\bigskip

\noindent
{\small
\begin{tabular}{lll}
Mark Blume&Lutz Hille\\
Mathematisches Institut, Universit\"at M\"unster&
Mathematisches Institut, Universit\"at M\"unster\\
Einsteinstrasse 62, 48149 M\"unster, Germany&
Einsteinstrasse 62, 48149 M\"unster,Germany\\
mark.blume@uni-muenster.de&lutzhille@uni-muenster.de\\
\end{tabular}
}


\bigskip


\begin{thebibliography}{..............}
\addcontentsline{toc}{section}{References}
\setlength{\parskip}{0.2mm}

{\small

\bibitem[Al14]{Al14} {\sc J.\ Alper}, {\it Adequate moduli spaces and 
geometrically reductive group schemes}, Algebraic Geometry 1 (2014), 489--531,
\href{https://arxiv.org/abs/1005.2398}{arXiv:1005.2398}.

\bibitem[An73]{An73}  {\sc S.\ Anantharaman}, {\it Sch\'emas en groupes, espaces 
homog\`enes et espaces alg\'ebriques sur une base de dimension 1},
Sur les groupes alg\'ebriques, pp. 5--79. Bull. Soc. Math. France, Mem. 33, 
Soc. Math. France, Paris, 1973.

\bibitem[BB11]{BB11} {\sc V.\ Batyrev, M.\ Blume}, {\it The functor of toric
varieties associated with Weyl chambers and Losev-Manin moduli spaces},
Tohoku Math.\ J.\ 63 (2011), 581--604, 
\href{https://arxiv.org/abs/0911.3607}{arXiv:0911.3607}.

\bibitem[Br14]{Br14} {\sc S.\ Brochard}, {\it Topologies de Grothendieck, 
descente, quotients}, Autour des sch\'emas en groupes. Vol.\ I, 1–62, 
Panor.\ Synth\`eses, 42/43, Soc.\ Math.\ France, Paris, 2014,
\href{https://arxiv.org/abs/1210.0431}{arXiv:1210.0431}.

\bibitem[Ch08]{Ch08} {\sc C.\ Chindris}, {\it Notes on GIT-fans for quivers}, 
\href{https://arxiv.org/abs/0805.1440}{arXiv:0805.1440}.

\bibitem[Cr96]{Cr96} {\sc W.\ Crawley-Boevey}, {\it Subrepresentations of 
general representations of quivers}, Bull.\ London Math. Soc. 28 (1996), 363--366. 

\bibitem[DH98]{DH98} {\sc I.\ Dolgachev, Y.\ Hu},
{\it Variation of geometric invariant theory quotients},
Publ.\ Math.\ IHES 87 (1998), 5--51,
\href{https://arxiv.org/abs/alg-geom/9402008}{arXiv:alg-geom/9402008}.

\bibitem[EGA]{EGA} {\sc A.\ Grothendieck, J.\ Dieudonn\'e},
{\it \'El\'ements de G\'eom\'etrie Alg\'ebrique},
Publ.\ Math.\ IHES 4,8,11,17,20,24,28,32 (1960--1967).

\bibitem[EGA1]{EGA1} {\sc A.\ Grothendieck, J.\ Dieudonn\'e},
{\it \'El\'ements de G\'eom\'etrie Alg\'ebrique I}, Springer, 1971.

\bibitem[GGMS87]{GGMS87}  {\sc I.\ Gelfand, M.\ Goresky, R.\ MacPherson, V.\ Serganova}, 
{\it Combinatorial Geometries, Convex Polyhedra, and Schubert Cells}, 
Adv.\ Math.\ 63 (1987), 301--316.

\bibitem[GHP88]{GHP88} {\sc L.\ Gerritzen, F.\ Herrlich, M.\ van der Put},
{\it Stable $n$-pointed trees of projective lines},
Nederl.\ Akad.\ Wetensch., Indag.\ Math.\ 50 (1988), 131--163.

\bibitem[GJM13]{GJM13} {\sc N.\ Giansiracusa, D.\ Jensen, H.-B.\ Moon},
{\it GIT compactifications of $M_{0,n}$ and flips},
Adv.\ Math.\ 248 (2013), 242--278, 
\href{https://arxiv.org/abs/1112.0232}{arXiv:1112.0232}.

\bibitem[GKZ]{GKZ} {\sc I.\ Gelfand, M.\ Kapranov, A.\ Zelevinsky}, 
{\it Discriminants, Resultants, and Multidimensional Determinants}, 
Birkh\"auser, 1994.

\bibitem[Ha03]{Ha03} {\sc B.\ Hassett}, {\it Moduli spaces of weighted
pointed stable curves}, Adv.\ Math.\ 173 (2003), 316--352,
\href{https://arxiv.org/abs/math/0205009}{arXiv:math/0205009}.

\bibitem[HP02]{HP02} {\sc L.\ Hille, J.\ de la Pe\~na},
{\it Stable representations of quivers},
J.\ Pure Appl.\ Algebra 172 (2002), 205--224. 

\bibitem[Ka93a]{Ka93a} {\sc M.\ Kapranov}, {\it Chow quotients of
Grassmannians I}, Adv.\ Soviet Math.\ 16 (1993), 29--110,
\href{https://arxiv.org/abs/alg-geom/9210002}{arXiv:alg-geom/9210002}.

\bibitem[Ka93b]{Ka93b} {\sc M.\ Kapranov}, {\it Veronese curves
and Grothendieck-Knudsen moduli space $\overline{M}_{0,n}$},\linebreak
J. Alg. Geom. 2 (1993), 239--262.

\bibitem[Ki94]{Ki94} {\sc A.\ King}, {\it Moduli of representations of
finite-dimensional algebras}, Quart.\ J.\ Math.\ Oxford 45 (1994), 
515--530.

\bibitem[Kn83]{Kn83} {\sc F.\ Knudsen}, {\it The projectivity of the
moduli space of stable curves II: The stacks $M_{g,n}$},
Math.\ Scand.\ 52 (1983), 161--199.

\bibitem[KSZ91]{KSZ91} {\sc M.\ Kapranov, B.\ Sturmfels, A.\ Zelevinsky},
{\it Quotients of toric varieties}, Math.\ Ann.\ 290 (1991), 643--655.

\bibitem[La69]{La69} {\sc D.\ Lazard}, {\it Autour de la platitude},
Bull.\ Soc.\ Math.\ France 97 (1969), 81--128. 
 
\bibitem[LM00]{LM00} {\sc A.\ Losev, Yu.\ Manin}, 
{\it New moduli spaces of pointed curves and pencils of 
flat connections}, Michigan Math.\ J.\ 48 (2000), 443--472,
\href{https://arxiv.org/abs/math/0001003}{arXiv:math/0001003}.

\bibitem[MFK]{MFK} {\sc D.\ Mumford, J.\ Fogarty, F.\ Kirwan}, 
{\it Geometric Invariant Theory}, Springer, 1994.

\bibitem[Re00]{Re00} {\sc N.\ Ressayre}, {\it The GIT-equivalence for G-line 
bundles}, Geom.\ Dedicata 81 (2000), 295--324,  
\href{https://arxiv.org/abs/math/9811053}{arXiv:math/9811053}.

\bibitem[Scho92]{Scho92} {\sc A.\ Schofield}, {\it General representations of quivers},
Proc.\ London Math.\ Soc.\ 65 (1992), 46--64. 

\bibitem[Sek94]{Sek94} {\sc J.\ Sekiguchi}, {\it Cross Ratio Varieties for
Root Systems}, Kyushu J.\ Math.\ 1 (1994), 123--168.

\bibitem[Sek96]{Sek96} {\sc J.\ Sekiguchi}, {\it Cross Ratio Varieties for
Root Systems of Type $A$ and the Terada Model},
J.\ Math.\ Sci.\ Univ.\ Tokyo 3 (1996), 181--197.

\bibitem[Ses77]{Ses77} {\sc C.\ Seshadri}, {\it Geometric reductivity 
over arbitrary base}, Adv.\ Math.\ 26 (1977),\linebreak 225--274. 

\bibitem[SGA7(1)]{SGA7(1)} {\sc A.\ Grothendieck et al.}, 
{\it S\'eminaire de g\'eom\'etrie alg\'ebrique, 
Groupes de Monodromie en G\'eom\'etrie Alg\'ebrique}, Tome 1,
Lecture Notes in Mathematics 288, Springer-Verlag, 
Berlin -- New York, 1972.

\bibitem[SP]{SP} {\sc Stacks Project Authors}, {\it Stacks Project}, 
\url{http://stacks.math.columbia.edu}.

\bibitem[Th96]{Th96} {\sc M.\ Thaddeus}, {\it Geometric invariant 
theory and flips}, J.\ Amer.\ Math.\ Soc.\ 9 (1996), 691--723,
\href{https://arxiv.org/abs/alg-geom/9405004}{arXiv:alg-geom/9405004}.

\bibitem[Th99]{Th99} {\sc M.\ Thaddeus}, {\it Complete collineations revisited.}
Math.\ Ann.\ 315 (1999), 469--495, 
\href{https://arxiv.org/abs/math/9808114}{arXiv:math/9808114}.


}

\end{thebibliography}
\end{document}